\documentclass[12pt]{amsart}
\usepackage[active]{srcltx}
\usepackage{calc,amssymb,amsthm,amsmath,amscd, eucal,ulem,stmaryrd}
\usepackage{alltt}
\usepackage[left=1.35in,top=1.25in,right=1.35in,bottom=1.25in]{geometry}
\usepackage{chngcntr}

\RequirePackage[dvipsnames,usenames]{color}

\input{kmacros3.sty}
\input{xy}
\xyoption{all}
\usepackage{tikz}
\usepackage{mabliautoref}

\DeclareSymbolFont{rsfs}{OMS}{rsfs}{m}{n}
\DeclareSymbolFontAlphabet{\scr}{rsfs}

\renewcommand{\m}{\mathfrak{m}}
\renewcommand{\n}{\mathfrak{n}}
\renewcommand{\fram}{\mathfrak{m}}
\numberwithin{equation}{theorem}

\usepackage{fullpage}

\usepackage{setspace}
\usepackage{hyperref}

\usepackage{enumerate}
\usepackage{graphicx}
\usepackage[all,cmtip]{xy}

\usepackage{upgreek}
\newcommand{\mytau}{{\uptau}}

\usepackage{verbatim}

\renewcommand{\O}{\mathcal O}
\newcommand{\BCMReg}[1]{{$\textnormal{BCM}_{#1}$-regular}}
\newcommand{\BCMRat}[1]{{$\textnormal{BCM}_{#1}$-rational}}
\newcommand{\BCM}{\textnormal{BCM}{}}
\newcommand{\perf}{\textrm{perf}}
\newcommand{\perfd}{\textnormal{perfd}}

\begin{document}

\title{An analog of adjoint ideals and PLT singularities in mixed characteristic}
\author{Linquan Ma, Karl Schwede, Kevin Tucker, Joe Waldron, Jakub Witaszek}
\address{Department of Mathematics, Purdue University, West Lafayette, IN 47907}
\email{ma326@purdue.edu}
\address{Department of Mathematics, University of Utah, Salt Lake City, UT 84112}
\email{schwede@math.utah.edu}
\address{Department of Mathematics, University of Illinois at Chicago, Chicago, IL 60607}
\email{kftucker@uic.edu}
\address{Department of Mathematics, Michigan State University, East Lansing, MI 48824}
\email{waldro51@msu.edu}
\address{Department of Mathematics, University of Michigan, Ann Arbor, MI 48109, USA}
\email{jakubw@umich.edu}

\thanks{This material is partially based upon work supported by the National Science Foundation under Grant No. DMS-1440140 while the authors were in residence at the Mathematical Sciences Research Institute in Berkeley California during the Spring 2019 semester.  The authors also worked on this while attending an AIM SQUARE in June 2019.}
\thanks{Ma was supported by NSF Grant DMS \#1901672, NSF FRG Grant \#1952366, and a fellowship from the Sloan Foundation.}
\thanks{Schwede was supported by NSF CAREER Grant DMS \#1252860/1501102, NSF FRG Grant \#1952522 and NSF Grant \#1801849 and \#2101800.}
\thanks{Tucker was supported by NSF Grant DMS \#1602070 and \#1707661, and by a fellowship from the Sloan Foundation.}
\thanks{Witaszek was supported by the National Science Foundation under Grant No.\ DMS-1638352 at the Institute for Advanced Study in Princeton.  Witaszek was additionally supported by NSF grant DMS-2101897.}
\thanks{Waldron was supported by the Simons Foundation under Grant No. 850684}
\maketitle

\begin{abstract}
	We use the framework of perfectoid big Cohen-Macaulay algebras to define a class of singularities for pairs in mixed characteristic, which we call purely BCM-regular singularities, and a corresponding adjoint ideal.   We prove that these satisfy adjunction and inversion of adjunction with respect to the notion of BCM-regularity and the BCM test ideal defined by the first two authors.  We compare them with the existing equal characteristic PLT and purely $F$-regular singularities and adjoint ideals. As an application, we obtain a uniform version of the Brian\c{c}on-Skoda theorem in mixed characteristic. We also use our theory to prove that two-dimensional KLT singularities are BCM-regular if the residue characteristic $p>5$, 
which implies an inversion of adjunction for three-dimensional PLT pairs of residue characteristic $p>5$. In particular, divisorial centers of PLT pairs in dimension three are normal when $p > 5$. Furthermore, in the appendix we provide a streamlined construction of perfectoid big Cohen-Macaulay algebras and show new functoriality properties for them using the perfectoidization functor of Bhatt and Scholze.
\end{abstract}

\setcounter{tocdepth}{1}
\tableofcontents


\section{Introduction}

In algebraic geometry one can often understand the singularities of a variety $X$ by studying the singularities of its hyperplane sections. For example, if $D \subseteq X$ is a regular Cartier divisor, then $X$ turns out to be regular along $D$ as well. More generally, one can ask if the singularities of $X$ are mild provided that those of $D$ are. This question has a fairly satisfactory answer both in characteristic zero and characteristic $p>0$, and  the goal of our article is to address this problem in mixed characteristic.

To be more formal, consider a variety $X$, an integral Weil divisor $D$ and a $\bQ$-divisor $\Delta\geq 0$ not containing $D$ in its support such that $K_X+D+\Delta$ is $\bQ$-Cartier, which we collect into a pair $(X,D+\Delta)$.  We wish to compare $X$ with the normalization $D^N$ of $D$.   Given this data, there is canonically defined divisor $\diff_{D^N}(\Delta+D)$ on $D^N$ called the different \cite{KollarFlipsAndAbundance,ShokurovThreeDimensionalLogFlips}, which satisfies $$(K_X+D+\Delta)|_{D^N}=K_{D^N}+\diff_{D^N}(D+\Delta).$$


In characteristic zero, the log terminal variant of inversion of adjunction (see \cite[Theorem 4.9 (1)]{KollarKovacsSingularitiesBook}) says that $(X, D+\Delta)$ is purely log terminal (PLT) if and only if $(D^N, \diff_{{D^N}}(D+\Delta))$ is Kawamata log terminal (KLT).  In positive characteristic, a similar result is proved in  \cite{TakagiAdjointIdealsAndACorrespondence} and \cite{Das_different_different_different}, where one replaces KLT with strongly $F$-regular and  PLT with purely $F$-regular.
In this paper, we prove an analogous result in mixed characteristic, in which KLT is replaced by BCM-regular  (defined in \cite{MaSchwedeSingularitiesMixedCharBCM}), and PLT is replaced by a new definition which we call purely BCM-regular.

The BCM-regularity of $R$ is defined in terms of a big Cohen-Macaulay (BCM) algebra $B$, which plays the role of a resolution of singularities in characteristic zero, and $R^+$ (the colimit of all finite domain extensions of $R$) in positive characteristic. In what follows, we will usually fix big Cohen-Macaulay $R$- and $R/I_D$-algebras, $B$ and $C$ respectively, which fit into a diagram as follows:
\[
\xymatrix{
R\ar[r]\ar[d]& R^+\ar[r]\ar[d]& B\ar[d]\\
R/I_D\ar[r] & (R/I_D)^+\ar[r]& C.
}
\]

We call such $B\to C$ compatibly chosen, and such a choice always exists by \cite[Theorem 1.2.1]{AndreWeaklyFunctorialBigCM}.
In fact, Andr\'{e} proved that if we start with a perfectoid big Cohen-Macaulay $R^+$-algebra $B$, then we can always find a perfectoid big Cohen-Macaulay $(R/I_D)^+$-algebra $C$ with which it is compatibly chosen. For this reason, when our results are applied, it will usually be with perfectoid big Cohen-Macaulay algebras, but most of the results themselves hold more generally.  Furthermore, if we are given that $R^+$ is a big Cohen-Macaulay algebra (as has been announced by Bhatt), then we can simply let $B$ and $C$ be the completion of $R^+$ and $(R/I_D)^+$ respectively.  Using $R^+$ as a big Cohen-Macauay algebra would simplify some of the functoriality arguments used in particular in \autoref{sec.birational}.  See \autoref{rem.R+CM} for additional discussion about when this particular big Cohen-Macaulay algebra is sufficient.

As in the equal characteristic settings, we obtain a more precise restriction theorem for a new adjoint-like ideal, which we denote by $\adj_{B\shortrightarrow C}^D(R, \Delta+D)$, and the BCM test ideals $\mytau_C(R,\Delta)$ introduced by the first two authors \cite{MaSchwedeSingularitiesMixedCharBCM} from which the above mentioned inversion of adjunction result follows immediately.

\begin{theoremA*}\textnormal{(\autoref{thm.AdjointMultiplierInversionOfAdjunction}, \autoref{cor:singularities}, and \autoref{cor:normal_centers})}
Let $(R, \fram)$ be a complete normal local domain of residual characteristic $p > 0$.
Let $D$ be a prime Weil divisor on $\Spec(R)$ with ideal $I_D$,  denote the normalization of $R/I_D$ by $R_{D^{\nm}}$,  and let $B$ and $C$ be compatibly chosen 
big Cohen-Macaulay algebras for $R$ and $R/I_D$.  Finally let  $\Delta\geq 0$ be a $\bQ$-divisor such that $D$ is not contained in the support of $\Delta$ and $K_R+D+\Delta$ is $\bQ$-Cartier.
Then:
\[
\adj_{B\shortrightarrow C}^D(R, D+\Delta) \cdot R_{D^{\nm}} = \mytau_C\big( R_{D^{\nm}}, \diff_{R_{D^{\nm}}}(D+\Delta) \big).
\]
In particular, $(R,D+\Delta)$ is purely $\mathrm{BCM}_{B\shortrightarrow C}$-regular if and only if $(R_{D^N},\diff_{R_{D^N}}(D+\Delta))$ is $\mathrm{BCM}_C$-regular.  Furthermore, these conditions imply that $D$ is normal.
\end{theoremA*}

We prove compatibilities of our new adjoint ideal with those already in use in equal characteristic.
First, that for appropriately chosen big Cohen-Macaulay algebras $B$ and $C$, the ideal $\adj_{B\shortrightarrow C}^D(R, \Delta+D)$  is contained in the valuatively defined adjoint ideal  originating in characteristic zero birational algebraic geometry:

\begin{theoremB*} \textnormal{(\autoref{thm:birational_adjoint_comparison})}
	Let $(R,\fram)$ be a complete normal local domain of residual characteristic $p > 0$, $\Delta\geq0$ a $\bQ$-divisor, and $D$ a prime Weil divisor such that $K_R+D+\Delta$ is $\bQ$-Cartier and $D$ is not contained in the support of $\Delta$.
	Then for any proper birational map $\mu \colon Y\to \Spec(R)$ from a normal variety $Y$ there exists a compatibly chosen perfectoid big Cohen-Macaulay $R$- and $R/I_D$-algebras, $B$ and $C$ such that
	\[
	\adj_{B\shortrightarrow C}^D(R,{D+\Delta})\subseteq \mu_*\O_Y(\lceil K_Y-\mu^*(K_R+D+\Delta)+D'\rceil),
	\]
	{where $D'$ is the strict transform of $D$.}
\end{theoremB*}

In fact, there is a single choice of $B\to C$ such that the above containment holds for all $\mu: Y\to \Spec(R)$ by \autoref{thm:uniformBC}. Note that, if log resolutions exist and the variety $Y$ is taken to be a resolution, then the right hand side coincides with the adjoint ideal from birational geometry.

We then show that in characteristic $p > 0$ the big Cohen-Macaulay algebras $B$ and $C$ can be chosen to ensure the ideal $\adj_{B\shortrightarrow C}^D(R, D+\Delta)$ coincides with the  adjoint-like ideal of Takagi \cite{TakagiPLTAdjoint, TakagiHigherDimensionalAdjoint, TakagiAdjointIdealsAndACorrespondence}.

\begin{theoremC*}
	\textnormal{(\autoref{thm.TauEqualityCharP})}
	Suppose that $(R, \fram)$ is a complete local $F$-finite normal domain of positive characteristic $p > 0$,
	$\Delta\geq 0$ a $\bQ$-divisor, and $D$ a prime Weil divisor such that $K_R+D+\Delta$ is $\bQ$-Cartier and $D$ is not contained in the support of $\Delta$.  Then there exists a map $B \to C$ of big Cohen-Macaulay $R^+$- and $(R/I_D)^+$-algebras, such that
	\[
	\mytau_{I_D}(R, D+\Delta) = \adj_{B\shortrightarrow C}^D(R,{D+\Delta}).
	\]
\end{theoremC*}

The most frustrating limitation of the BCM test ideal as defined in \cite{MaSchwedeSingularitiesMixedCharBCM} is that it is not clear if its formation commutes with localization.  As a consequence of our inversion of adjunction result, we obtain the following partial verification of localization.

\begin{theoremD*}\textnormal{(\autoref{thm.taucontainsSingularLocus})}
Suppose $(R,\fram)$ is a complete normal local domain of residual characteristic $p > 0$ and that $\Delta\geq 0$ is a $\mathbb{Q}$-divisor such that $K_R + \Delta$ is $\bQ$-Cartier.  Further suppose that $Q \in \Spec R$ is a point such that the localization $(R_Q, \Delta_Q)$ is simple normal crossing with $\lfloor \Delta_Q \rfloor = 0$ (in particular, $R_Q$ is regular).  Then for any perfectoid big Cohen-Macaulay $R^+$-algebra $B$,
\[
(\mytau_B(R, \Delta))_Q = R_Q.
\]
\end{theoremD*}

As a corollary, we get the following Brian\c{c}on--Skoda type result that is new in mixed characteristic.

\begin{theoremE*}\textnormal{(\autoref{cor:briancon-skoda})}
Let $(R,\fram)$ be a complete normal local domain of residue characteristic $p>0$ and of dimension $d$. Let $J$ be the defining ideal of the singular locus of $R$. Then there exists an integer $N$ such that $J^N\overline{I^h}\subseteq I$ for all $I\subseteq R$ where $h$ is the analytic spread of $I$. In particular, $J^N\overline{I^d}\subseteq I$ for all $I\subseteq R$.
\end{theoremE*}

By passing to an argument involving the inversion of adjunction on the extended Rees algebra, we also obtain the following result.
\begin{theoremF*}\textnormal{(\autoref{lem.GlobFRegExceptionalImpliesBCMReg})}
\label{theoremGlobImpliesBCMReg}
Suppose $(R, \fram, k)$ is a complete normal $\bQ$-Gorenstein local domain of dimension $\geq 2$ with residual characteristic $p > 0$ and $R/p$ $F$-finite.  Let $X = \Spec R$ and suppose that $\pi : Y \to X$ is the blowup of some $\fram$-primary ideal $I$ that $Y$ is normal.  Further suppose that $I \cdot \O_Y = \O_Y(-mE)$ where $E$ is a prime exceptional divisor. {
Let $\Delta$ be a $\Q$-divisor such that $K_X+\Delta$ is $\Q$-Cartier.}

Let $\Delta_E$ denote the different of $K_Y + E + \pi^{-1}_* \Delta$ along $E$.  Suppose that $(E, \Delta_E)$ is globally $F$-regular, then $(R, \Delta)$ is \BCMReg{} and in particular, $R$ itself is \BCMReg{}.
\end{theoremF*}

This implies that many simple singularities are \BCMReg{}.  Furthermore it shows that KLT surface singularities $(R, \Delta)$ are \BCMReg{} as long as the residual characteristic $p > 5$,
see \autoref{thm.KLTSurfaceImpliesBCMReg}.  In fact, the main results of \cite{CasciniGongyoSchwedeUniformBounds} also hold in mixed characteristic, see \autoref{thm.CasciniGongyoSchwedeUniformBounds}.  {
As a corollary, we obtain the inversion of adjunction for PLT pairs of dimension three in mixed characteristic with $p > 5$, see \autoref{cor.Dimension2To3InversionOfAdjunction}.  This then implies that divisorial centers of mixed characteristic three-dimensional PLT pairs with $p > 5$ are normal.

\begin{theoremG*}
\textnormal{(\autoref{cor.Dimension2To3InversionOfAdjunction})}
Suppose that $(R, \fram, k)$ is a normal 3-dimensional local  ring, essentially of finite type over an excellent DVR 	
with $F$-finite residue field of characteristic $p > 5$, and let $X = \Spec R$.
Suppose that $D$ is a prime divisor on $X$ and that $\Delta \geq 0$ is a $\bQ$-divisor such that $K_X + D + \Delta$ is $\bQ$-Cartier {and $\Delta$ has standard coefficients}.  Suppose that $(D^{\mathrm{N}}, \diff_{D^N}(\Delta+D))$ is KLT, then $(\widehat{R}, \widehat{D}+\widehat{\Delta})$ is purely \BCMReg{}.  In particular, $(X, D+\Delta)$ is PLT and {$D$ is normal}.
\end{theoremG*}

Finally, in \cite{MaSchwedeSingularitiesMixedCharBCM} it was proved that in a proper family over a Dedekind domain, if one fiber is strongly $F$-regular then the fibers over an open set are strongly $F$-regular and the generic fiber is KLT.  We remove a technical assumption on the index from that theorem in \autoref{sec:reduction_mod_p}, as well as prove the corresponding result for purely $F$-regular and PLT singularities.

\begin{theoremH*}\textnormal{(\autoref{thm:reduction_mod_p_plt} and \autoref{prop:SFR_on_fiber})}
	Let $\phi\colon X\to U=\Spec(A)$ be a proper flat family where $A$ is a localization of a finite extension of $\bZ$, with fraction field $K$.  Let $D$ be an integral Weil divisor which is horizontal over $U$ (resp. $D=0$), and $\Delta\geq0$ which does not contain $D$ in its support, such that $K_X+D+\Delta$ is $\bQ$-Cartier.  Suppose $(X_p,D_p+\Delta_p)$ is purely $F$-regular (resp. $(X_p,\Delta_p)$ is strongly $F$-regular) for some $p\in A$.  Then $(X_K,\Delta_K+D_K)$ is PLT (resp. $(X_p,\Delta_p)$ is KLT) and there is an open subset $V$ of $U$ such that $(X_q,D_q+\Delta_q)$ is purely $F$-regular (resp. $(X_q,\Delta_q)$ is strongly $F$-regular) for all $q\in V$.
	
\end{theoremH*}

Since pure and strong $F$-regularity can be checked (and is easily computable) via Fedder-type criterion, this theorem provides an efficient and computable algorithm to test whether a variety in characteristic zero is PLT or KLT.

\subsection*{Acknowledgements:} The authors thank Yves Andr\'{e}, Bhargav Bhatt, Kestutis Cesnavicius, Ray Heitmann, Mel Hochster, Craig Huneke, Srikanth Iyengar and Zsolt Patakfalvi for valuable conversations related to this project.  We thank Zsolt Patakfalvi for comments on previous drafts of this paper.  They would also like to thank the referees for their careful reading of the paper and many helpful comments and suggestions.  This material is partially based upon work done while the authors were in residence at the Mathematical Sciences Research Institute in Berkeley California during the Spring 2019 semester.  The authors also worked on this while attending an AIM SQUARE in June 2019. 

\section{Definitions and basic properties}
\label{sec.Definitions}

Throughout this paper, we will use perfectoid algebras as in \cite{BhattMorrowScholzeIHES}, which is recalled in \autoref{sec.perfectoidBCM}. We note that when the ring is $p$-torsion free and contains a compatible system of $p$-power roots of $p$, then this is the same as the definition given in \cite[Section 2]{AndreWeaklyFunctorialBigCM} and \cite[Subsection 2.1]{MaSchwedeSingularitiesMixedCharBCM}. Following the tradition of commutative algebra, when we call a ring {\it local}, we implicitly mean it is a Noetherian ring. A (not necessarily Noetherian) algebra $B$ over a local ring $(R,\m)$ is called {\it big Cohen-Macaulay} if every system of parameters of $R$ is a regular sequence on $B$. {
In particular, $\m B \neq B$.}

We now recall the definition of BCM test ideal from \cite{MaSchwedeSingularitiesMixedCharBCM}. Let $(R,\fram)$ be a complete normal local domain of residue characteristic $p>0$.  Fix an effective canonical divisor $K_R$, and let $\Delta\geq 0$ be a $\bQ$-divisor such that $K_R+\Delta$ is $\bQ$-Cartier.  Then there exists $n\in\bN$ and $f\in R$ such that $n(K_R+\Delta)=\Div_R(f)$.

\begin{definition}\cite[Definition 6.2]{MaSchwedeSingularitiesMixedCharBCM}\label{def:BCM_test_ideal}
	In the above set-up, let $B$ be a big Cohen-Macaulay $R^+$-algebra.  Then we set
	$$0^{B,K_R+\Delta}_{H^d_{\fram}(R)}=\ker\left(H^d_{\fram}(R)\xrightarrow{f^{1/n}}H^d_{\fram}(B)\right).$$

	\noindent The corresponding \emph{BCM test ideal} is defined as:
	\[ \mytau_B(R,\Delta)=\Ann_{\omega_R}0^{B,K_R+\Delta}_{H^d_{\fram}(R)}.
	\]	
	\end{definition}
{Equivalently, we can define $\mytau_B(R,\Delta)$ to be the Matlis dual of $\im(H^d_{\fram}(R)\xrightarrow{f^{1/n}}H^d_{\fram}(B))$. }It is proved in \cite[Proposition 6.10]{MaSchwedeSingularitiesMixedCharBCM} that we can take  a sufficiently large perfectoid $B$ such that $\mytau_B(R,\Delta)\subseteq \mytau_{B'}(R,\Delta)$ for any other perfectoid big Cohen-Macaulay $R^+$-algebra $B'$.  We say that $(R,\Delta)$ is $\mathrm{BCM}_B$-regular if $\mytau_B(R,\Delta)=R$, and that it is $\mathrm{BCM}$-regular\footnote{More accurately, it is called {\it perfectoid} $\mathrm{BCM}$-regular in \cite{MaSchwedeSingularitiesMixedCharBCM} to emphasize that we only consider perfectoid algebras. In this paper we suppress this notion for simplicity.} if it is $\mathrm{BCM}_{B}$-regular for all (and hence for a single sufficiently large) perfectoid big Cohen-Macaulay $B$. It is known that $\mathrm{BCM}$-regularity coincides with strong $F$-regularity if $R$ has positive characteristic $p>0$ (and $K_R+\Delta$ is $\bQ$-Cartier), see \cite[Corollary 6.23]{MaSchwedeSingularitiesMixedCharBCM}

Now suppose that $(R, \fram)$ is a complete normal local domain of mixed characteristic $(0, p)$ or characteristic $p > 0$.  Fix a prime Weil divisor $D$ on $\Spec R$ with $I_D = R(-D)$ the defining ideal.  For every diagram
\begin{equation}
\label{eq.BCCompatibleDiagram}
\xymatrix{
0 \ar[r] & I_D \ar@{.>}[d] \ar[r] & R \ar[d] \ar[r] & R/I_D \ar[d] \ar[r] & 0 \\
0 \ar[r] & (I_D)^+ \ar@{.>}[d] \ar[r] & R^+ \ar[d] \ar[r] & (R/I_D)^+ \ar[d] \ar[r] & 0 \\
& I_{B \shortrightarrow C} \ar[r] & B \ar[r] & C \ar[r]_{+1} &
}
\end{equation}
where $B$ and $C$ are big Cohen-Macaulay $R^+$ (respectively $(R/I_D)^+$) algebras, we can construct an object in $I_{B \shortrightarrow C} \in D^b(R)$ as pictured above.

Here, $(I_D)^+$ is some minimal prime ideal over $I_DR^+$. With this choice $R^+/(I_D)^+$ may be identified with $(R/I_D)^+$ (see \cite[p.\ 27]{HochsterFoundations}).

\begin{remark}
\label{rem.R+CM}
Recently, Bhatt \cite{BhattCMabsoluteintegralclosure} proved that the $p$-adic completion of $R^+$ is a (perfectoid) big Cohen-Macaulay algebra, thus in \autoref{eq.BCCompatibleDiagram}, we may simply let $B=\widehat{R^+}$ and $C=\widehat{(R/I_D)^+}$. In this case, $I_{B\shortrightarrow C}$ is an ideal of $B=\widehat{R^+}$, and we don't need to work in $D^b(R)$. This suffices for many applications in this paper (and it strengthens some of our results, see \autoref{rmk:BhattR^+}). On the other hand, it doesn't seem allow us to build in ``perturbation elements''.  In particular, it is not enough for our proof of a Brian\c{c}on--Skoda type result, see \autoref{cor:briancon-skoda}.
\end{remark}


\begin{definition}\label{def:adjoint}
With notation as above, fix a $\bQ$-divisor $\Delta \geq 0$ such that $K_R + D + \Delta$ is $\bQ$-Cartier and no component of $\Delta$ is equal to $D$.  Select $K_R = -D + G$ with $G \geq 0$ and choose $f \in R$ such that $\Div_R(f) = n(K_R + D + \Delta)$.  By the commutative diagram above, we have a map
\[
\psi_{f^{1/n}} \colon I_D \to I_{B\shortrightarrow C} \xrightarrow{\cdot f^{1/n}} I_{B\shortrightarrow C}.
\]
 We define
 the \emph{BCM adjoint ideal} with respect to $B, C$, denoted $\adj_{B\shortrightarrow C}^D(R, D + \Delta)$,
to be the Matlis dual of \[\Image\Big(H^d_{\fram}(I_D) \xrightarrow{H^d_{\fram}(\psi_{f^{1/n}})} H^d_{\fram}(I_{B\shortrightarrow C})\Big).\]
\end{definition}

\begin{remark}
We have the factorization:
\[
R \to R(K_R + D) \xrightarrow{\cdot f^{1/n}} R^+ \to B.
\]
If one chooses a different representative of $K_R$, say $K_R' = K_R + \Div(h)$, even such that $K_R'$ has components in common with $D$ or is non-effective, and so obtains $f' = h^n f \in K(R)$,
 one still has a map
$R(K_R' + D) \xrightarrow{\cdot f'^{1/n}} B$ which induces $H^d_{\fram}(R(K_R' + D)) \to H^d_{\fram}(B)$, the image of which is still Matlis dual to $\adj_{B\shortrightarrow C}^D(R, D + \Delta)$, as we shall see in the proof of \autoref{lem.AdjBasic}\autoref{lem.adjIsIndep} below.  \end{remark}

We prove two basic properties.
\begin{lemma}
\label{lem.AdjBasic}
With notation as in \autoref{def:adjoint}:
\begin{enumerate}
\item $\adj_{B\shortrightarrow C}^D(R, D + \Delta)$ is an ideal in $R$.
\label{lem.adjIsIdeal}
\item The ideal $\adj_{B\shortrightarrow C}^D(R, D + \Delta)$ is independent of the choice of $K_R$, $f$ and $f^{1/n} \in R^+$.
\label{lem.adjIsIndep}
\end{enumerate}
\end{lemma}
\begin{proof}
First we prove \autoref{lem.adjIsIdeal}.  The map $I_D \xrightarrow{\cdot f^{1/n}} I_{B\shortrightarrow C}$ factors as $I_D \to R(K_R) \xrightarrow{\cdot f^{1/n}} I_{B\shortrightarrow C}$ by construction.  Since $I_D \to R(K_R)$ is generically an isomorphism, $H^d_{\fram}(I_D) \to H^d_{\fram}(R(K_R))$ is surjective and so both $H^d_{\fram}(I_D)$ and $H^d_{\fram}(R(K_R))$ have the same image in $H^d_{\fram}(I_{B\shortrightarrow C})$.  Thus, by Matlis duality, we have an injection
\[
\adj_{B\shortrightarrow C}^D(R, D + \Delta) \to \big(H^d_{\fram}(R(K_R)))^{\vee} \cong R.
\]
This makes $\adj_{B\shortrightarrow C}^D(R, D + \Delta)$ an ideal and proves \autoref{lem.adjIsIdeal}.  Additionally, it is worth noting that we could just as well have defined
\begin{equation}
\label{eq.AdjAsKernel}
\adj_{B\shortrightarrow C}^D(R, D + \Delta) := \Ann_R{\ker}\Big(H^d_{\fram}(R(K_R)) \xrightarrow{\cdot f^{1/n}} H^d_{\fram}(I_{B\shortrightarrow C})\Big).
\end{equation}

To prove \autoref{lem.adjIsIndep}, first note that the diagram defining $I_{B\shortrightarrow C}$ was independent of all choices of $K_R$, $f$ and $n$.
It is clear that the ideal $\adj_{B\shortrightarrow C}^D(R, D + \Delta)$ is independent of $n$ if it is independent of $f$, for if we choose $n(K_X+D+\Delta)=\Div(f)$ and $kn(K_X+D+\Delta)=\Div(f^k)$, and choose the same $n$th root of $f$, the definitions read the same.
Next, for fixed $K_R$, two choices of $f$ or of $f^{1/n}$ differ only by units of $R^+$, which certainly does not change \autoref{eq.AdjAsKernel}.

Finally, to show independence on $K_R$, suppose we have $n$ and $f$ such that $n(K_R+D+\Delta)=\Div(f)$ and $n(K_R'+D+\Delta)=\Div(fg^n)$.  That is, $(K_R'-K_R)=\Div(g)$, and we may assume that $\Div(g)\geq 0$.  Notice that $R(K_R') \xrightarrow{\cdot g} R(K_R)$ is an isomorphism.
Consider the commutative diagram:
\[
    \xymatrix@C=36pt{
        H^d_{\fram}(R(K_R')) \ar[d]_{\cdot g}^{\sim} \ar[r]^-{\cdot f^{1/n}g} &  H^d_{\fram}(I_{B\shortrightarrow C})\ar@{=}[d]\\
        H^d_{\fram}(R(K_R))  \ar[r]_-{\cdot f^{1/n}} & H^d_{\fram}(I_{B\shortrightarrow C})
    }
\]
By the proof of \autoref{lem.adjIsIdeal}, the Matlis dual of the image of the top row is $J' = \adj_{B\shortrightarrow C}^D(R, D + \Delta)$; the adjoint ideal computed with respect to $K_R'$.  The Matlis dual of the bottom row is $J = \adj_{B\shortrightarrow C}^D(R, D + \Delta)$; the adjoint ideal computed with respect to $K_R$.  From the diagram, we see that the two horizontal images are the same.  Hence the diagram induces an isomorphism $J \to J'$.  Now, the Matlis dual of the map
\[
    H^d_{\fram}(R(K_R')) \xrightarrow{\cdot g} H^d_{\fram}(R(K_R))
\]
is just an $R$-module isomorphism between rank-1 free $R$-modules, hence it can be identified with multiplication by a unit.  This multiplication by a unit also induces our isomorphism $J \to J'$ and so $J = J'$.
\end{proof}

\begin{remark}[Non-effective $\Delta$]
\label{rmk.NonEffectiveDeltaForMixedChar}
With notation as above, suppose that $d \in R$ does not vanish at the generic point of $D$.  It then easily follows by the arguments above that
\[
d \cdot \adj_{B\shortrightarrow C}^D(R, D + \Delta) = \adj_{B\shortrightarrow C}^D(R, D + \Delta + \Div(d)).
\]
Hence if $\Delta$ is non-effective, we can choose such a $d$ so that $\Delta + \Div(d) \geq 0$.  Thus we can define
\[
\adj_{B\shortrightarrow C}^D(R, D + \Delta) = {1 \over d} \cdot \adj_{B\shortrightarrow C}^D(R, D + \Delta + \Div(d)).
\]
Many of the results of this paper easily generalize to this setting by the formula above, but we leave details to the reader.
\end{remark}




\begin{definition}\label{def:purely_BCM_regular}
	Let $R$ be a complete normal local domain of residue characteristic $p>0$. Let $\Delta\geq0$ be a $\bQ$-divisor and $D$ a prime Weil divisor such that $(R, \Delta+D)$ is $\bQ$-Cartier and no component of $\Delta$ is equal to $D$.  Let $B$ and $C$ be compatibly chosen big Cohen-Macaulay $R^+$ and $(R/I_D)^+$-algebras respectively. We say that $(R,\Delta+D)$ is \emph{purely $\mathrm{BCM}_{B \shortrightarrow C}$-regular} 
	if
	\[
        \adj_{B\shortrightarrow C}^D(R,D+\Delta)=R.
	\]
\end{definition}

We can then give a definition of purely $\mathrm{BCM}$-regular that is independent of choices:

\begin{definition}
	With notation as in \autoref{def:purely_BCM_regular}, we say that $(R,\Delta+D)$ is \emph{purely $\mathrm{BCM}$-regular} if it is purely $\mathrm{BCM}_{B\shortrightarrow C}$-regular for every compatible choice of
perfectoid big Cohen-Macaulay algebras
$B$ and $C$ over $R^+$ and $(R/I_D)^+$ respectively.
	\end{definition}

\begin{remark}
It will follow from \autoref{thm:uniformBC} that, in fact, $(R,\Delta+D)$ is \emph{purely $\mathrm{BCM}$-regular} if and only if it is purely $\mathrm{BCM}_{B\shortrightarrow C}$-regular for one single sufficiently large compatible choice of perfectoid big Cohen Macaulay $R^+$- and $(R/I_D)^+$- algebras $B\to C$.
\end{remark}

We now compare the ideal $\adj_{B\shortrightarrow C}^D(R, D + \Delta)$ with the BCM test ideal $\mytau_B(R,\Delta)$, defined in \cite{MaSchwedeSingularitiesMixedCharBCM} and recalled in \autoref{def:BCM_test_ideal}.  This is an analog of the fact that if $(X,D+\Delta)$ is PLT then $(X,(1-\epsilon)D+\Delta)$ is KLT.

\begin{proposition}
\label{prop.AdjIsSmallerThanTau}
Let $(R,\m)$ be a complete normal local domain of residue characteristic $p>0$ and let $D$ be a prime divisor. Suppose that $\Delta, \Delta' \geq 0$ are two $\bQ$-divisors with no common components with $D$ such that $K_R + \Delta'$ is $\bQ$-Cartier. Further suppose that $D + \Delta$ is $\bQ$-Cartier. Then
\[
\adj_{B\shortrightarrow C}^D(R, \Delta' + D + \Delta ) \subseteq \mytau_B(R, \Delta' + (1-\epsilon) (D + \Delta)) \subseteq \mytau_B(R, \Delta')
\]
for any {rational number} $1 \geq \varepsilon > 0$.
\end{proposition}
\begin{proof}
The second containment is clear (see \cite[Lemma 6.11]{MaSchwedeSingularitiesMixedCharBCM}),  so we need only prove the first one.
Choose $g \in K(R)$ such that $\Div g = n(K_R + \Delta' + (1-\varepsilon)(D+\Delta))$ (note $g$ is only in the fraction field of $R$) and choose $h \in R$ such that $\Div h = n\varepsilon (D+ \Delta)$.  We have the diagram:
\[
\xymatrix{
0 \ar[r] & I_D \ar[r] \ar[d] & R \ar[r] \ar[d] &  R/I_D \ar[d] \ar[r] & 0\\
0 \ar[r] & (I_D)^+ \ar[r]\ar[d]  & R^+ \ar[r]\ar[d]  & (R/I_D)^+ \ar[r]\ar[d]  & 0\\
  \ar[r] & I_{B\shortrightarrow C} \ar[r] & B \ar[r] & C \ar[r]^{+1} &
}
\]
Since $h \in I_D$ we have $h^{1/n} \in (I_D)^+$, since $(I_D)^+$ is prime, and hence from the map of triangles
\[
\xymatrix{
0 \ar[r] & h^{1/n} B \ar@{.>}[d]_{\nu} \ar[r] & B \ar[r] \ar[d] & B/(h^{1/n}B) \ar[r] \ar[d] & 0\\
& I_{B\shortrightarrow C} \ar[r] & B \ar[r] & C \ar[r]_{+1} &
}
\]
we obtain a map $\nu : h^{1/n} B \to I_{B\shortrightarrow C}$.  Consider the factorization:
\[
\xymatrix{
                        &                                                     & \mathrm{im}_1 \ar@{.>}[dd] \ar@{^{(}->}[dr]\\
H^d_{\fram}(I_D) \ar@{->>}[drr] \ar@{->>}[r] & H^d_{\fram}(R(K_R)) \ar@{->>}[ur] \ar[rr]^{\cdot g^{1/n} h^{1/n}} & & H^d_{\fram}(h^{1/n} B) \ar[r]^-{\nu} & H^d_{\fram}(I_{B\shortrightarrow C}).\\
                        &                                                     & \mathrm{im}_2 \ar@{^{(}->}[urr] \\
}
\]
The vertical dotted arrow is surjective by the diagram.
On the other hand, via the isomorphism of $h^{1/n} B \cong B$, we may view the middle map (defining $\mathrm{im}_1$) to be $H^d_{\fram}(R(K_R)) \xrightarrow{\cdot g^{1/n}} H^d_{\fram}(B)$, and hence we see that the Matlis dual of $\mathrm{im}_1$ is just $\mytau_B(R, \Delta' + (1-\epsilon) (D + \Delta))$ {
(note that $K_R$ is not effective, but there is still a well defined map $R(K_R) \xrightarrow{\cdot g^{1/n}} B$)}.

Taking Matlis duals $(\bullet)^{\vee}$ of the entire diagram, we have
\[
\xymatrix{
\big(H^d_{\fram}(K_R)\big)^{\vee} \ar[d]_{\sim} &\ar@{_{(}->}[l]  (\mathrm{im}_1)^{\vee} \ar[d]_{\sim}& \ar@{_{(}->}[l] (\mathrm{im}_2)^{\vee} \ar[d]_{\sim}\\
R & \ar@{_{(}->}[l] \mytau_B(R, \Delta' + (1-\epsilon) (D + \Delta)) & \ar@{_{(}->}[l] \adj_{B\shortrightarrow C}^D(R, D+ \Delta' + \Delta )
}
\]
which completes the proof.
\end{proof}

\begin{corollary}
\label{cor.PBCMImpliesBCM}
	With notation as in \autoref{prop.AdjIsSmallerThanTau}, if $(R,D+\Delta+\Delta')$ is  purely $\mathrm{BCM}$-regular then $(R,\Delta'+(1-\epsilon)(D+\Delta))$ is $\mathrm{BCM}$-regular.
	\end{corollary}
\begin{proof}
	Choose a sufficiently large perfectoid big Cohen-Macaulay algebra $B$ such that $\mytau_B(R,\Delta'+(1-\epsilon)(D+\Delta))=R$ if and only if $(R,\Delta'+(1-\epsilon)(D+\Delta))$ is $\mathrm{BCM}$-regular by \cite[Proposition 6.10]{MaSchwedeSingularitiesMixedCharBCM}.  Let $C$ be an perfectoid big Cohen-Macaulay $R/I_D$-algebra such that $B\shortrightarrow C$ is a compatible choice, which exists by \cite[Theorem 1.2.1]{AndreWeaklyFunctorialBigCM}.  The result follows by applying \autoref{prop.AdjIsSmallerThanTau} to $B\shortrightarrow C$.
\end{proof}

\begin{remark}[Non-prime $D$]
    \label{rem.NonPrimeD}
    For non-prime $D$, there are several potential definitions and we work exclusively in the case that $D$ is prime in this paper.  Nevertheless, let us suggest a more general definition inspired by \cite[Lemma 4.22]{BMPSTWW-MMP}, also see \cite{TY-MMP}.
    If $D = \sum_{i = 1}^t D_i$, we choose a $C_i$, a perfectoid big Cohen-Macaulay $(R/I_{D_i})^+$-algebra for each $i$ with a map $B_i \to C_i$, for $B_i$ a perfectoid big Cohen-Macaulay $R^+$ algebra, satisfying the diagram \autoref{eq.BCCompatibleDiagram}.  
    We define
    \[
        \adj^D_{\oplus B_i\to \oplus C_i}(R, \Delta + D) = \sum_{i = 1}^t \adj^{D_i}_{B_i \to C_i}(R, \Delta+D).
    \]
    Several properties of this object are effectively proven in \cite{BMPSTWW-MMP} in the special case that, for all $i$, $B_i$ is the $\fram$-adic completion of $R^+$ and $C_i$ is the $\fram$-adic completion of $(R/I_{D_i})^+$ for all $i$.  Simply set $X = \Spec R$.
\end{remark}

\subsection{The different vs the different}
\label{sec.Different}
Let $R$ be a Noetherian normal domain, $X = \Spec R$, $D$ a prime Weil divisor, and $\Delta \geq 0$ a $\bQ$-divisor  with no components of $\Delta$ equal to $D$ and such that $K_R + D + \Delta$ is $\bQ$-Cartier of index $n$. Let $\pi \colon D^{N} \to D$ be the normalization of $D$. By abuse of notation, we also denote the composition $D^N \twoheadrightarrow D \hookrightarrow X = \Spec R$ by $\pi$. Following \cite[Section 4.1]{KollarKovacsSingularitiesBook}, we can define a $\bQ$-divisor $\diff_{R_{D^N}}(D + \Delta)$ on $D^N$ so that
\[
\pi^*(K_R + D + \Delta) \sim_{\bQ} K_{D^N} +  \diff_{R_{D^N}}(D + \Delta).
\]

This is done as follows. Since $R$ is normal, we can pick $U \subseteq \Spec R$ to be a regular open subset with complement of codimension at least two and such that $D\cap U$ is regular. 
The residue map induces an isomorphism $\omega_R(D)|_{D\cap U} \simeq \omega_{D}|_{D \cap U}$ (cf.\ \cite[Definition 4.1]{KollarKovacsSingularitiesBook}) which in turn yields a rational section $s$ of the rank-1-reflexive sheaf
\[
\sHom_{\cO_{D^{\nm}}}\big(\pi^* \cO_X(n(K_R + D + \Delta)), \cO_{D^{\nm}}(nK_{D^{\nm}})\big) \cong \cO_{D^{\nm}}(nK_{D^{\nm}}) \otimes \pi^* (\cO_X(n(K_R + D + \Delta)))^{-1}.
\]
We define $\diff_{R_{D^{\nm}}}(D + \Delta) = -\frac{1}{n} \mathrm{div}(s)$.\\

Equivalently, we can choose $K_R$ (using prime avoidance)  to be equal to $-D + G$ for some Weil divisor $G$ with no common components with $D$.
Consider the exact sequence
\[
0 \to \omega_R \to \omega_R(D) \to \omega_D \to \myH^{-d+1}(\omega_R^{\mydot}) \to \dots
\]
coming from applying Grothendieck duality to $0 \to \cO(-D) \to \cO \to \cO_D \to 0$ and taking cohomology.
Since $R$ is S2, it is Cohen-Macaulay in codimension 2.  Thus, $\myH^{-d+1}(\omega_R^{\mydot})$ is zero in codimension 2 on $R$ and hence in codimension 1 on $R/I_D$.  In particular, the S2-ification of $\coker(\omega_R \to \omega_R(D))$ is isomorphic to $\omega_D$, using \cite[0AWE]{stacks-project}.  Now then, since $K_R + D \geq 0$, we have a chosen section $1 \in \omega_R(D)$ and its image in $\omega_D$.  On the other hand, we have the generic isomorphism $\omega_{D^{N}} \to \omega_D$, and so we obtain a rational section of $\omega_{D^{N}}$, and hence a possibly non-effective
divisor $K_{D^{\nm}}$ (if $D$ is normal, it is effective).  We can now define
\[
\diff_{R_{D^{\nm}}}(D + \Delta) = \frac{1}{n}{\mathrm{div}_{D^{\nm}}}\, \pi^*f - K_{D^{\nm}}
\]
for $f$ such that $\mathrm{div}_R(f) = n(K_R + D + \Delta)$.
It is easy to see that the above two constructions, in fact, coincide.

\section{Adjunction and inversion of adjunction}

We prove the first main result of this article.

\begin{theorem}
\label{thm.AdjointMultiplierInversionOfAdjunction}
Suppose $(R, \fram)$ is a complete normal local domain of residue characteristic $p > 0$.  Fix a prime Weil divisor $D$ on $\Spec R$ with $I_D = R(-D)$ the defining ideal and set $R_D = R/I_D$ as well as $R_{D^{\nm}}$ to be the normalization of $R/I_D$.  Suppose that $\Delta \geq 0$ is a $\bQ$-divisor such that no component of $\Delta$ is equal to $D$ and such that $K_R + D + \Delta$ is $\bQ$-Cartier.  Then for compatibly chosen $B, C$ as in \autoref{eq.BCCompatibleDiagram}, we have
\[
\adj_{B\shortrightarrow C}^D(R, D+\Delta) \cdot R_{D^{\nm}} = \mytau_C\big( R_{D^{\nm}}, \diff_{R_{D^{\nm}}}(D+\Delta)) \big).
\]
\end{theorem}
\begin{proof}
Fix a canonical divisor $K_R = -D + G$, where $G \geq 0$ does not contain $D$ as a component.  Following the notation as in \autoref{sec.Definitions}, choose $f \in R$ such that $\Div(f) = n(K_R + D + \Delta)$.  Consider the diagram:
\[
\xymatrix{
0 \ar[r] & I_D \ar[d]_{\cdot f^{1/n}} \ar[r] & R \ar[d]_{\cdot f^{1/n}} \ar[r] & R/I_D \ar[d]^{\cdot f^{1/n}} \ar[r] & 0 \\
0 \ar[r] & (I_D)^+ \ar[r]  & R^+ \ar[r] & (R/I_D)^+ \ar[r] & 0.
}
\]
Since in a normal finite extension $R \subseteq S \subseteq R^+$ with $f^{1/n} \in S$, we have that $\Div_S(f^{1/n}) \geq \pi^* (K_R + D)$, it follows that $f^{1/n} \cdot R(K_R+D) \subseteq f^{1/n} S(\pi^*(K_R + D)) \subseteq R^+$.  Hence we have the following diagram:
\[
\xymatrix{
0 \ar[r] & I_D \ar[d] \ar[r] & R \ar[d] \ar[r] & R/I_D \ar[d] \ar[r] & 0 \\
0 \ar[r] & R(K_R) \ar[d]_{\cdot f^{1/n}} \ar[r] & R(K_R + D) \ar[d]_{\cdot f^{1/n}} \ar[r] & (\coker) \ar@{.>}[d]^{\cdot f^{1/n}} \ar[r] & 0 \\
0 \ar[r] & (I_D)^+ \ar[r]  & R^+ \ar[r] & (R/I_D)^+ \ar[r] & 0
}
\]
As in \autoref{sec.Different}, the S2-ification of $(\coker)$ is isomorphic to $\omega_D$, and so $f^{1/n}$ multiplies $\omega_D$ into $(R/I_D)^+$.

We see by \autoref{sec.Different} that
\[
\diff_{R_{D^{\nm}}}(D+\Delta) = {1 \over n} \Div_{D^{\nm}} \overline{f} - K_{D^{\nm}}.
\]
for $\overline{f} = f|_{D^N}$. Consider the following diagram:
\[
\xymatrix{
0 \ar[r] & R(K_R) \ar[dd]_{\cdot f^{1/n}} \ar[r] & R(K_R + D) \ar[dd]_{\cdot f^{1/n}} \ar[r] & (\coker) \ar[dd]^{\cdot f^{1/n}} \ar[drr] \ar[r] & 0 & \\
         &                                        &                                          & & & \omega_{D} \ar[dll]_{\cdot f^{1/n}} & \omega_{D^{\nm}} \ar[l]^{\Tr}  \\
0 \ar[r] & (I_D)^+ \ar[r] \ar[d]  & R^+ \ar[r] \ar[d]  & (R/I_D)^+ \ar[r] \ar[d] & 0 &\\
& I_{B \shortrightarrow C} \ar[r] & B \ar[r] & C \ar[r]_{+1} & &
}
\]
Since $K_{D^{\nm}}+\diff_{R_{D^{\nm}}}(D+\Delta) = {1 \over n} \Div_{D^{\nm}} \overline{f}$, by construction in \cite{MaSchwedeSingularitiesMixedCharBCM}, the Matlis dual of
\[
\Image\big( H^{d-1}_{\fram}(\omega_{D^{\nm}} ) \to H^{d-1}_{\fram}(C) \big)
\]
is $\mytau_C\big( R_{D^{\nm}}, \diff_{R_{D^{\nm}}}(D+\Delta)) \big)$. On the other hand, the maps $H^{d-1}_{\fram}(\omega_{D^{\nm}} ) \to H^{d-1}_{\fram}(\omega_D)$ and $H^{d-1}_{\fram}(\coker) \to H^{d-1}_{\fram}(\omega_D)$ are surjective, since the modules agree generically.  Therefore, the Matlis dual of
\[
\Image\big( H^{d-1}_{\fram}(\coker) \xrightarrow{\cdot{f^{1/n}}} H^{d-1}_{\fram}(C)\big)
\]
 is also $\mytau_C\big( R_{D^{\nm}}, \diff_{R_{D^{\nm}}}(D+\Delta)) \big)$.

 Taking local cohomology, we obtain the following diagram with images of maps in the middle row:
 \[
 \xymatrix{
&  H^{d-1}_{\fram}(\coker) \ar[r] \ar@{->>}[d] & H^d_{\fram}(R(K_R)) \ar@{->>}[d] \ar[r] & H^d_{\fram}(R(K_R + D)) \ar[dd] \ar[r] & 0 \\
0 \ar[r] & \mathrm{im}_1 \ar@{_{(}->}[d] \ar@{^{(}->}[r] & \mathrm{im}_2 \ar@{^{(}->}[d]\\
H^{d-1}_{\fram}(B) = 0 \ar[r] & H^{d-1}_{\fram}(C) \ar[r] & H^d_{\fram}(I_{B\shortrightarrow C}) \ar[r] & H^d_{\fram}(B)
 }
 \]
Recall that the S2-ification of $(\mathrm{coker})$ is $\omega_D$, and hence $H^{d-1}_{\fram}(\coker) = H^{d-1}_{\fram}(\omega_D)$.
 Taking Matlis dual of the top and middle row yields the diagram
 \[
 \xymatrix{
 \adj_{B\shortrightarrow C}^D(R, \Delta+D) \ar@{_{(}->}[d] \ar@{->>}[r] & \mytau_C\big( R_{D^{\nm}}, \diff_{R_{D^{\nm}}}(D+\Delta)) \ar@{_{(}->}[d] \\
 R \ar[r] & (R/I_D)^{\nm}=R_{D^{\nm}}.
 }
 \]
 This completes the proof.
\end{proof}

\begin{remark}[Adjunction for non-prime $D$]
    Using the notation of \autoref{rem.NonPrimeD}, if $D = \sum_{i = 1}^t D_i$ is only reduced but not prime, then notice that $\adj^{D_i}_{B_i \to C_i}(X, D_i + \Delta + \sum_{j \neq i} D_j)$ is contained in the Matlis dual of $\Image\big(H^d_{\fram}(R(K_R)) \to H^d_{\fram}(R(K_R + D_j))\big) = I_{D_j}^{\vee}$ for any $j \neq i$.  Hence $\adj^{D_i}_{B_i \to C_i}(X, D_i + \Delta + \sum_{j \neq i} D_j) \subseteq I_{D_j}$.  In particular,
    \[
        \begin{array}{rl}
            \big(\adj^D_{\oplus B_i\to \oplus C_i}(R, \Delta + D)\big) \cdot R_{{D_j}}^{\nm}
            = & \Big(\sum_{i = 1}^t \adj^{D_i}_{B_i \to C_i}(R, \Delta+D) \big) \cdot R_{{D_j}}^{\nm}\\
            = & \big(\adj^{D_j}_{B_j \to C_j}(R, \Delta + D)\big) \cdot R_{{D_j}}^{\nm}\\
            = & \mytau_{C_j}\big(R_{D_j}^{\nm}, \diff_{R_{D_j^{\nm}}(D+ \Delta)}\big).
        \end{array}
    \]
    Therefore, since $R_D^{\nm} := (R/I_D)^{\nm} = \prod_{j = 1}^t (R/I_{D_j})^{\nm} = \prod_{j = 1}^t R_{D_j}^{\nm}$ we obtain that
    \[
        \big(\adj^D_{\oplus B_i\to \oplus C_i}(R, \Delta + D)\big) \cdot R_{{D}}^{\nm} = \mytau_{\oplus C_j}\big(R_{D}^{\nm}, \diff_{R_{D^{\nm}}(D+ \Delta)}\big)
    \]
    where we define the right side to be the product of the $\mytau_{C_j}\big(R_{D_j}^{\nm}, \diff_{R_{D_j^{\nm}}(D+ \Delta)}\big)$ since $R_{D}^{\nm}$ is not local.
    This generalizes \autoref{thm.AdjointMultiplierInversionOfAdjunction} to the case of a non-prime $D$.
\end{remark}

\begin{corollary}\label{cor:normal_centers}
    With notations as in \autoref{thm.AdjointMultiplierInversionOfAdjunction}, if $(D^{\nm}, \diff_{R_{D^{\nm}}}(D+\Delta))$ is \BCMReg{C}, then $D$ is normal.
    \end{corollary}
    \begin{proof}
    \autoref{thm.AdjointMultiplierInversionOfAdjunction} tells us that $\adj_{B\shortrightarrow C}^D(R, D+\Delta)\cdot R_{D^{\nm}}=R_{D^{\nm}}$. But $\adj_{B\shortrightarrow C}^D(R, D+\Delta)\subseteq R$. So $R\to R_{D^{\nm}}=(R/I_D)^{\nm}$ is surjective (see the last diagram in the proof of \autoref{thm.AdjointMultiplierInversionOfAdjunction}) and hence $D=D^{\nm}$.
    \end{proof}

	



\begin{corollary}\label{cor:singularities}
	With notation as in \autoref{thm.AdjointMultiplierInversionOfAdjunction}, $(R,D+\Delta)$ is purely $\mathrm{BCM}_{B\shortrightarrow C}$-regular  if and only if $((R/I_D)^{\nm},\diff_{D^{\nm}}(D+\Delta))$ is $\mathrm{BCM}_{C}$-regular.  In either case $D = D^{\nm}$.
	\end{corollary}
\begin{proof}
    The first statement follows from Nakayama's lemma and \autoref{thm.AdjointMultiplierInversionOfAdjunction}.  The second statement follows from \autoref{cor:normal_centers}.
\end{proof}

\begin{corollary}\label{cor:universal_singularities}
	With notation as in \autoref{thm.AdjointMultiplierInversionOfAdjunction}, $(R,D+\Delta)$ is purely $\mathrm{BCM}$-regular  if and only if $((R/I_D)^{\nm},\diff_{D^{\nm}}(D+\Delta))$ is $\mathrm{BCM}$-regular.
\end{corollary}
\begin{proof}
	Firstly, if $((R/I_D)^{\nm},\diff_{D^{\nm}}(D+\Delta))$ is $\mathrm{BCM}$-regular, then given any compatible choice $B\shortrightarrow C$, $((R/I_D)^{\nm},\diff_{D^{\nm}}(D+\Delta))$ is $\mathrm{BCM}_C$-regular and so by \autoref{cor:singularities} $(R,D+\Delta)$ is purely $\mathrm{BCM}_{B\shortrightarrow C}$-regular.
	
	Conversely, suppose $(R,D+\Delta)$ is purely $\mathrm{BCM}$-regular, in which case we may assume that $D$ is normal, and let $C$ be a perfectoid big Cohen-Macaulay $(R/I_D)^+$ algebra large enough to ensure $\mathrm{BCM}$-regularity.  Then choose a compatible $B\shortrightarrow C'$ by \cite[Theorem 1.2.1]{AndreWeaklyFunctorialBigCM}.  By \cite[Lemma 4.5]{MaSchwedeSingularitiesMixedCharBCM} there exists $\tilde{C}$ which comes with maps $C'\to \tilde{C}$ and $C\to \tilde{C}$, so that by \autoref{cor:singularities} applied to $B\to\tilde{C}$, $(R/I_D,\diff_D(D+\Delta))$ is $\mathrm{BCM}_{\tilde{C}}$-regular, and hence $\mathrm{BCM}_{C}$-regular, and hence $\mathrm{BCM}$-regular by our choice of $C$.
\end{proof}

\section{Application to BCM test ideal and the Brian\c{c}on--Skoda theorem}

We prove the following result which substantially generalizes \cite[Theorem 6.27 and Proposition 6.31]{MaSchwedeSingularitiesMixedCharBCM}.

\begin{theorem}
\label{thm.taucontainsSingularLocus}
Suppose $(R,\fram)$ is a complete normal local domain of residual characteristic $p > 0$ and that $\Delta\geq 0$ is a $\mathbb{Q}$-divisor such that $K_R + \Delta$ is $\bQ$-Cartier.  Further suppose that $Q \in \Spec R$ is a point such that the localization $(R_Q, \Delta_Q)$ is simple normal crossing with $\lfloor \Delta_Q \rfloor = 0$ (in particular, $R_Q$ is regular).  Then for any perfectoid big Cohen-Macaulay $R^+$-algebra $B$,
\[
(\mytau_B(R, \Delta))_Q = R_Q.
\]
\end{theorem}
\begin{proof}
We proceed by induction on the dimension of $R_Q$.  Consider first the case where $\dim R_Q = 0$.  Then $R_Q$ is a field and the statement is obvious since $\mytau_B(R, \Delta)$ is nonzero.

Now suppose we know the statement for dimension $< d$ and $\dim R = d$.  There are two cases.  First suppose that the localization $\Delta_Q = 0$.  In this case choose $D$ to be a prime divisor on $\Spec R$ passing through $Q$ such that $D_Q$ is nonsingular.  Further we may choose an effective $\bQ$-divisor $\Theta' \geq \Delta$ and $\Theta \geq 0$, both not containing $Q$ in their support, such that $K_R + \Theta'$ and $D + \Theta$ are $\bQ$-Cartier (in fact, we can take $\Theta'=\Delta$ in this case).

Otherwise if $\Delta_Q\neq 0$, choose $D \in \Supp \Delta$ to be some prime divisor passing through $Q$.  Now, $K_R + \Delta \vee D$ need not be $\bQ$-Cartier.  Fix $\Delta_D = \Delta - (\Delta \wedge D)$ and choose $\Theta' \geq \Delta_D$ such that $\Theta'_Q = (\Delta_D)_Q$ and $K_R + \Theta'$ is $\bQ$-Cartier.  Further choose $\Theta \geq 0$, not containing $Q$ in its support, so that $D + \Theta$ is $\bQ$-Cartier.

Under either assumption, we have $D + \Theta +\Theta' \geq \Delta$ while $K_R + \Theta' + D + \Theta$ is $\bQ$-Cartier.  Furthermore $(R_Q, D_Q + \Theta_Q+\Theta'_Q)$ is SNC and $\lfloor D_Q + \Theta_Q + \Theta'_Q \rfloor = D_Q$. Choosing a suitably compatible perfectoid big Cohen-Macaulay $(R/I_D)^+$-algebra $C$, we know that
\[
\mytau_B(R, \Delta) \supseteq \mytau_B(R, \Theta' + (1-\epsilon)(D + \Theta) ) \supseteq \adj_{B\shortrightarrow C}^D(R, \Theta'+\Theta+D)
\]
where the first inequality is simply because $\Theta' + (1-\epsilon)(D + \Theta) \geq \Delta$ and the second is by \autoref{prop.AdjIsSmallerThanTau}.  Multiplying by $(R/I_D)^{\nm} = R_{D^{\nm}}$ we use \autoref{thm.AdjointMultiplierInversionOfAdjunction} to obtain that
\[
\adj_{B\shortrightarrow C}^D(R, \Theta'+\Theta+D)\cdot R_{D^{\nm}} = \mytau_C\big( R_{D^{\nm}}, \diff_{R_{D^{\nm}}}(\Theta'+\Theta+D)) \big).
\]
By the induction hypothesis, the right side localized at $Q$ is $(R_{D^{\nm}})_Q$.  Hence by Nakayama's lemma, so is the left side.  The result follows.
\end{proof}

As an application, we obtain the following Brian\c{c}on--Skoda type result in mixed characteristic. To the best of our knowledge this version of the Brian\c{c}on--Skoda theorem was not known before in mixed characteristic.

\begin{corollary}[Brian\c{c}on--Skoda Theorem]\label{cor:briancon-skoda}
Let $(R,\fram)$ be a complete normal local domain of residue characteristic $p>0$ and of dimension $d$. Let $J$ be the defining ideal of the singular locus of $R$. Then there exists an integer $N$ such that $J^N\overline{I^h}\subseteq I$ for all $I\subseteq R$ where $h$ is the analytic spread of $I$. In particular, $J^N\overline{I^d}\subseteq I$ for all $I\subseteq R$.
\end{corollary}
\begin{proof}
First note that we can replace $I$ by its minimal reduction, so without loss of generality we may assume that $I$ is generated by $h$ elements. Now by \cite[Theorem 2.7]{HeitmannMaExtendedPlusClosure}, $\overline{I^h}\subseteq I^{epf}$ where $I^{epf}$ denotes the extended plus closure of $I$:
$$I^{epf}:=\{z\in R \;| \text{ there exists $c\neq 0$ such that } c^{1/p^e}z\in (I,p^n)R^+ \text{ for all $e$ and all $n$} \}.$$
Next we fix a perfectoid $R^+$-algebra $B'$, it is clear that
$$I^{epf}\subseteq \{z\in R \;| \text{ there exists $c\neq 0$ such that } c^{1/p^\infty}z\in (I,p^n)B' \text{ for all $n$} \}.$$

For a fixed $c$, we now apply Gabber's construction (see \cite[page 3]{GabberMSRINotes}) by setting $B=S_c^{-1}B'^\diamond$, where $B'^\diamond=(\prod^{\mathbb{N}} B')/(\bigoplus^{\mathbb{N}} B')$ and $S_c$ is the multiplicative system consisting of $(c^{\varepsilon_0}, c^{\varepsilon_1},\dots)\in B'^\diamond$ such that $\varepsilon_i\in \mathbb{N}[1/p]$ and $\varepsilon_i\to 0$. It is straightforward to check that $B$ is a big Cohen-Macaulay algebra of $R$ (see \cite{GabberMSRINotes}) and that if $c^{1/p^\infty}z\in \mathfrak{a}B'$ for some ideal $\mathfrak{a}\subseteq R$, then $z\in \mathfrak{a}B$. Moreover, we can replace $B$ by its $\m$-adic completion to assume that it is a perfectoid big Cohen-Macaulay $R^+$-algebra (first use \cite[Example 3.8 (7)]{BhattIyengarMaRegularRingsPerfectoid}, the $p$-adic completion of $B$ is perfectoid, and then note that $\m$-adic completion is the same as $(p,x_2,\dots,x_d)$-adic completion where $p,x_2,\dots,x_d$ is a regular sequence since $B$ is big Cohen-Macaulay, so \cite[Proposition 2.2.1]{AndreWeaklyFunctorialBigCM} applies). We can now either take a direct limit of this construction for all $c\neq 0$, or pick $c$ that works for every generator of $I^{epf}$ (note that $I^{epf}$ is a finitely generated ideal since it is inside $R$) to assume that we have a perfectoid big Cohen-Macaulay $R^+$-algebra $B$ such that:
\[
I^{epf}\subseteq \{z\in R \;|\; z\in (I,p^n)B \text{ for all $n$} \}.
\]
Thus $\overline{I^h}\subseteq (I,p^n)B\cap R$ for some fixed perfectoid big Cohen-Macaulay $R^+$-algebra $B$ and every $n$. For all $Q\in\Spec R$ such that $R_Q$ is regular, we can pick $\Delta\geq0$ such that $K_R+\Delta$ is $\mathbb{Q}$-Cartier and $\Delta_Q=0$ (since $K_R$ is principal at $Q$). If we pick $f$ such that $\Div_R(f)=n(K_R+\Delta)$, then $R\xrightarrow{\cdot f^{1/n}}B$ factors through $\omega_R\cong R(K_R)$ so we have induced maps
$$\cdot f^{1/n}: H_\m^d(\omega_R)\to H_\m^d(B\otimes\omega_R)\to H_\m^d(B).$$
Applying  Matlis duality, we have
\[
\xymatrix{
H_\m^d(\omega_R)^\vee \ar[d]^\cong & H_\m^d(B\otimes\omega_R)^\vee \ar[l] \ar[d]^\cong & H_\m^d(B)^\vee\ar[d]^\cong\ar[l] \\
R & \Hom_R(B, R) \ar[l] & \Hom_R(B, \omega_R) \ar[l]
}
\]
By construction we know that $\Image( H_\m^d(B)^\vee\to H_\m^d(\omega_R)^\vee)$ is $\mytau_B(R,\Delta)$ (see \cite[Proof of Theorem 6.12]{MaSchwedeSingularitiesMixedCharBCM}). Therefore by the commutative diagram, we know that
$$\Image(\Hom_R(B, R)\to R)\supseteq \mytau_B(R, \Delta).$$
By \autoref{thm.taucontainsSingularLocus}, $\mytau_B(R,\Delta)_Q=R_Q$. Thus we know there exists $\phi\in\Hom_R(B, R)$ such that $\phi(1)=x$ for some $x\notin Q$. Since this is true for every $Q$ such that $R_Q$ is regular, we see that $$J\subseteq \sqrt{\Image(\Hom_R(B, R)\to R)}.$$
Thus there exists $N$ such that $J^N\subseteq \Image(\Hom_R(B, R)\to R)$, that is, for every $r\in J^N$, there exists $\phi\in \Hom_R(B, R)$ such that $\phi(1)=r$. Finally, since $\overline{I^h}\subseteq (I,p^n)B\cap R$, applying $\phi$ we see that $r\overline{I^h}\subseteq (I,p^n)$. As this is true for all $r\in J^N$ and every $n$, we have that $J^N\overline{I^h}\subseteq \cap_n(I, p^n)=I$.
\end{proof}



\section{Comparison with the adjoint ideal from birational geometry}\label{sec.birational}

The main result of this section, \autoref{thm:birational_adjoint_comparison}, is a variant of \cite[Theorem 6.21]{MaSchwedeSingularitiesMixedCharBCM} stating that the BCM-test ideal is contained in the multiplier ideal sheaf.  Before moving on to the main result, we first give a slightly different proof of \cite[Theorem 6.21]{MaSchwedeSingularitiesMixedCharBCM} based on \cite[Proposition 3.7]{MaSchwedeSingularitiesMixedCharBCM}. We will apply the same strategy in the proof of \autoref{thm:birational_adjoint_comparison}.

 \begin{theorem}\cite[Theorem 6.21]{MaSchwedeSingularitiesMixedCharBCM}\label{thm:test_comparison}
 Given a complete normal local domain $(R, \mathfrak{m}, k)$ of dimension $d$ of residue characteristic $p > 0$, a $\bQ$-divisor $\Delta \geq 0$ for which $K_R + \Delta$ is $\bQ$-Cartier, and a proper birational map $\mu \colon Y \to \Spec R$ with $Y$ normal, there exists a big Cohen-Macaulay $R^+$-algebra $B$ such that
	\[
	\mytau_B(R,\Delta) \subseteq \mu_*\mathcal{O}_Y(\lceil K_Y - \mu^*(K_R+\Delta) \rceil).
	\]
\end{theorem}
\begin{proof}
	
	First, we may assume that $\mu$ is projective, and so it is the blow up of some ideal sheaf $J\subseteq R$, that is $Y = \Proj R[Jt]$. Second, by replacing $R[Jt]$ by its integral closure, we may assume $R[Jt]$ is normal.  Let $E$ be the reduced pre-image of $\{\m\}$. Arguing as in \cite[Proposition 3.7]{MaSchwedeSingularitiesMixedCharBCM}, we can find a commutative diagram
	\[
	\xymatrix{
		B' \ar[r] & B \\
		R[Jt] \ar[u] \ar[r] & R \ar[u],
	}
	\]
	with $B$ and $B'$ being big Cohen-Macaulay algebras over $R^+$ and $(\widehat{R[Jt]_{\mathfrak{n}+Jt}})^+$ respectively. This combined with the Sancho-de-Salas sequence \cite{SanchodeSalasBlowingupmorphismswithCohenMacaulayassociatedgradedrings} fits in the following diagram:
	\[
	\xymatrix{
		0=H^d_{\mathfrak{n}+Jt}(B') \ar[r] & H^d_{\fram}(B) \\
		H^d_{\mathfrak{n}+Jt}(R[Jt])_0 \ar[u] \ar[r] & H^d_{\mathfrak{m}}(R) \ar@{->>}[r] \ar[u] & H^d_E(Y, \mathcal{O}_Y),
	}
	\]
	where the bottom row is exact. Here $H^d_{\mathfrak{n}+Jt}(B')=0$, as $R[Jt]$ is of dimension $d+1$ and $B'$ is a big Cohen-Macaulay algebra, while $H^d_{\mathfrak{m}}(R) \to H^d_E(Y, \mathcal{O}_Y)$ is surjective, because it is Matlis dual to the injective morphism $\mu_* \omega_Y \to \omega_R$.
	
	In particular, we get an induced map $H^d_E(Y, \mathcal{O}_Y) \xrightarrow{\alpha} H^d_{\fram}(B)$. Let $f \in R$ be as in \autoref{def:BCM_test_ideal}, assuming $K_R \geq 0$. The following commutative diagram
	\[
	\xymatrix{
		H^d_{\fram}(B) & H^d_E(Y, \mathcal{O}_Y(\lfloor \mu^*(K_R+\Delta)\rfloor)) \ar[l]_(0.7){f^{\frac{1}{n}}\alpha}\\
		H^d_{\mathfrak{m}}(R) \ar@{->>}[r] \ar[u]_{f^{\frac{1}{n}}} & H^d_E(Y, \mathcal{O}_Y) \ar@{^{(}->}[u].
	}
	\]
	shows that the image of $H^d_{\mathfrak{m}}(R) \to H^d_E(Y, \mathcal{O}_Y(\lfloor \mu^*(K_R+\Delta) \rfloor)$ surjects onto the image of $H^d_{\mathfrak{m}}(R) \xrightarrow{f^{\frac{1}{n}}} H^d_{\fram}(B).$ Thus by Matlis duality
	\[
	\mytau_B(R,\Delta) \subseteq \mu_*\mathcal{O}_Y(\lceil K_Y - \mu^*(K_R+\Delta) \rceil). \qedhere
	\]
\end{proof}

\numberwithin{subremark}{theorem}
\begin{subremark}
    The proof above goes through in characteristic zero except that we do not know of a reference for the existence of weakly functorial BCM $R^+$-algebras, but only weakly functorial $R$-algebras \cite[Theorem 3.9]{HochsterHunekeApplicationsOfTheExistenceOfBCM}, see also the discussion at the end of \cite{HochsterHomologicalConjecturesAndLimCMSequences}.  
    
    However, in the proof of \autoref{thm:test_comparison} we really only needed that $f^{1/n} \in B$, which we can arrange again letting $S$ be the normalization of $R[f^{1/n}]$ and applying \cite[Theorem 3.9]{HochsterHunekeApplicationsOfTheExistenceOfBCM} to $S[JSt] \to S$.  Thus the result also holds in characteristic zero as long as one chooses a large enough big Cohen-Macaulay $R$-algebra.
\end{subremark}

We now recall the definition of the adjoint ideal  from characteristic zero birational geometry.

\begin{definition}
	Let $R$ be a normal local domain, $\Delta$ a $\bQ$-divisor and $D$ a prime Weil divisor such that $K_R+D+\Delta$ is $\bQ$-Cartier.  We define the birational adjoint to be
	$$\adj_{\mathrm{bir}}^D(R,D+\Delta)=\bigcap_{\mu: Y\to \Spec(R)}\mu_*\O_Y(\lceil K_Y-\mu^*(K_R+D+\Delta)+D'\rceil)$$ where the intersection runs over all proper birational morphisms, and $D'$ is the strict transform of $D$.
	If log resolutions exist in dimension $\dim(R)$ then the intersection stabilizes and can be computed on any  log resolution of $(R,D+\Delta)$ such that the strict transform of $D$ is nonsingular.
\end{definition}

\begin{remark}[Nonprime $D$]
    If $R$ is finite type over a field of characteristic zero and if $D = \sum_{i = 1}^t D_i$ where the $D_i$ are prime divisors, then one still defines $\adj_{\mathrm{bir}}^D(R, D+\Delta)$ with the same formula.  In that case, we believe it is well known to experts that
    \begin{equation}
        \label{eq.BirAdjointIsSumInChar0}
        \adj^D_{\mathrm{bir}}(R, D+ \Delta) = \sum_{i = 1}^t \adj^{D_i}_{\mathrm{bir}}(R, D+ \Delta).
    \end{equation}
    The $\supseteq$ containment is straightforward from the definitions.
    To show $\subseteq$, suppose $\pi : Y \to \Spec R$ is a log resolution separating the components of $D$ and $D_{i}'$ (respectively $D'$) is the strict transform of $D_i$ (respectively $D$).  Then, setting $M = \mu^* (K_R + D + \Delta)$ we have the diagram:
    \[
        {\small
            \xymatrix@C=12pt{
                0 \ar[r] & \bigoplus_{i = 1}^t \mu_*\O_Y(\lceil K_Y-M\rceil) \ar[d] \ar[r] & \bigoplus_{i = 1}^t \mu_*\O_Y(\lceil K_Y-M+D'_i\rceil) \ar[r] \ar[d] & \bigoplus_{i = 1}^t \mu_* \O_{D_i'}(\lceil K_{D_i'} - M|_{D'}\rceil) \ar[r] \ar[d] & 0 \\
                0 \ar[r] & \mu_*\O_Y(\lceil K_Y-M\rceil) \ar[r] & \mu_*\O_Y(\lceil K_Y-M+D'\rceil) \ar[r] & \mu_* \O_{D'}(\lceil K_{D'} - M|_{D'}\rceil) \ar[r] & 0 \\
            }
        }
    \]
    where the zeros on the right are due to local vanishing for multiplier ideals (in other words, relative Kawamata-Viehweg vanishing).  The maps on the left and right are surjective, hence so is the map in the middle.  As a consequence the containment $\subseteq$ holds in \autoref{eq.BirAdjointIsSumInChar0}.

    In view of this, we suggest that in positive characteristic it may be better to define $\adj^D_{\mathrm{bir}}(R, D+ \Delta)$ as a sum as in \autoref{eq.BirAdjointIsSumInChar0} as well.  We continue to work with prime $D$ however.
\end{remark}

Now we state and prove our main comparison result.

\begin{theorem} \label{thm:birational_adjoint_comparison}
	Let $(R,\m)$ be a complete local normal domain of residue characteristic $p>0$, $\Delta\geq 0$ a $\bQ$-divisor, and $D$ a prime Weil divisor such that $K_R+D+\Delta$ is $\bQ$-Cartier.
	Then for any proper birational map $\mu \colon Y\to \Spec(R)$ from a normal variety $Y$, there exists a compatibly chosen perfectoid big Cohen-Macaulay $R^+$- and $(R/I_D)^+$-algebras $B$ and $C$ such that
	\[
	\adj_{B\shortrightarrow C}^D(R,D+\Delta)\subseteq \mu_*\O_Y(\lceil K_Y-\mu^*(K_R+D+\Delta)+D'\rceil).
	\]
\end{theorem}

\begin{proof}
As in \autoref{thm:test_comparison}, we may assume that $Y = \Proj R[Jt]$ for some ideal $J \subseteq R$ such that $R[Jt]$ is normal and $E$ the reduced pre-image of $\{\m\}$.
Let $\bar{J}$ be the image of $J$ in $R/I_D$. By \autoref{thm:weaklyfunctorialdiagram},\footnote{Note that the assumption of \autoref{thm:weaklyfunctorialdiagram} is satisfied: the Rees algebra $R[Jt]$ (completed at the maximal ideal $\m+Jt$) is normal, $P_1=(Jt)$ and $P_2=\ker( R[Jt]\to (R/I_D)[\bar{J}t])$ are height one primes and their sum is a height two prime in $R[Jt]$ whose localization is regular {(as it is isomorphic to $R_{I_D}[(J R_{I_D})t]$)}.} we can find perfectoid big Cohen-Macaulay algebras $B'$ and $C'$ of the completions of $R[Jt]$ and $(R/I_D)[\bar{J}t]$ respectively, as well as perfectoid big Cohen-Macaulay $R^+$- and $(R/I_D)^+$-algebras $B$ and $C$, sitting inside the following commutative diagram:
\[
    \xymatrix{
        & \ar[rr]& & \mathcal{I}_{B'\shortrightarrow C'} \ar[rr] \ar[dd]|\hole   & &  B' \ar[rr] \ar[dd]|\hole & & C' \ar[dd]|\hole \ar[rr]^{+1} & & \\
       0 \ar[rr] && \mathcal{I}_D \ar[ru]  \ar[rr]\ar[dd] & & R[Jt] \ar[ru]  \ar[rr]\ar[dd] & & (R/I_D)[\bar{J}t]  \ar[ru]  \ar[dd] \ar[rr] & & 0 & \\
        & \ar[rr]|(0.43)\hole & & I_{B\shortrightarrow C} \ar[rr]|(0.53)\hole   & & B  \ar[rr]|(0.5)\hole & & C \ar[rr]^{+1} & & \\
      0 \ar[rr] && I_D \ar[ru]    \ar[rr] & &  R \ar[ru]  \ar[rr] & & R/I_D \ar[ru] \ar[rr] & & 0  ,
    }
\]
where $\mathcal{I}_D$, $I_{B\shortrightarrow C}$, and $\mathcal{I}_{B'\shortrightarrow C'}$ are appropriate kernels and homotopy kernels.

Pick $f$ as in \autoref{def:adjoint}. Analogously to \autoref{thm:test_comparison}, we can apply the Sancho-de-Salas sequence (see \cite[page 150]{LipmanCohenMacaulaynessInGradedAlgebras}) to obtain the following diagram with exact bottom row
\[
\xymatrix{
0=H^d_{\mathfrak{n}+Jt}(\mathcal{I}_{B'\shortrightarrow C'})_0 \ar[r] & H^d_{\fram}(I_{B\shortrightarrow C}) \\
H^d_{\mathfrak{n}+Jt}(\mathcal{I}_D)_0 \ar[u] \ar[r] & H^d_{\mathfrak{m}}(I_D) \ar@{->>}[r] \ar[u] & H^d_E(Y, \mathcal{O}_Y(-D')) \ar@{.>}[lu],
}
\]
Here $H^d_{\mathfrak{n}+Jt}(\mathcal{I}_{B'\shortrightarrow C'})=0$ for dimension reasons,
because it sits in an exact sequence
\[0= H^{d-1}_{\mathfrak{n}+Jt}(C') \to H^d_{\mathfrak{n}+Jt}(\mathcal{I}_{B'\shortrightarrow C'})\to H_{\frak{n} +Jt}^d(B')=0.
\]

 The bottom right arrow in the diagram is surjective as it is Matlis dual to $\mu_*\omega_Y(D') \hookrightarrow \omega_X(D)$. In particular, the morphism
\[
H^d_{\mathfrak{m}}(I_D) \xrightarrow{f^{\frac{1}{n}}} H^d_{\fram}(I_{B\shortrightarrow C})
\]
factors through
 $H^d_{\mathfrak{m}}(I_D) \to H^d_E(Y, \mathcal{O}_Y(\lfloor \mu^*(K_R+D+\Delta) - D'\rfloor)$, and so the image of the latter surjects onto the image of the former. Thus, by Matlis duality
 \[
\adj_{B\shortrightarrow C}^D(R,D+\Delta) \subseteq \mu_*\mathcal{O}_Y(\lceil K_Y - \mu^*(K_R+D+\Delta)+D' \rceil). \qedhere
\]
\end{proof}

\begin{remark}
In the above theorem, we can in fact pick a compatible choice of perfectoid big Cohen-Macaulay $R^+$- and $(R/I_D)^+$-algebras $B$ and $C$ that works for all possible birational maps $\mu$. This follows from \autoref{thm:uniformBC}. An alternative method might be possible by using the perfectoid nature of $B$ and $C$, and such comparisons should follow from \cite[Section 5]{MaSchwedePerfectoidTestideal}. We leave the interested reader to carry out the details of the alternative approach.
\end{remark}

\begin{remark}
\label{rmk:BhattR^+}
Recently, Bhatt \cite{BhattCMabsoluteintegralclosure} proved that the $p$-adic completion of $R^+$ is a (perfectoid) big Cohen-Macaulay algebra. Therefore, in \autoref{thm:test_comparison} and \autoref{thm:birational_adjoint_comparison}, we can take $B=\widehat{R^+}$ and $C=\widehat{(R/I_D)^+}$ and they work for all possible $\mu$ already. As a consequence of this fact, \autoref{thm:test_comparison} (resp., \autoref{thm:birational_adjoint_comparison}) holds for arbitrary big Cohen-Macaulay $\widehat{R^+}$-algebra $B$ (resp., arbitrary compatibly chosen big Cohen-Macaulay $\widehat{R^+}$ and $\widehat{(R/I_D)^+}$-algebras) because $\mytau_B(R, \Delta)\subseteq \mytau_{\widehat{R^+}}(R, \Delta)$ and $\adj_{B\shortrightarrow C}^D(R,D+\Delta)\subseteq \adj_{\widehat{R^+}\shortrightarrow \widehat{(R/I_D)^+}}^D(R,D+\Delta)$. We will not need this stronger result in the sequel though.
\end{remark}

\section{Comparison with the test ideal in characteristic $p>0$}
{The notions of pure BCM-regularity and pure F-regularity agree in positive characteristic. Indeed, this follows by adjunction as BCM-regularity and strong F-regularity agree by \cite[Proposition 5.3]{MaSchwedeSingularitiesMixedCharBCM}. In this section we show that in fact our BCM adjoint ideal coincides with Takagi's adjoint ideal in positive characteristic.  Throughout this section all rings will be $F$-finite and of characteristic $p>0$.

	First, we recall the definition of the latter ideal, \cite{TakagiPLTAdjoint, TakagiHigherDimensionalAdjoint,  TakagiAdjointIdealsAndACorrespondence}.}


\begin{definition}
    \label{def.TakagiCharPAdjoint}
Suppose that $R$ is a normal $F$-finite complete local domain of characteristic $p > 0$ and that $(R, D+\Delta)$ is a pair with $D$ reduced and $\Delta \geq 0$ a $\bQ$-divisor with no components in common with $D$.  Let $E = H^d_{\fram}(R(K_R))$ be the injective hull of the residue field.  Set $R^{\circ, D}$ to be the elements of $R$ not in any minimal prime of $I_D$ (in other words, the elements that do not vanish on any component of $D$).  We define
\[
    0_{E}^{*_D(D+\Delta)} = \{z \in E \;|\; \exists c \in R^{\circ, D} \text{ such that } 0 = (F^e_* c) \otimes z \in (F^e_* R((p^e-1)D + \lfloor p^e \Delta \rfloor) \otimes_R E, \forall e > 0 \}.
\]
We define
\[
    \tau_{I_D}(R, D+\Delta) = \Ann_R(0_{E}^{*_D(D+\Delta)}).
\]
\end{definition}
The roundings in \cite{TakagiPLTAdjoint} are slightly different than the ones above, however any difference can be absorbed into $c$.
It turns out if a $c' \in R^{\circ, D}$ is chosen such that $\Supp(\Delta) \subseteq \Supp(\Div(c))$ and such that $R_{c'}$ and $(R_D)_{c'}$ are regular, then a fixed power\footnote{To see this, using \cite[Proposition 3.5(1)]{TakagiPLTAdjoint}, we must show that $c'^n \in \tau_{I_D}(R, D+\Delta)$ for some $n$.  But we know that the formation of $\tau_{I_D}(R, D+\Delta)$ commutes with localization by \cite[Corollary 3.4]{TakagiPLTAdjoint} and for such a $c'$, it is not difficult to see that $\tau_{I_D}(R, D+\Delta)_{c'} = R_{c'}$.} $c = c'^n$ of that $c'$ can be chosen that works in the definition of all $z \in 0_{E}^{*_D(D+\Delta)}$ (in other words, such a $c = c'^n$ is a \emph{divisorial test element}).


Note that, since $E \cong H^d_{\fram}(R(K_R))$, we have that $(F^e_* R( (p^e-1)D + \lfloor p^e \Delta\rfloor)) \otimes_R E \cong F^e_* H^d_{\fram}(R(p^e(K_R + D) - D + \lfloor p^e \Delta)\rfloor))$.

On the other hand, by \cite[Proposition 1.1]{TakagiAdjointIdealsAndACorrespondence}  $\tau_{I_D}(R, D+\Delta)$ is the smallest ideal $J$ not contained in any minimal prime of $I_D$, such that for every
\[
\phi \in \Hom_R\big( (F^e_* R(\lceil (p^e-1)(D + \Delta)\rceil)), R \big) \subseteq \Hom_R\big( F^e_* R, R \big)
\]
we have $\phi(F^e_* J) \subseteq J$.  

\begin{remark}
Even if $D + \Delta$ is not effective we can still make sense of this definition.  The point is that $\tau_{I_D}(R, D+\Delta + \Div(c)) = c \tau_{I_D}(R, D+\Delta)$ for $c \in R^{\circ, D}$ as can be easily checked.  Therefore if $D + \Delta$ is not effective, choose $c \in R^{\circ, D}$ so that $D + \Delta + \Div(c)$ is effective, and define the fractional ideal:
\[
\tau_{I_D}(R, D+\Delta) = {1 \over c} \tau_{I_D}(R, D+\Delta + \Div(c)).
\]
Thus for what follows we will reduce without mention to the case that $D + \Delta \geq 0$.
\end{remark}

\begin{remark}[Non-prime $D$]
    Suppose $D = \sum_{i = 1}^t D_i$ with each $D_i$ prime.  It follows from the definition that $\tau_{I_{D_i}}(R, D + \Delta) \subseteq \tau_{I_{D}}(R, D + \Delta)$
    for every $i$.  Thus
    \[
        \tau_{I_{D}}(R, D + \Delta) \supseteq \sum_{i = 1}^t \tau_{I_{D_i}}(R, D + \Delta).
    \]
    However, the right side is also compatible with every $\phi$ as above and so by the minimality of $\tau_{I_{D}}(R, D + \Delta)$ we see that $\tau_{I_{D}}(R, D + \Delta) = \sum_{i = 1}^t  \tau_{I_{D_i}}(R, D + \Delta).$  Hence, if one makes the definition of $\adj_{\oplus B_i \shortrightarrow C_i}(R, D+\Delta)$ as suggested in \autoref{rem.NonPrimeD}, then the  main result of this section, \autoref{thm.TauEqualityCharP}, immediately generalizes to the case of non-prime $D$ since we can work one $D_i$ at a time. 
\end{remark}

We now prove a containment relating Takagi's characteristic $p > 0$ adjoint ideal with ours.  The strategy is similar to \cite{SmithFRatImpliesRat}.

\begin{proposition}
\label{prop.TauEasyCharPContainment}
Suppose that $(R, \fram)$ is an $F$-finite complete local ring of characteristic $p > 0$ and $(R, D + \Delta)$ is as in \autoref{def.TakagiCharPAdjoint}.  Further assume that $D$ is prime.  Then
\[
    \tau_{I_D}(R, D+\Delta) \subseteq \adj_{B\shortrightarrow C}^D(R, D+\Delta).
\]
\end{proposition}
\begin{proof}
Recall first that 
\[
    \adj_{B\shortrightarrow C}^D(R, D+\Delta) \cdot (R/I_D)^{\nm} = \mytau_C\big( R_{D^{\nm}}, \diff_{R_{D^{\nm}}}(D+\Delta)) \big)
\] 
by \autoref{thm.AdjointMultiplierInversionOfAdjunction}.  Hence since the right side is nonzero by \autoref{thm.taucontainsSingularLocus} (\cf \cite{MaSchwedePerfectoidTestideal}) the left side is also nonzero.  Thus, to prove the proposition, we must show that $\adj_{B\shortrightarrow C}^D(R, D+\Delta) \cdot (R/I_D)^{\nm}$ is compatible with the maps $\phi \in \Hom_R\big( (F^e_* R(\lceil (p^e-1)(D + \Delta)\rceil)), R \big) \subseteq \Hom_R(F^e_* R, R)$.

Let $f \in R$ be such that $\Div_R(f) = n(K_R + D + \Delta)$ as in \autoref{sec.Definitions}.  Fix a finite normal local extension $\eta \colon \Spec S \to \Spec R$ with $f^{1/n} \in S$.  We have the map
$R(K_R) \subseteq S(\eta^* K_R) \xrightarrow{\cdot f^{1/n}} I_{B \shortrightarrow C}$.
The Matlis dual of the image of $H^d_{\fram}(S(\eta^* K_R)) \xrightarrow{\cdot f^{1/n}} H^d_{\fram}(I_{B \shortrightarrow C})$ is a submodule $\kappa \subseteq \Hom_R(S, R)\cong S(K_S - \eta^*K_R)$ whose image under the evaluation-at-1 map to $R$ is exactly $\adj_{B\shortrightarrow C}^D(R, D+\Delta)$.  

Any $\phi \in \Hom_R\big( (F^e_* R(\lceil (p^e-1)(D + \Delta)\rceil)), R \big) \subseteq \Hom_R(F^e_* R, R)$ induces a map
\[
\varphi : \Hom_{F^e_* R}(F^e_* S, F^e_* R) \xrightarrow{\text{restrict domain}} \Hom_R(S, F^e_* R) \xrightarrow{\Hom(S, \phi)} \Hom(S, R).
\]
It would suffice to show that $\varphi(F^e_* \kappa) \subseteq \kappa$ since the following diagram commutes:
\[
\xymatrix{
F^e_* \Hom_R(S, R) \ar@{<->}[d]_{\sim} \\
\Hom_{F^e_* R}(F^e_* S, F^e_* R)  \ar[d]_{\varphi} \ar[r]^-{\mathrm{ev}@1} & F^e_* R \ar[d]^{\phi} \\
\Hom_{R}(S, R) \ar[r]^-{\mathrm{ev}@1} & R.
}
\]
Notice that $\varphi$ factors through
\[
\Hom_{F^e_* R}(F^e_* S, F^e_* R(\lceil (p^e-1)(D + \Delta)\rceil)) \cong F^e_* S(K_S - \eta^*K_R + \eta^* \lceil (p^e-1)(D + \Delta)\rceil)
\]
and hence it also factors through the smaller module
\begin{equation}
\label{eq.WhatItFactorsThrough}
F^e_* S(K_S - \eta^*K_R + \lceil \eta^* (p^e-1)(D + \Delta)\rceil).
\end{equation}
Now, we have the following commutative diagram
\[
\xymatrix{
H^d_{\fram}(S(\eta^*K_R)) \ar[d]_{\cdot F^e_* f^{(p^e-1)/n}} \ar[r]^-{\cdot f^{1/n}} 
& H^d_{\fram}(I_{B \shortrightarrow C}) \ar[d]^{F^e}\\
H^d_{\fram}(F^e_* S(\eta^*K_R)) \ar[r]_-{\cdot F^e_* f^{1/n}} 
& H^d_{\fram}(I_{B \shortrightarrow C}).
}
\]
The Matlis dual of the image of the first row is $\kappa$.  The Matlis dual of the image of the second row is $F^e_* \kappa$.
Therefore by Matlis duality, the dual to the left vertical map, $\psi \colon F^e_*S(K_S - \eta^*K_R) \cong F^e_* \Hom_R(S, R) \to \Hom_R(S, R) \cong S(K_S - \eta^*K_R)$
induced by $\cdot F^e_* f^{(p^e-1)/n}$ sends $F^e_* \kappa$ to $\kappa$.  It is not difficult to see that any map which factors through $F^e_* S(K_S - \eta^*K_R + \eta^* (p^e-1)(D + \Delta))$, is an $F^e_*S$-pre-multiple of $\psi$.  Indeed, this follows since the dual to $H^d_{\fram}(S) \to H^d_{\fram}(F^e_* S)$ generates $\Hom_S(F^e_* \omega_S, \omega_S)$ as an $S$-module (since both are dual in various ways to the Frobenius map $S \to F^e_* S$ which generates the Frobenius structures as an $F^e_* S$-module on $S$).  Thus $\varphi$ sends $F^e_* \kappa$ to $\kappa$ as well as desired.
\end{proof}

Before showing the reverse containment for appropriately chosen $B,C$, we first record the following analog of a transformation rule for $\tau_{I_D}(R, D+\Delta)$ under certain finite morphisms.  Note that recently Carvajal-Rojas and St\"abler \cite[Theorem 6.12]{CarvajalRojasStablerFiniteMapsSplittingPrimes} have proven a similar result, which for the most part holds much more generally.  Unfortunately, we do not believe that the implicit Cartier algebra we are using on $S$ is exactly the one pulled back in the  sense of Carvajal-Rojas and St\"abler (although the test ideal may be the same).  For this reason, we provide a direct proof.
\begin{proposition}
\label{prop.TransformationRuleForTakagiAdjoint}
 Let $R$ be an $F$-finite normal local domain of characteristic $p>0$, $\Delta$ a $\mathbb{Q}$-divisor and $D$ a reduced Weil divisor. Suppose that $R \subseteq S$ is a module finite normal extension and consider $\pi \colon \Spec(S) \to \Spec(R)$. Assume $\pi$ is {\`e}tale in a neighborhood of the generic point of any component of $D$.  Let $D'$ be a reduced divisor on $\Spec S$ with $D' \leq \pi^* D$ and such that for each prime component of $D$, there exists at least one component of $D'$ lying over it.
 Then the trace map $\Tr : S(K_S) \to R(K_R)$ induces
 \begin{equation}
 \label{takagiunderfinitemap}
\Tr(\tau_{I_{D'}}(S, \pi^*D+\pi^*\Delta - \Ram_\pi)) = \Tr(\tau_{I_{\pi^*D}}(S, \pi^*D+\pi^*\Delta - \Ram_\pi)) = \tau_{I_D}(R, D+\Delta)
 \end{equation}
\end{proposition}

\begin{proof}
First notice that
\[
\tau_{I_{D'}}(S, \pi^*D+\pi^*\Delta - \Ram_\pi) \subseteq \tau_{I_{\pi^*D}}(S, \pi^*D+\pi^*\Delta - \Ram_\pi).
\]
 Using the map-divisor correspondence \cite[Section 4]{BlickleSchwedeSurveyPMinusE}, an alternate formulation of $\tau_{I_D}(R, D+\Delta)$ is given by
 \[
  \begin{array}{l}
 \tau_{I_D}(R, D+\Delta) \smallskip \\
 = \sum_{e > 0} \Tr_{F^e} \left(
 F^e_*R(\lceil (1-p^e)(K_R + D) - p^e \Delta- \Div_R(c)\rceil )
 \right) \smallskip \\
 = \sum_{e > 0} \Tr_{F^e} \left(
 F^e_*R( (1-p^e)(K_R + D) - \lfloor p^e \Delta \rfloor - \Div_R(c) )
 \right)
 \end{array}
 \]
 where $c \in R$ is (an appropriate power of) an element of $R^{\circ, D}$ with $\Supp(\Delta) \subseteq \Supp(\Div_R(c))$ and such that $R_c$ and $(R_D)_c$ are regular. Assuming further that the ramification locus of $\pi$ is contained in $\Supp(\Div_R(c))$, applying the analogous formula on $S$, and using that the functoriality of the trace gives $\Tr_{F^e} \circ \Tr = \Tr_{F^e \circ \pi} = \Tr_{\pi \circ F^e} = \Tr \circ \Tr_{F^e}$, we have
 \[
 \begin{array}{l}
 \Tr\left( \tau_{I_{\pi^*D}}(S, \pi^*D+\pi^*\Delta - \Ram_\pi)\right) \smallskip \\ = \sum_{e > 0} \Tr \circ \Tr_{F^e} \left(
 F^e_*S(\lceil (1-p^e)(K_S + \pi^*D) - p^e(\pi^*\Delta - \Ram_\pi)- \Div_S(c)\rceil ) \right)
  \smallskip \\ = \sum_{e > 0} \Tr_{F^e} \left( F^e_*  \Tr \left(
 S(\lceil (1-p^e)(K_S + \pi^*D) - p^e(\pi^*\Delta - \Ram_\pi)- \Div_S(c)\rceil )
  \right) \right)
   \smallskip \\ = \sum_{e > 0} \Tr_{F^e} \left( F^e_* \Tr \left(
S(K_S + (1-p^e) \pi^*D - \lfloor p^e\pi^*\Delta \rfloor- p^e\pi^*K_R - \pi^*\Div_R(c) )
  \right)\right)
    \smallskip \\ \subseteq \sum_{e > 0} \Tr_{F^e} \left( F^e_* \Tr \left(
    S(K_S + (1-p^e) \pi^*D - \pi^*\lfloor p^e\Delta \rfloor- p^e\pi^*K_R - \pi^*\Div_R(c) )
  \right) \right)
   \smallskip \\ \subseteq \sum_{e > 0} \Tr_{F^e} \left(
    F^e_* R(K_R + (1-p^e) D - \lfloor p^e\Delta \rfloor- p^e K_R - \Div_R(c) )
  \right) \smallskip \\
  = \tau_{I_D}(R, D+\Delta).
  \end{array}
 \]

 The reverse inclusion for the equality in \eqref{takagiunderfinitemap} follows similarly, using the strategy of proof in \autoref{prop.TauEasyCharPContainment}  stemming from \cite{SmithFRatImpliesRat} to say that $\Tr\left( \tau_{I_{D'}}(S, \pi^*D+\pi^*\Delta - \Ram_\pi)\right)$ is appropriately uniformly compatible and not contained in any minimal prime of $I_D$.  By definition $\tau_{I_{D'}}(S, \pi^*D+\pi^*\Delta - \Ram_\pi)$ is not contained in any minimal prime of $I_{D'}$, and hence since $\pi$ is \'etale over the generic points of $D$, neither is its trace.  We have, for $b \in \tau_{I_{D'}}(S, \pi^*D+\pi^*\Delta - \Ram_\pi)$ but $b$ not in any minimal prime of $I_{D'}$ (in other words a ``divisorial test element''), the following chain of containments.
 Here the sum in the first line above is taken over all $\phi \in \Hom_R\big( (F^e_* R(\lceil (p^e-1)(D + \Delta)\rceil)), R \big).$ 
 \[
 \begin{array}{l}
 \sum_\phi \phi\big( \Tr\left( \tau_{I_{D'}}(S, \pi^*D+\pi^*\Delta - \Ram_\pi)\right) \big) \smallskip \\
 = \Tr_{F^e}
 \Big( F^e_* \big( R ((1-p^e)K_R -\lceil (p^e-1)(D + \Delta) \rceil) \cdot \Tr\left( \tau_{I_{D'}}(S, \pi^*D+\pi^*\Delta - \Ram_\pi)\right) \big) \Big)
 \smallskip \\ =
 \Tr_{F^e} \Big( F^e_* \Tr (S((1-p^e)\pi^*K_R -\pi^* \lceil (p^e-1)(D + \Delta) \rceil) \cdot  \tau_{I_{D'}}(S, \pi^*D+\pi^*\Delta - \Ram_\pi)) \Big)
  \smallskip \\ =
  \Tr\Big(
  \sum_{e'}
  \Tr_{F^e} ( F^e_*
  (
  S((1-p^e)\pi^*K_R -\pi^* \lceil (p^e-1)(D + \Delta) \rceil)  \smallskip \\
  \qquad \cdot
    \Tr_{F^{e'}} ( F^{e'}_*S(\lceil (1-p^{e'})(K_S + \pi^*D) - p^{e'}(\pi^*\Delta - \Ram_\pi)- \Div_S(b)\rceil)
  )))\Big)
    \smallskip \\ =
    \Tr\Big(\sum_{e'} \Tr_{F^{e+e'}}(F^{e+e'}_*S(p^{e'}(1-p^e)\pi^*K_R -p^{e'}\pi^* \lceil (p^e-1)(D + \Delta) \rceil
    \smallskip \\
  \qquad + \lceil (1-p^{e'})(K_S + \pi^*D) - p^{e'}(\pi^*\Delta - \Ram_\pi)- \Div_S(b) \rceil
    ))\Big)
      \smallskip \\ = \Tr\Big(\sum_{e'} \Tr_{F^{e+e'}}(F^{e+e'}_*S(K_S - p^{e+e'}\pi^*K_R + (1-p^{e+e'})\pi^*D  - \Div_S(b)
      \smallskip \\
  \qquad + p^{e'} \pi^* \lfloor  (1-p^e) \Delta\rfloor - \lfloor p^{e'}\pi^* \Delta \rfloor
      ))\Big)
      \smallskip \\ \subseteq \Tr\Big(\sum_{e'} \Tr_{F^{e+e'}}(F^{e+e'}_*S(K_S - p^{e+e'}\pi^*K_R + (1-p^{e+e'})\pi^*D  - \Div_S(b)
   - \lfloor p^{e+e'}\pi^* \Delta \rfloor
      ))\Big) \smallskip \\ =
      \Tr\Big(\sum_{e'} \Tr_{F^{e+e'}}(F^{e+e'}_*S(\lceil (1-p^{e+e'})(K_S + \pi^*D) - p^{e+e'}(\pi^*\Delta - \Ram_\pi)- \Div_S(b)\rceil
      ))\Big) \smallskip \\ \subseteq  \Tr\big( \tau_{I_{D'}}(S, \pi^*D+\pi^*\Delta - \Ram_\pi)\big).

 \end{array}
 \]
 This completes the proof of the equality in \eqref{takagiunderfinitemap}. 
\end{proof}

\begin{theorem}
\label{thm.TauEqualityCharP}
Suppose that $(R, \fram)$ is an $F$-finite complete local ring of positive characteristic $p > 0$ and $(R, D + \Delta)$ is a pair as in \autoref{def.TakagiCharPAdjoint}.  Additionally assume that $D$ is prime and that $K_R + D + \Delta$ is $\bQ$-Cartier.  Then for any map $B \to C$ of big Cohen-Macaulay $R^+$ and $(R/R(-D))^+$ modules (for instance $B = R^+$ and $C = (R/R(-D))^+$), we have that
\[
    \tau_{I_D}(R, D+\Delta) = \adj_{B\shortrightarrow C}^D(R, D+\Delta).
\]
\end{theorem}

\begin{proof}
We begin by supposing that $K_R = -D + G$ where $G$ has no components common with $D$.  Thus if $\Div(f) = n(K_R + D + \Delta)$, we see that $f$ is a unit at the generic point of $D$.  Since we are in characteristic $p > 0$, by \cite[Lemma 4.5]{BlickleSchwedeTuckerTestAlterations}, there exists a finite separable extension $R \subseteq S$ with $\pi : \Spec S \to \Spec R$ with $\pi^* (K_R + D + \Delta)$ Cartier.  If $n$ is not divisible by $p$, this is very easy, simply take the $n$th root of $f$.  If $n$ is divisible by $p$, then by construction in that proof, $\pi$ is ramified only where $f$ vanishes, and $f$ does not vanish at $D$.  Either way $\pi$ is \'etale over the generic points of $D$.  Choose the prime divisor $D'$ on $S$ lying over $D$ with $I_D$ contained in our choice of $I_D^+$.

In this case, we write $\pi^*(K_R + D + \Delta) = K_S + D' + (\pi^* D - D' + \pi^* \Delta - \Ram_{\pi}) = \Div(g)$.  Note that $(\pi^* D - D' + \pi^* \Delta - \Ram_{\pi})$ may not be effective, and so we implicitly use the idea of \autoref{rmk.NonEffectiveDeltaForMixedChar} to reduce to the case that it is.  Consider the factorization
\[
\alpha: H^d_{\fram}(R(K_R)) \to H^d_{\fram}(S(\pi^* K_R)) \xrightarrow{\beta} H^d_{\fram}(I_D^+)
\]
Since $H^d_{\fram}(S(K_S)) = H^d_{\fram}(S(\pi^* K_R - \Ram)) \twoheadrightarrow H^d_{\fram}(S(\pi^* K_R))$ surjects for dimension reasons, we see that the image of $\beta$ is the Matlis dual of $\adj^D_{B\shortrightarrow C}(S, \pi^* (D+\Delta) - \Ram_{\pi})$, and the image of $\alpha$ is $\adj^D_{B\shortrightarrow C}(R, D+\Delta)$.  Hence we have that
\[
\Tr\big(\adj^D_{B\shortrightarrow C}(S, \pi^* (D+\Delta) - \Ram_{\pi})\big) = \adj^D_{B\shortrightarrow C}(R, D+\Delta).
\]
Therefore, in view of \autoref{prop.TransformationRuleForTakagiAdjoint}, it suffices to prove the result on $S$.  Hence, we may now assume that $K_R + D + \Delta = \Div(f)$ is Cartier.

Suppose $z \in 0_{E}^{*_D(D+\Delta)}$.  This just means that there exists $c \in R^{\circ, D}$ such that
 \[c^{1/p^e} \otimes  z = 0 \in (R((p^e - 1)(D + \Delta)))^{1/p^e} \otimes E.\]
But \[(R((p^e - 1)(D + \Delta)))^{1/p^e} \otimes E \cong F^e_* H^d_{\fram}(R(K_R + (p^e-1)(K_R + D + \Delta)) \cong F^e_* E.\]
Define:
\[
\begin{array}{rcl}
I_D^{D + \Delta, \infty} & := & \bigcup_{e > 0} (R((p^e-1)(D +  \Delta)))^{1/p^e} \subseteq K(R)^{1/p^{\infty}} \subseteq K(R^+).
\end{array}
\]
\begin{claim} We have $f\cdot I_D^{D + \Delta, \infty} \otimes_R E$ maps into $H^d_{\fram}(I_D^+)$.
\label{clm.fMultipliesIDInfToIDPlus}
\end{claim}
\begin{proof}[Proof of claim \autoref{clm.fMultipliesIDInfToIDPlus}]
We work at a finite level.
\[
\begin{array}{rl}
& f\cdot (R((p^e-1)(D+\Delta) ))^{1/p^e} \otimes_R E \\
=  & (R((p^e-1)(D + \Delta) - p^e(K_R + D + \Delta)))^{1/p^e} \otimes_R H^d_{\fram}(K_R) \\
 = & H^d_{\fram}(R(-D - \Delta)^{1/p^e})
 \end{array}
\]
which certainly maps into $H^d_{\fram}(I_D^+)$ under the natural map.
\end{proof}
Since $c^{1/p^e} \otimes z = 0$ in one term in the limit, we see that $c^{1/p^e}$ annihilates $1 \otimes z$ in $I_D^{D + \Delta, \infty} \otimes_R E$.  Hence by \autoref{clm.fMultipliesIDInfToIDPlus}, the image of $z$ under the map
\[
H^d_{\fram}(R) \xrightarrow{\cdot f} H^d_{\fram}(I_D^+)
\]
is also annihilated by $c^{1/p^e}$.
\begin{claim}
\label{clm.TightClosureInIDPlus}
If $z \in H^d_{\fram}(I_D^+)$ is such that $c^{1/p^e} z = 0$ for all $e$ and some fixed $c \in R^{\circ, D}$, then $z = 0$.
\end{claim}
\begin{proof}[Proof of claim \autoref{clm.TightClosureInIDPlus}]
We have \[0 \to H^{d-1}_{\fram}(R_D^+) \to H^d_{\fram}(I_D^+) \to H^d_{\fram}(R^+) \to 0.\]  The image $z'$ of $z$ in $H^d_{\fram}(R^+)$ also has $c^{1/p^e}z = 0$.  But now, $z' \in H^d_{\m}(S)$ for some finite extension $S$ of $R$.  By the valuative criterion of tight closure, \cite[Theorem on page 194]{HochsterFoundations}, since $c^{1/p^e} \otimes z' = 0 \in S^+ \otimes_S H^d_{\fram}(S)$, we see that $z'$ is in the tight closure of $0$ in $H^d_{\fram}(S)$.  But the tight closure of zero is just the kernel of $H^d_{\fram}(S) \to H^d_{\fram}(R^+)$ by \cite{SmithTightClosureParameter}.  Hence $z' = 0$.

Thus there exists $y \in H^{d-1}_{\fram}(R_D^+)$ mapping to $z$ and by the fact $R^+$ is big Cohen-Macaulay, $c^{1/p^e} y = 0$ for all $e$.  Repeating the above argument, we see that $y = 0$.  Thus $z = 0$ as well, proving \autoref{clm.TightClosureInIDPlus}.
\end{proof}
We just showed that \[0_{E}^{*_D(D+\Delta)} \subseteq \ker\big(H^d_{\fram}(R(K_R)) \xrightarrow{\cdot f} H^d_{\fram}(I_D^+)\big),\] and so by duality, $\tau_{I_D}(R, D+\Delta) \supseteq \adj_{B\shortrightarrow C}(R, D+\Delta)$.
\end{proof}

\begin{corollary}
	In the notation of \autoref{thm.TauEqualityCharP}, $(R,D+\Delta)$ is purely $\mathrm{BCM}$-regular if and only if it is purely $F$-regular.
	\end{corollary}

\section{Application to surface and threefold log terminal singularities}
\label{sec.ApplicationToSurfaceAndThreefoldKLTSingularities}

Our goal in this section is to prove that a 2-dimensional KLT singularity is \BCMReg{} if the residual characteristic $p$ is bigger than 5, generalizing a result of \cite{CRMPSTCoversofRDP}. {
This implies the inversion of adjunction for three-dimensional PLT pairs with $p>5$ and normality of divisorial centres for such pairs.}

First we state a lemma that was essentially proven in \cite{MaSchwedeSingularitiesMixedCharBCM}.
\begin{lemma}
\label{lem.PureQGorensteinPair}
Suppose $(R, \fram)$ is a $\bQ$-Gorenstein complete normal local ring with residual characteristic $p > 0$.  Suppose that $\Delta = 1/n\Div_R(g)$ is a $\bQ$-Cartier divisor and $B$ is a big-Cohen-Macaulay $R^+$-algebra (or at least containing a fixed $g^{1/n}$).  Then $(R, \Delta)$ is \BCMReg{B} if and only if the map $R \xrightarrow{\cdot g^{1/n}} B$ is pure.
\end{lemma}
\begin{proof}
The proof is essentially the same as \cite[Theorem 6.12]{MaSchwedeSingularitiesMixedCharBCM} and \cite[Proposition 6.14]{MaSchwedeSingularitiesMixedCharBCM}. Indeed, to modify those proofs write $K_R = {1 \over n}\Div(f)$ and $\Delta = {1 \over n} \Div(g)$.  Replace the map $\mu$ in the diagrams of those proofs by multiplication by $g^{1/n}$ followed by the map $\mu$.
\end{proof}

We now prove a crucial lemma which can be thought of as a special case of inversion of adjunction in higher codimension.

\begin{lemma}
\label{lem.GlobFRegExceptionalImpliesBCMReg}
Suppose $(R, \fram, k)$ is a complete normal $\bQ$-Gorenstein local domain of dimension $\geq 2$ with residual characteristic $p > 0$ and $R/\m$ $F$-finite.  Let $X = \Spec R$ and suppose that $\pi : Y \to X$ is the blowup of some $\fram$-primary ideal $I$ such that $Y$ is normal.  Further suppose that $I \cdot \O_Y = \O_Y(-mE)$ where $E$ is a prime exceptional divisor.
{
Let $\Delta$ be a $\Q$-divisor such that $K_X+\Delta$ is $\Q$-Cartier and} let $\Delta_E$ denote the different of $K_Y + E + \pi^{-1}_* \Delta$ along $E$.  Suppose that $(E, \Delta_E)$ is globally $F$-regular, then $(R, \Delta)$ is \BCMReg{} and in particular, $R$ itself is \BCMReg{}.
\end{lemma}
\begin{remark}
In the statement of \autoref{lem.GlobFRegExceptionalImpliesBCMReg}, suppose we write
\[
\pi^*(K_X + \Delta) = K_Y + \beta E + \pi^{-1}_* \Delta.
\]
Then $K_Y + E + \pi^{-1}_* \Delta = \pi^*(K_X + \Delta) + (1-\beta) E$.  Thus $-(K_E +\Delta_E) \sim_{\bQ} (\beta - 1)E|_E$.  Since global $F$-regularity forces $-(K_E + \Delta_E)$ to be big and we have that $-E$ is relatively ample, we thus see that $\beta < 1$.
\end{remark}
\begin{proof}

Let $T = \bigoplus_{n \in \bZ}H^0(Y, \cO_Y(-nE))t^n$ denote the  \emph{generalized extended} Rees algebra.  Notice that by replacing $m$ by a multiple $ml$ and $I$ by $\overline{I^l}$, we can assume that if we take the $m$th Veronese subring of $T$, then we obtain the usual extended Rees algebra $R[Is,s^{-1}]$ (which we can assume is normal, see \autoref{sec.ExtendedRees}). Also notice that $[T]_i = Rt^{i}$ when $i\leq 0$.

Notice that $T$ has a prime ideal $J_E$ corresponding to the prime exceptional divisor $E$, in fact it is easy to see by \autoref{subsec.DiscrepComputationsForExtendedRees} that $J_E = t^{-1}T$.  This ideal agrees with $T$ in negative degrees, and agrees with $\fram$ in degree zero since $\pi(E) = V(\fram)$.  On the other hand, $T$ has a maximal ideal $\frn = t^{-1}T + \fram T + T_{>0}$.  Let $\widehat{T}$ denote the completion of $T$ with respect to $\frn$.

\begin{claim}
\label{clm.RToReesAlgebraIsPure}
For every integer $c\geq 0$, the map $R \to \widehat{T}$ sending $1$ to $t^{-c}$ is pure.
\end{claim}
\begin{proof}
Since the map factors as $R\to T_{\n}\to \widehat{T}$ and $T_{\n}\to \widehat{T}$ is faithfully flat, it is enough to show that $R\to T_\n$ sending $1$ to $t^{-c}$ is pure. Since $R$ is a complete local domain, it is approximately Gorenstein and thus it is enough to show that $R/J\to T_\n/JT_\n$ sending $1$ to $t^{-c}$ is injective for every $\m$-primary ideal $J\subseteq R$ by \cite{HochsterCyclicPurity}. Now we pick $l\gg c$ such that $[T]_{l-c}\subseteq Jt^{l-c}$. Consider the ideal $J':=t^{-l}T+T_{\geq l}+JT$ of $T$. Clearly $J'$ is an $\n$-primary ideal of $T$ that contains $JT$. Thus in order to show $R/J\to T_\n/JT_\n$ sending $1$ to $t^{-c}$ is injective, it is enough to show that $R/J\to T_\n/J'T_\n\cong T/J'$ sending $1$ to $t^{-c}$ is injective. But since $[T]_{l-c}\subseteq J t^{l-c}$, it is easy to check that the degree $-c$ piece of $T/J'$ can be identified with $R/J$. Thus the map $R/J\to T/J'$ sending $1$ to $t^{-c}$ is split.
\end{proof}

We now fix $g \in R$ such that $n \Delta = \Div_X(g)$. We notice that $\pi^* \Delta = \pi^{-1}_* \Delta + \{\lambda \}E + \lfloor \lambda \rfloor E$.  Let $c = \lfloor \lambda \rfloor$.

\begin{claim}
\label{clm.ReesBCMRegularSoIsR}
Fix a map $R^+ \to \widehat{T}^+$.  If $\widehat{T} \xrightarrow{\cdot g^{1/n}t^c} C$ is pure for every perfectoid big Cohen--Macaulay $\widehat{T}^+$-algebra $C$  (e.g., if $(\widehat{T}, {1 \over n}\Div_{\widehat{T}}(gt^{cn}) + \Gamma)$ is \BCMReg{} for some $\Gamma \geq 0$), then $(R, \Delta)$ is \BCMReg{}.
\end{claim}
\begin{proof}[Proof of \autoref{clm.ReesBCMRegularSoIsR}]
Let $B$ be a perfectoid big Cohen--Macaulay $R^+$-algebra. By \autoref{thm:GeneralweaklyfunctorialBCM} we see that there exists a perfectoid big Cohen--Macaulay ${\widehat{T}}^+$-algebra $C$ such that the following diagram commutes:
\[
\xymatrix{
B \ar[r] & C \\
R \ar[r]^{\cdot t^{-c}} \ar[u]^{\cdot g^{1/n}} & \widehat{T}. \ar[u]_{\cdot g^{1/n}t^c}
}
\]
where the top row is a map of $R^+$-algebras. Notice that the bottom row $R \to T \to \widehat{T}$ is pure by \autoref{clm.RToReesAlgebraIsPure}.  Therefore since $\widehat{T} \xrightarrow{\cdot g^{1/n}t^c} C$ is pure, we also have that $R \xrightarrow{\cdot g^{1/n}} C$ is pure.  This implies that $R \xrightarrow{\cdot g^{1/n}} B$ is also pure.  Therefore by \autoref{lem.PureQGorensteinPair} $(R, \Delta)$ is \BCMReg{B}.  As this is true for every $B$, $(R, \Delta)$ is \BCMReg{}.
\end{proof}

Continuing the proof of \autoref{lem.GlobFRegExceptionalImpliesBCMReg}, 
it suffices to show that $(\widehat{T}, {1\over n}\Div_{\widehat{T}}(gt^{cn}))$ satisfies the condition of \autoref{clm.ReesBCMRegularSoIsR}.
%
We let $E_{\widehat T}$ and $(\pi^{-1}_* \Delta)_{\widehat{T}}$ correspond to $E$ and $\pi_*^{-1} \Delta$ respectively via \autoref{subsec.DiscrepComputationsForExtendedRees}.  Notice also that
\[
(\pi^{-1}_* \Delta)_{\widehat{T}} + \{\lambda\}E_{\widehat{T}} = {1\over n}\Div_{\widehat{T}}(gt^{cn})
\]

Let $T'$ denote the $m$th Veronese subalgebra of $T$ and let $J_{T'}$ denote the ideal of $E$ in $T'$.  We consider $T' \to T$ with induced map $\kappa : \Spec T \to \Spec T'$.  We know by \autoref{lem.HomTT'} that there is a $\Psi : T \to T'$ which generates $\Hom_{T'}(T, T')$ as a $T$-module and which sends the homogeneous maximal ideal of $T$ to the homogeneous maximal ideal of $T'$.  Following the argument of \cite[Lemma 5.1]{CRMPSTCoversofRDP}, we see that the completion of $(T', \Theta)$ is \BCMReg{} if and only if the completion of $(T, \kappa^*(\Theta))$ is \BCMReg{}.  In particular, by \autoref{lem.PullbackOfDOnExtendedRees}, it suffices to show that
\[
(\widehat{T'}, (\pi^{-1}_* \Delta)_{\widehat{T'}} + \{\lambda\}E_{\widehat{T'}})
\]
is \BCMReg{}.  Hence by \autoref{cor.PBCMImpliesBCM} it suffices to show that $(\widehat{T'}, (\pi^{-1}_* \Delta)_{\widehat{T'}} + E_{\widehat{T'}})$ is purely \BCMReg{}.

Let $J'_E$ denote the ideal defining $E_{T'}$ on $T'$.  The normalization $S^{\mathrm{N}}$ of $S = T'/J'_E$ is the section ring of $E$ with respect to the very ample divisor $-mE|_E$.  By \cite[Proposition 5.3]{SchwedeSmithLogFanoVsGloballyFRegular}, the pair $(S^{\mathrm{N}}, (\Delta_E)_{S^{\mathrm{N}}})$ is strongly $F$-regular and hence the completion is \BCMReg{} by \cite[Corollary 6.23]{MaSchwedeSingularitiesMixedCharBCM}.
\begin{claim}
The divisor $(\Delta_E)_{S^{\mathrm{N}}}$ agrees with the different of $K_{T'} + (\pi^{-1}_* \Delta)_{{T'}} + E_{T'}$ along the normalization of $E_{T'}$.
\end{claim}
\begin{proof}
We are essentially computing the different in two different ways.  Note that $\Delta_E$ can be computed by considering the {rational sheaf map}
\[
\cO_Y(n(K_Y + \pi^{-1}_* \Delta + E)) \dashrightarrow \cO_E(nK_E),
\]
noting the left side is a line bundle for $n$ sufficiently divisible.
{
After tensoring this map by $\cO_Y(-n(K_Y + \pi^{-1}_* \Delta + E))$, we let $-D_E$ be the divisor corresponding to the image of $1$}. Then the different of $K_Y + \pi^{-1}_* \Delta + E$ along $E$ is simply ${1\over n}D_E$.  But we can apply our $\Gamma_{*, T'}$  functor (with respect to $-mE$) to that map of sheaves, and so obtain the map:
\[
T'(n(K_{T'} + (\pi^{-1}_* \Delta)_{{T'}} + E_{T'})) \dashrightarrow S^{\mathrm{N}}(nK_{S^{\mathrm{N}}}).
\]
Since the left side is (abstractly) isomorphic to $T'$, it also selects an effective divisor $D_S$ with ${1 \over n} D_S$ the different of $K_{T'} + (\pi^{-1}_* \Delta)_{{T'}} + E_{T'}$ along $E_{T'}$.  Note that $(D_E)_{S^{\mathrm{N}}} = D_{S}$ by construction.  The claim follows.
\end{proof}
But now by \autoref{cor:universal_singularities} we have that $(\widehat{T'}, E_{\widehat{T'}} + (\pi^{-1}_* \Delta)_{\widehat{T'}})$ is purely \BCMReg{}.  This completes the proof of \autoref{lem.GlobFRegExceptionalImpliesBCMReg}.
\end{proof}

\begin{remark}
The proof above shows that the map
$\widehat{T'} \to S^{\mathrm{N}}$
is surjective, since if $(\widehat{T'}, E_{T'} + \Delta_{T'})$ is purely \BCMReg{}, then $E_{T'}$ is normal by \autoref{cor:normal_centers}.  Taking graded pieces, this means that
\[
H^0(X, \cO_Y(-lmE)) \to H^0(E, \cO_E \otimes \cO_Y(-lmE))
\]
surjects for all $l$.  In particular, it surjects for $l = 0$.
\end{remark}

\begin{example}
Consider the ring $R = \bZ_p\llbracket x_2, \dots, x_d \rrbracket/(p^n + x_2^n + \dots + x_d^n)$ for $n < d$.  We blowup the origin and obtain $Y \to X = \Spec R$.  The exceptional divisor $E$ is isomorphic to $\Proj \bF_p[x_1, \dots, x_d]/(x_1^n + x_2^n + \dots + x_d^n)$.  We also have $K_Y = (d -1 - n)E$ which since $K_X = 0$, we have $\pi^*K_X = K_Y + \beta E$ with $\beta = (n-d+1) < 1$.  Thus $R$ is \BCMReg{} whenever $E$ is globally $F$-regular.

Note that in this case we can also argue directly as follows using only \cite{MaSchwedeSingularitiesMixedCharBCM}.  Indeed, consider the completion of the extended Rees algebra: $T=\widehat{R[\m t, t^{-1}]}$, where the completion is at $(t^{-1})+\m t$. If $T/t^{-1}T\cong \widehat{gr_\m R}$ is Gorenstein and $F$-rational (which happens precisely when $E$ is globally $F$-regular), then $T$ is Gorenstein and \BCMRat{} by \cite[Proposition 3.4 and 3.5]{MaSchwedeSingularitiesMixedCharBCM}, hence \BCMReg{} by \cite[Corollary 6.15]{MaSchwedeSingularitiesMixedCharBCM}. This implies $R$ is \BCMReg{} since $R\to T$ is pure by \autoref{clm.RToReesAlgebraIsPure}.
\end{example}

\begin{example}
Consider $R = \bZ_p\llbracket y \rrbracket$ and let $D = \Div(py(y-p))$.  Then one can define the \emph{\BCMReg{}-threshold} of $(R, tD)$, denoted $c(R, tD)$, to be the supremum of $t$ such that $(R, tD)$ is \BCMReg{}.  If we replace $\bZ_p$ by $\bF_p[x]$ and $D$ by $\Div(xy(y-x))$, then this is simply the $F$-pure threshold, which is well known to be $2 \over 3$ if $p \equiv_3 1$ and ${2 \over 3} - {1 \over 3p}$ if $p \equiv_3 2$.  Therefore by switching between a graded ring and the Proj as in \cite{SchwedeSmithLogFanoVsGloballyFRegular}, we see $(\bP^1, t P_1 + tP_2 + tP_3)$ is globally $F$-regular for $t < {2 \over 3}$ if $p \equiv_3 1$ and $t < {2 \over 3} - {1 \over 3p}$ if $p \equiv_3 2$.  But if we blowup the origin in $\Spec R$, we have the same $\bP^1$ and same boundary divisors restricted to the boundary divisor.  Thus we also obtain that $(R, tD)$ is \BCMReg{} for the same values of $t$.  Since $t = 2/3$ is certainly an upper limit on the \BCMReg{}-threshold (since \BCMReg{} implies KLT), we therefore obtain that
\[
c(R, tD) = {2 \over 3} \text{ if $p \equiv_3 1$}
\]
and that
\[
{2 \over 3} - {1 \over 3p} \leq c(R, tD) \leq {2 \over 3} \text{ if $p \equiv_3 2$}.
\]
\end{example}

\begin{example}
Consider the ring $R = \bZ_p\llbracket y \rrbracket$ and let $D = \Div(y^2 - p^3)$ (nothing essential changes if let $D = \Div(p^2 - y^3)$).  The log resolution of singularities of $(R, D)$ is obtained in the same way it is in equal characteristic.  We again compute the \BCMReg{}-threshold $c(R, D)$.  Let $\pi : Y \to X = \Spec R$ denote the blowup of $(y^2, yp^2, p^3)$.  It is not difficult to check that this blowup has a single prime exceptional divisor $E \cong \bP^1_k$.  There are two singular points on this blowup, and the strict transform $\widetilde{D}$ of $D$ is a Cartier divisor that meets $E$ transversally.  By understanding this blowup in view of a log resolution of $(R, D)$, or by doing direct computations on local charts, one can see that the different of $K_Y + E + t \widetilde{D}$ along $E$ is ${1 \over 2}P_1 + {2 \over 3}P_2 + t P_3$.  Straightforward computation then gives that if
\[
c = c(R, tD)
\]
is the \BCMReg{}-threshold of $(R, tD)$, then we have by \autoref{lem.GlobFRegExceptionalImpliesBCMReg} that
\[
\begin{array}{ccccc}
{1 \over 2} & \leq & c & \leq 5/6 & \text{ if $p = 2$.}\\
{2 \over 3} & \leq & c & \leq 5/6 & \text{ if $p = 3$.}\\
{5 \over 6} - {1 \over 6p} & \leq & c & \leq 5/6 & \text{ if $p \equiv_6 5$.}\\
{5 \over 6} & = & c & & \text{ if $p \equiv_6 1$}.
\end{array}
\]
\end{example}

We can generalize these examples and obtain the following generalization of a result of \cite{CRMPSTCoversofRDP}.  In particular, it follows from their work that if $R$ is 2-dimensional KLT with perfect residue field and that the index of $K_R$ is not divisible by $p > 0$, then $R$ is \BCM{}-regular.  We remove the index not divisible by $p$ hypothesis.  We also add a boundary divisor with standard coefficients.

\begin{theorem}
\label{thm.KLTSurfaceImpliesBCMReg}
	Let $(R, \fram, k)$ be a normal $2$-dimensional excellent local ring with a dualizing complex, and  
	with $F$-finite residue field of characteristic $p > 5$.
If $(R, \Delta)$ is KLT and $\Delta$ has standard coefficients, then $(\widehat{R}, \widehat{\Delta})$ is \BCMReg{}.
\end{theorem}
\begin{proof}
Note that $R$ is $\mathbb{Q}$-factorial by \cite{LipmanRationalSingularities} or \cite[Corollary 4.11]{tanaka_mmp_excellent_surfaces} {(note the latter's results hold when the regular base is replaced by a base that is excellent, finite-dimensional, and has a dualizing complex)}. Set $X = \Spec R$. By the same argument as in \cite[Proposition 2.13]{CasciniGongyoSchwedeUniformBounds} (cf.\ \cite[Lemma 5.6]{SatoTakagiGeneralHyperplaneSections}), we can construct a projective morphism $\pi \colon Y \to X$ such that $\pi$ is an isomorphism over $X \setminus \{\fram\}$, the exceptional divisor $E$ is irreducible, $-(K_Y+E+\Delta_Y)$ is ample where $\Delta_Y$ is the strict transform of $\Delta$, and $(Y,E+\Delta_Y)$ is purely log terminal. The proof uses log resolutions of singularities which are valid for quasi-excellent Noetherian rings and the Minimal Model Program for quasi-projective surfaces over $S$ (see \cite{tanaka_mmp_excellent_surfaces}). By adjunction, write
\[
K_E + \Delta_E = (K_Y + E + \Delta_Y)|_E.
\]
Then $\Delta_E$ has standard coefficients and $-(K_E+\Delta_E)$ is ample. Thus, $(E,\Delta_E)$ is globally F-regular by \cite[Proposition 5.5]{SatoTakagiGeneralHyperplaneSections}. Therefore $(\widehat{R}, \widehat{\Delta})$ satisfies the hypotheses of \autoref{lem.GlobFRegExceptionalImpliesBCMReg} and hence $(\widehat{R}, \widehat{\Delta})$ is \BCMReg{}.
\end{proof}

By taking canonical covers, we obtain the following generalization of \cite[Theorem A]{CRMPSTCoversofRDP}, which handled the Gorenstein case.

\begin{corollary}
Let $(R, \fram, k = \overline{k})$ be a local 2-dimensional KLT singularity of mixed characteristic $(0,p > 5)$ that is essentially of finite type over an excellent DVR $S$. Then there exists a finite split extension $\widehat{R} \subseteq S$ to a regular ring.
\end{corollary}
\begin{proof}
By \autoref{thm.KLTSurfaceImpliesBCMReg} we know that $\widehat{R}$ is \BCMReg{}. Since $\widehat{R}$ is KLT it is $\bQ$-Gorenstein and we may take a finite canonical cover $\widehat{R} \subseteq R'$.  Now by \cite[Lemma 5.1]{CRMPSTCoversofRDP}, $R'$ is \BCMReg{} and it is Gorenstein by construction.  Now applying \cite[Theorem A]{CRMPSTCoversofRDP} to $R'$ completes the proof.
\end{proof}

\autoref{thm.KLTSurfaceImpliesBCMReg} also implies that we have the KLT to PLT inversion of adjunction in dimension 3 if the residual characteristic is $> 5$.

\begin{corollary}
\label{cor.Dimension2To3InversionOfAdjunction}
Suppose that $(R, \fram, k)$ is a normal 3-dimensional {excellent} local  ring {with a dualizing complex},
and with $F$-finite residue field of characteristic $p > 5$.  Set $X = \Spec R$.
Suppose that $D$ is a prime divisor on $X$ and that $\Delta \geq 0$ is a $\bQ$-divisor such that $K_X + D + \Delta$ is $\bQ$-Cartier {and $\Delta$ has standard coefficients}.  Suppose that $(D^{\mathrm{N}}, \diff_{D^N}(\Delta+D))$ is KLT, then $(\widehat{R}, \widehat{D}+\widehat{\Delta})$ is purely \BCMReg{}.  In particular, $(X, D+\Delta)$ is PLT and {$D$ is normal}.
\end{corollary}
\begin{proof}
We see that $(\widehat{D^{\mathrm{N}}}, \widehat{\diff}_{D^N}(\Delta+D))$ is \BCMReg{} by \autoref{thm.KLTSurfaceImpliesBCMReg}.  Thus by \autoref{cor:singularities} we see that $(\widehat{R}, \widehat{D}+\widehat{\Delta})$ is purely \BCMReg{}.  Finally, as being PLT is unaffected by passing to completion, we see that $(X, D+\Delta)$ is PLT by \autoref{thm:birational_adjoint_comparison}. Notice that here resolutions of singularities exist by \cite{CP19}.
\end{proof}

Finally, we also observe that the main results of \cite{CasciniGongyoSchwedeUniformBounds} essentially holds in mixed characteristic, replacing strongly $F$-regular with \BCMReg{}.  In particular:

\begin{theorem}
\label{thm.CasciniGongyoSchwedeUniformBounds}
Following the notation of \cite{CasciniGongyoSchwedeUniformBounds}, suppose $I\subseteq (0,1)\cap \bQ$ ia a finite set and let $\Gamma=D(I)$.
Then there exists a positive constant $p_0$, depending only on $I$, such that if
\begin{itemize}
\item{}  $(R,\Delta)$ is a KLT pair,
\item{} $R$ is a mixed characteristic 2-dimensional complete local ring,
\item{} $R$ has {$F$-finite}
residue field of characteristic $p > p_0$,
\item{} and the coefficients of $\Delta$ belong to $\Gamma$,
\end{itemize}
then $(R,\Delta)$ is \BCMReg{}.
\end{theorem}
\begin{proof}
This is essentially the same as \cite{CasciniGongyoSchwedeUniformBounds}.  The bounds $p_0$ were constructed by analyzing global $F$-regularity of $(\bP^1, \Delta)$ or equivalently the strong $F$-regularity of a line arrangement on $\bA^2$, see \cite[Corollary 3.13]{CasciniGongyoSchwedeUniformBounds}.  {The corresponding arguments for the $F$-finite residue field case were completed in \cite[Proposition 5.5]{SatoTakagiGeneralHyperplaneSections}.}
Thus, as in the proof of \autoref{thm.KLTSurfaceImpliesBCMReg}, we may construct a projective morphism $\pi \colon Y \to X$ such that $\pi$ is an isomorphism over $X \setminus \fram$, the exceptional divisor $E$ is irreducible, $-(K_Y+E+\Delta_Y)$ is ample where $\Delta_Y$ is the strict transform of $\Delta$, and $(Y,E+\Delta_Y)$ is purely log terminal {then apply the arguments of \cite{CasciniGongyoSchwedeUniformBounds, SatoTakagiGeneralHyperplaneSections}}. 
\end{proof}

\section{Application to $F$-singularities in families}
\label{sec:reduction_mod_p}

In this section, we generalize the results of \cite[Section 7]{MaSchwedeSingularitiesMixedCharBCM} from the BCM test ideal to the adjoint-like ideal we have defined earlier in this paper.

\begin{setting}\label{set:reduction}
	Let $A$ be a Dedekind domain which is a localization of a finite extension of $\bZ$, with fraction field $K$.  Let $f \colon X\to U:=\Spec(A)$ be a flat family  essentially of finite type and let $\Delta$ be an effective  $\bQ$-divisor with no vertical components.  For $p\in U$, denote the fiber over $p$ by $(X_p,\Delta_p)$, that is to say $\Delta_p=\Delta|_{X_p}$.  In some situations $X$ will be affine: $X=\Spec(R)$.
	\end{setting}

First we give a strengthening of \cite[Theorem 7.9]{MaSchwedeSingularitiesMixedCharBCM} using our new adjunction.  The difference is that we no longer need to assume that the index is not divisible by the characteristic.  Note that if $f \colon X \to U$ is not proper, then we cannot expect strong $F$-regularity of $X_p$ for some $p \in U$ to imply that $X$ is \BCMReg{} over a neighborhood of $p \in U$ (see \cite[Remark 7.7]{MaSchwedeSingularitiesMixedCharBCM}).
What we show is that this is true at the closed points lying along a horizontal point of $X$ which intersects $X_p$.

\begin{proposition}\label{prop:non_empty_open}
	Suppose $X=\Spec(R)$ is affine and  of finite type and flat over $U$, as in \autoref{set:reduction}, and let $\Delta\geq 0$ be a $\bQ$-divisor such that $K_R+\Delta$ is $\bQ$-Cartier.  Choose a prime ideal $Q\subseteq R$ such that $Q\cap A=(0)$.  For each $t\in A$ let $\sqrt{tR+Q}=\cap^{n_t}_{i=1}\frak q_{t,i}$ be a decomposition into minimal primes.  	
 Let
		$$W=\{t\in \m\text{-}\mathrm{Spec} A \mid \{\mathfrak{q}_{t,i}\}_{i=1}^{n_t} \mathrm{\ is\ nonempty\ and \ } (R/t,\Delta_t)_{\mathfrak q_{t,i}}\mathrm{\ is \ strongly\ }F\text{-}\mathrm{regular\ for \ all\ }\mathfrak{q}_{t,i}\} $$
		Then $W$ is open in $\fram$-$\Spec A$, and $(\widehat{R_{\frak q_{t,i}}},\widehat{\Delta}_{\frak q_{t,i}})$ is \BCMReg{} for all $t\in W$ and all $i$.
\end{proposition}
\begin{proof}
	 Suppose that $\{\frak q_{p,i}\}_{i=1}^{n_p}$ is nonempty for some $p\in\m$-$\Spec(A)$, and that $(R/p,\Delta_p)$ is strongly $F$-regular at all $\frak q_{p,i}$. Note that this implies $\Div(p)$ is a prime Cartier divisor after localizing $R$ at $\frak q_{p,i}$. 
In particular, we have $\diff_{\Div(p)}(\Div(p)+\Delta)=\Delta_p$ after localizing and taking completion along $\frak q_{p,i}$.

By \cite[Theorem 4.1.1]{AndreWeaklyFunctorialBigCM}, for any given perfectoid big Cohen--Macaulay $(\widehat{R_{\frak q_{p,i}}})^+$-algebra $B$ there is a compatibly chosen perfectoid big Cohen--Macaulay $\widehat{(R/p)}_{\frak q_{p,i}}^+$-algebra $C$.  By \cite[Corollary 6.23]{MaSchwedeSingularitiesMixedCharBCM}, $((\widehat{R/p})_{\mathfrak q_{p,i}},(\widehat{\Delta_p})_{\mathfrak q_{p,i}})$ is \BCMReg{C}. Now we apply \autoref{thm.AdjointMultiplierInversionOfAdjunction} to obtain that $(\widehat{R_{\frak q_{p,i}}}, \Div(p)+\widehat{\Delta}_{\frak q_{p,i}})$ is purely $\mathrm{BCM}_{B\shortrightarrow C}$-regular.  Accordingly this implies that $(\widehat{R_{\frak q_{p,i}}}, \widehat{\Delta}_{\frak q_{p,i}})$ is $\mathrm{BCM}_B$-regular by \autoref{prop.AdjIsSmallerThanTau}.  In particular this implies that $(\widehat{R_{\frak q_{p,i}}}, \widehat{\Delta}_{\frak q_{p,i}})$ is KLT by \autoref{thm:test_comparison}, and thus $(R, \Delta)_{\frak q_{p,i}}$ is KLT since being KLT is unaffected up to completion. But then $(R,\Delta)_Q$ is KLT since KLT is preserved under localization (even without resolution of singularities).  Finally, a KLT singularity is of strongly $F$-regular type by \cite{TakagiInterpretationOfMultiplierIdeals}, so we obtain the required open subset.  Now the above argument applies to all $p$ in this open subset and all given $B$, which proves the final statement.
\end{proof}

\begin{proposition}\label{prop:SFR_on_fiber}
	Let $\phi\colon X\to U$ be a proper flat family and $\Delta\geq 0$ a $\bQ$-divisor such that $K_X+\Delta$ is $\bQ$-Cartier.  Suppose $(X_p,\Delta_p)$ is strongly $F$-regular for some point $p\in U$.  Then $(X_K,\Delta_K)$ is KLT, and there exists a non-empty open subset $V$ of $U$ such that the  closed fibers of $(X,\Delta)$ over $V$ are strongly $F$-regular. Furthermore, the completion of $(X,\Delta)_{\frak q}$ is \BCMReg{} at all points $\frak q \in X$ which are vertical over $V$. 
	\end{proposition}
\begin{proof}
Fix a point $Q\in X$ which is surjective to $U$. Then $V(Q)\cap X_{p}$ is non-empty, since $\phi$ is proper.  Therefore at every point $\frak q\in V(Q)\cap X_{p}$, $(X_p,\Delta_p)_{\frak q}$  is strongly $F$-regular. It follows from the same argument as in \autoref{prop:non_empty_open} that $(X,\Delta)_K$ is KLT at $Q_K$, and therefore by varying $Q$ we see that $(X,\Delta)_K$ is KLT everywhere.  Hence it is of strongly $F$-regular type by \cite{TakagiInterpretationOfMultiplierIdeals}. The final claim follows from the same argument as in \autoref{prop:non_empty_open}.
%
\end{proof}

We obtain analogous results for pure BCM-regularity.
\begin{proposition}\label{thm:plt_in_familes_locally}
	Suppose $X=\Spec(R)$ is affine and of finite type and flat over $U$, as in \autoref{set:reduction}, and choose a prime ideal $Q\subseteq R$ such that $Q\cap A=(0)$.  For each $t\in A$ write $\sqrt{tR+Q}=\cap^{n_t}_{i=1}\frak q_{t,i}$ a decomposition into minimal primes.   Suppose $D$ is an integral Weil divisor which is horizontal over $U$, and $\Delta\geq 0$ a $\bQ$-divisor such that $K_R+D+\Delta$ is $\bQ$-Cartier, then the following set
		$$W=\{t\in\mathfrak m\text{-}\mathrm{Spec} A \mid \{\mathfrak{q}_{t,i}\}_{i=1}^{n_t} \mathrm{\ is\ nonempty\ and \ } (R/t,D_t+\Delta_t)_{\mathfrak q_{t,i}}\mathrm{\ is \ purely\ }F\text{-}\mathrm{regular\ for \ all\ }\mathfrak{q}_{t,i}\} $$
		is open in $\frak{m}$-$\Spec(A)$. Furthermore,
		$(\widehat{R_{\frak q_{\tilde{t},i}}},\widehat{D}_{\frak q_{\tilde{t},i}}+\widehat{\Delta}_{\frak q_{\tilde{t},i}})$
		is purely \BCMReg{} for all $t\in W$ and all $i$.
\end{proposition}

\begin{proof}
	Suppose  $t\in \fram$-$\Spec(A)$ is such that $\{\mathfrak{q}_{t,i}\}_{i=1}^{n_t}$ is non-empty  and $(R/t,D_t+\Delta_t)_{\frak q_{t,i}}$ is purely $F$-regular for all $i$, which in particular implies that $\Div(t)$ is a prime Cartier divisor near all $\frak q_{t,i}$.
For each $\frak q_{t,i}\notin \Supp(D)$ the statement follows from  \autoref{prop:non_empty_open}.  So we may assume that $\frak q_{t,i}\in\Supp(D)$ for some $i$.  Let $B$ and $C$ be a compatible choice of big Cohen--Macaulay $\widehat{R_{\frak q_{t,i}}}^+$ and $(\widehat{R_D})_{\frak q_{t,i}}^+$- algebras respectively.
	
	By assumption, $(R/t,D_t+\Delta_t)$ is purely $F$-regular at $\frak q_{t,i}$, and so $X_t=\Spec(R/t)$ is a prime Cartier divisor near $\frak q_{t,i}$, and so is $D_t=\Spec(R_D/t)$ near $\frak q_{t,i}$.  It follows that near $\frak q_{t,i}$, we have $\diff_{X_t}(X_t+D+\Delta)=(D+\Delta)|_{X_t}=D_t+\Delta_t$, and $\diff_{D_t}(\diff_D(D+\Delta)+D_t)=\diff_D(D+\Delta)|_{D_t}$.  So our assumption becomes that $(R/t,\diff_{X_t}(X_t+D+\Delta))$ is purely $F$-regular at $\frak q_{t,i}$.
	
	Therefore at $\frak q_{t,i}$, we have \[((R/t)_D,\diff_{D_t}(\diff_{X_t}(X_t+\Delta+D)))=((R/t)_D,\diff_{D_t}(\diff_{D}(X_t+D+\Delta)))\]\[=((R/t)_D,\diff_D(D+\Delta)|_{D_t})\] is strongly $F$-regular.  Now \autoref{prop:non_empty_open} provides  an open subset $V$ of $U$ such that for any choice of $\frak q_{\tilde{t},i}$ which is vertical over a point $\tilde{t}\in V$, the completion of $(R_D,\diff_D(D+\Delta))_{\frak q_{\tilde{t},i}}$ is \BCMReg{}, and its reduction mod $\tilde{t}$ is strongly $F$-regular. By inversion of adjunction \autoref{thm.AdjointMultiplierInversionOfAdjunction}, we know that $(\widehat{R_{\frak q_{\tilde{t},i}}},\widehat{D}_{\frak q_{\tilde{t},i}}+\widehat{\Delta}_{\frak q_{\tilde{t},i}})$ is purely \BCMReg{}. Also by inversion of $F$-adjunction (see \cite[Theorem A]{Das_different_different_different}, whose proof only uses that the ground field is $F$-finite), we also know that $(R/{\tilde{t}},D_{\tilde{t}}+\Delta_{\tilde{t}})_{\frak q_{\tilde{t},i}}$ is purely $F$-regular.
\end{proof}

\begin{proposition}\label{thm:reduction_mod_p_plt}
	Let $\phi \colon X\to U$ be a proper flat family and let $D$ be an integral Weil divisor such that $K_X+D+\Delta$ is $\bQ$-Cartier, and $D$ is horizontal over $U$.  Suppose $(X_p,D_p+ \Delta_p)$ is purely $F$-regular for some $p\in U$.  Then $(X_K,D_K + \Delta_K)$ is PLT and there exists a non-empty open subset $V$ of $U$ such that the closed fibers of $(X,D+\Delta)$ over $V$ are purely $F$-regular. Furthermore, the completion of $(X,D+\Delta)_{\frak q}$ is purely \BCMReg{} at all points $\frak q\in X$ which are vertical over $V$.
\end{proposition}
\begin{proof}
	Fix a point $Q\in X$ which is surjective to $U$.  Then	$V(Q)\cap X_{p}$ is non-empty since $\phi$ is proper.  Therefore there is a point in this intersection at which $(X_p,D_p+\Delta_p)$  is purely $F$-regular. It follows from  \autoref{thm:plt_in_familes_locally} that there is a non-empty open $V$ such that for any $\frak q$ vertical over $V$, the completion of $(X,D+\Delta)_{\frak q}$ is purely \BCMReg{}. Therefore by \autoref{thm:birational_adjoint_comparison} and the fact that being PLT is unaffected up to completion, $(X,D + \Delta)$ is PLT at these points, and hence is PLT at $Q$. By varying $Q$ we obtain that $(X_K,D_K+\Delta_K)$ is PLT everywhere. Hence this pair is of purely $F$-regular type by \cite[Corollary 5.4]{TakagiPLTAdjoint}. The final claim follows from the same argument as in \autoref{thm:plt_in_familes_locally}.
\end{proof}

\appendix

\section{Perfectoid big Cohen-Macaulay algebras}
\label{sec.perfectoidBCM}

In this appendix, we use techniques of Andr\'{e} \cite{AndreWeaklyFunctorialBigCM} together with the ``perfections in mixed characteristic" of Bhatt--Scholze \cite[Sections 7 and 8]{BhattScholzepPrismaticCohomology} to prove:
\begin{enumerate}
  \item The existence of weakly functorial perfectoid big Cohen-Macaulay algebras of certain {\it diagrams}: \autoref{thm:weaklyfunctorialdiagram}.
  \item A domination result for certain {\it maps} of perfectoid big Cohen-Macaulay algebras: \autoref{thm:dominateMap}.
\end{enumerate}

We believe these results can also be proved carefully using techniques from \cite{AndreWeaklyFunctorialBigCM} alone (i.e., without referring to \cite{BhattScholzepPrismaticCohomology}). However, we think that the treatment using perfections in mixed characteristic is the simplest in terms of presentation in many situations. Along the way, we also obtain slight generalizations and somewhat neater proofs of results in \cite{AndreWeaklyFunctorialBigCM} and \cite{MaSchwedeSingularitiesMixedCharBCM} which we believe are of independent interest.

Throughout this appendix, we freely use some language on perfectoid rings as in \cite{BhattMorrowScholzeIHES,BhattScholzepPrismaticCohomology}. A ring $S$ is {\it perfectoid} if and only if it is $\pi$-adically complete for some element $\pi\in S$ such that $\pi^p$ divides $p$, the Frobenius on $S/pS$ is surjective, and the kernel of Fontaine's map $\theta$: $W(S^\flat)\to S$ is principal.\footnote{We refer to \cite[Section 3]{BhattMorrowScholzeIHES} for detailed definition of $\theta$: essentially, this is the unique map lifting the natural surjection $S^\flat\to S/p$.} We point out that if $S$ is $p$-torsion free and $\pi=p^{1/p}$, then this is the same as \cite[2.2]{AndreWeaklyFunctorialBigCM} (or \cite[Definition 2.2]{MaSchwedeSingularitiesMixedCharBCM}) by \cite[Lemma 3.10]{BhattMorrowScholzeIHES}. The advantage of this definition is that if $S$ has characteristic $p>0$, then a perfectoid ring is the same as a perfect ring, see \cite[Example 3.15]{BhattMorrowScholzeIHES}. Moreover, there is an equivalence of categories between perfectoid rings and the category of {\it perfect prisms}, see \cite[Theorem 3.10]{BhattScholzepPrismaticCohomology}. We will repeatedly use that perfectoid rings are reduced.

Further, note that if $I$ is an ideal of a perfectoid ring $R$, then $R/I$ need not be $p$-adically separated. We define $I^{-} = \bigcap_n (I+p^n)$, which is the closure of $I$ in the $p$-adic topology. Then the $p$-adic completion of $R/I$ is isomorphic to $R/I^{-}$. If $\{f_i\}_{i=1}^n$ is a sequence of elements in $R$, each admits a compatible system of $p$-power roots $\{f_i^{1/p^\infty}\}$, and $I=(f_1^{1/p^\infty},\dots,f_n^{1/p^\infty})$, then $R/I^-$ is perfectoid.

We next recall a definition from \cite{MaSchwedeSingularitiesMixedCharBCM} which we will use throughout.

\begin{definition}
Let $(R,\m)$ be a complete local domain such that $R/\m$ has characteristic $p>0$. An $R$-algebra $S$ is called a {\it perfectoid seed} if $S$ is perfectoid and it maps to a perfectoid big Cohen-Macaulay $R$-algebra.
\end{definition}

Now if $R$ is a perfectoid ring and $S$ is a derived $p$-adic complete $R$-algebra, then \cite{BhattScholzepPrismaticCohomology} defined the {\it perfectoidization} $S_{\perfd}$ of $S$ using (derived) prismatic cohomology. In general $S_{\perfd}$ only lives in $D^{\geq0}(R)$, but it is an honest perfectoid ring in all the cases that we consider. We will not give the precise definition here but we point out the following facts:
\begin{enumerate}[(1)]
  \item In characteristic $p>0$, $S_{\perfd}$ is the usual perfection $\varinjlim_eF^e_*S$ \cite[Example 8.3]{BhattScholzepPrismaticCohomology}.
  \item If $S$ is a derived $p$-complete quotient of $R$ (e.g., $S=R/J$ for a finitely generated ideal $J\subseteq R$), then $S_{\perfd}$ is a perfectoid ring and is a quotient of $S$ \cite[Theorem 7.4]{BhattScholzepPrismaticCohomology}.
  \item If $R\to S$ is the $p$-adic completion of an integral map, then $S_{\perfd}$ is a perfectoid ring \cite[Theorem 10.11]{BhattScholzepPrismaticCohomology}.
  \item $S_{\perfd}$ can be characterized as the derived limit of $R'$ over all maps from $S$ to perfectoid rings $R'$, and it does not depend on the choice of $R$ \cite[Proposition 8.5]{BhattScholzepPrismaticCohomology}. In particular, if $S_{\perfd}$ is a perfectoid ring then $S\to S_{\perfd}$ is the universal map to a perfectoid ring.
\end{enumerate}

As a consequence of $(2)$, for any finitely generated ideal $J\subseteq R$, we can define an ideal $J_{\perfd}=\ker(R\to (R/J)_{\perfd})$. It turns out that we have a well-behaved almost mathematics theory with respect to $J_{\perfd}$, see \cite[Section 10]{BhattScholzepPrismaticCohomology} (the essential point that lurks behind this is Andr\'{e}'s flatness lemma, see \cite[Theorem 7.12 and Theorem 7.4]{BhattScholzepPrismaticCohomology}).

Now suppose $(A,\m)$ is a Noetherian complete local domain of residue characteristic $p>0$ and $S$ is a perfectoid $A$-algebra. Let $x_1,\dots, x_d$ be a system of parameters of $A$ and $0\neq g\in A$. We say that $S$ is $(g)_{\perfd}$-almost big Cohen-Macaulay with respect to $x_1,\dots,x_d$ (note that $(g)_{\perfd}$ makes sense in $S$ by the above discussion) if:
\begin{itemize}
 \item $x_1,\dots,x_d$ is a $(g)_{\perfd}$-almost regular sequence on $S$, i.e., the ideal $(g)_{\perfd}$ annihilates $\frac{(x_1,\dots, x_i)S:_S x_{i+1}}{(x_1,\dots, x_i)S}$ for each $i$.
 \item $S/(x_1,\dots,x_d)S$ is not $(g)_{\perfd}$-almost zero, i.e., $(g)_{\perfd}\nsubseteq (x_1,\dots,x_d)S$.
\end{itemize}
One fact that we will be using repeatedly is that if $S$ is a perfectoid $A$-algebra that is $(g)_{\perfd}$-almost big Cohen-Macaulay with respect to $x_1, x_2,\dots,x_d$, and if $A$ has mixed characteristic we assume $x_1=p$,\footnote{This assumption is actually not necessary, one can first enlarge $S$ to $S'$ using \cite[Theorem 7.12]{BhattScholzepPrismaticCohomology} such that $g$ admits a compatible system of $p$-power roots so that $(g)_{\perfd}=(g^{1/p^\infty})^-$, and then map $S'$ to a perfectoid big Cohen-Macaulay algebra using Gabber's method \cite[17.5]{GabberRameroFoundationsAlmostRingTheory}. We omit the details since in our context, we can always arrange $x_1=p$ in mixed characteristic.} then $S$ is a perfectoid seed. The proof of this fact follows the same line as in \cite[Proof of Theorem 3.1.1 (1)]{AndreWeaklyFunctorialBigCM}: the point is to first apply the flatness lemma (see \cite[Theorem 7.12]{BhattScholzepPrismaticCohomology}) to assume $x_1,\dots,x_d$ and $g$ all have compatible system of $p$-power roots and then map $S^\flat$ to a perfect algebra $C$ such that $x_1^\flat,\dots,x_d^\flat$ is a regular sequence on $C$, then untilt to obtain $S^\natural:=(S^\flat)^\sharp \to C^\sharp$ such that $x_1,\dots,x_d$ is a regular sequence on $C^\sharp$. Since $S$ is perfectoid, $S=S^\natural$ and we can replace $C^\sharp$ by its $\m$-adic completion so that $\widehat{C}^\m$ is perfectoid and (balanced) big Cohen-Macaulay \cite[Proposition 2.2.1]{AndreWeaklyFunctorialBigCM} and \cite[Corollary 8.5.3]{BrunsHerzog}.



We record the following lemma on perfectoidization which we will use throughout. We would like to thank Bhargav Bhatt for providing us the argument.

\begin{lemma}
\label{lem.perfdptorsionfree}
Let $S$ be a $p$-torsion free ring over a perfectoid ring $R$. Suppose $S_{\perfd}$ is a (perfectoid) ring, then $S_{\perfd}$ is $p$-torsion free.
\end{lemma}
\begin{proof}
Let $S'_{\perfd}$ be the largest $p$-torsion free quotient of $S_{\perfd}$. Then $S'_{\perfd}$ is a perfectoid ring and we know that $\ker(S_{\perfd}\to S'_{\perfd})$ can be identified with $\ker((S_\perfd/p)_{\perf}\cong (S/p)_{\perf}\to (S'_{\perfd}/p)_{\perf})$ by \cite[2.1.3]{CesnaviciusScholzePurityflatCohomology} or \cite[Lecture IV, Proposition 3.2]{BhattPrismaticLectureNotes}. Thus it is enough to prove that $(S/p)_{\perf}\to (S'_{\perfd}/p)_{\perf}$ is injective.

Since $(S/p)_{\perf}$ is perfect, it embeds into a product of perfect fields $\prod K_i$ (each $K_i$ corresponds to the quotient field of each minimal prime of $(S/p)_{\perf}$). Fix such a $K_i$, it is enough to show that each map $(S/p)_{\perf}\to K_i$ factors over $(S'_{\perfd}/p)_{\perf}$. Since $S$ is $p$-torsion free, the minimal prime $P\in \Spec(S/p)$ corresponding to $K_i$ admits a generalization $Q\in \Spec S$ not containing $p$. We can then choose a map $S\to W$ such that $W$ is a valuation ring witnessing this generalization (i.e., $W$ is supported at $Q$ and the maximal ideal $\m$ of $W$ contracts to $P$). We can enlarge $W$ to assume $W$ is perfectoid without changing this property (e.g., we can replace $W$ by the $p$-adic completion of the absolute integral closure of $W$). Since the map $S/P\to W/\m$ factors over $K_i$, we can further replace $W$ by $W\times_{W/\m} K_i$ to assume $W$ is a perfectoid valuation ring whose residue field is $K_i$. Since $W$ is $p$-torsion free and perfectoid, we have factorizations:
$$S\to S_{\perfd}\to S'_{\perfd}\to W \to K_i.$$
Since $K_i$ is a perfect field, this induces
$$(S/p)_{\perf}\to (S'_{\perfd}/p)_{\perf}\to K_i$$
as desired.
\end{proof}

We start by providing a shorter proof of the existence of perfectoid big Cohen-Macaulay $R^+$-algebras.

\begin{theorem}
\label{thm:bigCMR^+}
Let $(R,\m)$ be a Noetherian complete local domain of mixed characteristic $(0,p)$. Then there exist perfectoid big Cohen-Macaulay $R^+$-algebras.
\end{theorem}
\begin{proof}
By enlarging $R$ if necessary, we may assume $k=R/\m$ is algebraically closed. We fix a complete unramified regular local ring $A\cong W(k)[[x_2,\dots,x_{d}]]$ such that $A\to R$ is a module-finite extension. Then we {\it fix} $$A_{\infty,0}:=\text{$p$-adic completion of } {A[p^{1/p^\infty},x_2^{1/p^\infty},\dots,x_d^{1/p^\infty}]}\subseteq \widehat{R^+},$$
which is a perfectoid algebra. Now for any module-finite domain extension $S$ of $R$ in $R^+$, consider $S^A_{\perfd}:=(A_{\infty,0}\otimes_A S)_{\perfd}$. By the almost purity theorem \cite[Theorem 10.9]{BhattScholzepPrismaticCohomology}, $A_{\infty, 0}/p^n \to S^A_{\perfd}/p^n$ is $(g)_{\perfd}$-almost finite projective where $g$ is the discriminant of $A\to S$. By \autoref{lem.perfdptorsionfree} and \cite[Lemma 4.1.3 (b)]{AndreDirectsummandconjecture}, $S^A_{\perfd}$ is $(g)_{\perfd}$-almost big Cohen-Macaulay with respect to $p, x_2,\dots,x_d$ and hence $S^A_{\perfd}$ is a perfectoid seed. Thus by \cite[Lemma 4.8]{MaSchwedeSingularitiesMixedCharBCM}, $\widehat{\varinjlim}_S S^A_{\perfd}$ is a perfectoid seed. Since $\widehat{R^+}=\widehat{\varinjlim} S\to \widehat{\varinjlim}_S S^A_{\perfd}$\footnote{In fact, as each $S^A_{\perfd}$ maps to $\widehat{R^+}$ by the universal property of the perfection functor, we also have $\widehat{\varinjlim}_S S^A_{\perfd}\to \widehat{R^+}$ and hence $\widehat{R^+}$ is a direct summand of $\widehat{\varinjlim}_S S^A_{\perfd}$.}, $\widehat{R^+}$ is a perfectoid seed and hence it maps to a perfectoid big Cohen-Macaulay algebra.
\end{proof}

The next theorem gives a simpler proof of the existence of weakly functorial perfectoid big Cohen-Macaulay $R^+$-algebras for surjective maps, recovering \cite[Theorem 1.2.1]{AndreWeaklyFunctorialBigCM}. One advantage of this argument is that we do not need to induct on the height of $P$ as in \cite{AndreWeaklyFunctorialBigCM}.

\begin{theorem}
\label{thm:weaklyfunctorialBCM}
Let $(R,\m)$ be a  Noetherian complete local domain of mixed characteristic $(0,p)$ and let $S=R/P$. Then given any perfectoid big Cohen-Macaulay $R^+$-algebra $B$ and a map $R^+\to S^+$, there exists a perfectoid big Cohen-Macaulay $S^+$-algebra $C$ that fits in the following commutative diagram:
\[\xymatrix{
R \ar[r] \ar[d] & S \ar[d] \\
R^+ \ar[r] \ar[d] & S^+ \ar[d]\\
B\ar[r] & C.
}
\]
\end{theorem}
\begin{proof}
Given a map $R^+\to S^+$ is the same as choosing a prime $P^+$ of $R^+$ such that $S^+=R^+/P^+$. We can write $R^+$ as a direct limit of module-finite domain extensions of $R$, and $P^+$ is a thus direct limit of certain primes $Q$ lying over $P$ of $R$. Thus $B/P^+B=\varinjlim_Q B/QB$ and so $B/P^+B$ maps to $\widehat{\varinjlim}_Q(B/QB)_{\perfd}$. Therefore by \cite[Lemma 4.8]{MaSchwedeSingularitiesMixedCharBCM}, it is enough to prove that each $(B/QB)_{\perfd}$ is a perfectoid seed. But then it is enough to show that $(B/PB)_{\perfd}$ is a perfectoid seed (as $Q$ lives on some finite domain extension of $R$ so the argument for such $Q$ is the same as for $P$).

Now we pick $(f_1,\dots,f_c)\in P$ such that $f_1,\dots,f_c$ is part of a system of parameters of $R$ and $P$ is a minimal prime of $(f_1,\dots,f_c)$. Thus there exists $g\notin P$ such that $gP\in \sqrt{(f_1,\dots,f_c)}$. Moreover, if $p\in P$, we take $f_1=p$, and if $p\notin P$, we can assume $p,f_1,\dots,f_c$ is also part of a system of parameters and that $g=pg'$ for $g' \in R$. Extend $f_1,\dots,f_c$ to a full system of parameters $f_1,\dots,f_c, y_1,\dots,y_t$ on $R$ such that the image of $y_1,\dots,y_t$ forms a system of parameters on $S=R/P$ (and if $p\notin P$ we take $y_1=p$). Since $f_1,\dots,f_c$ is a regular sequence on $B$ and they all have a compatible system of $p$-power roots in $B$, we know that $(B/(f_1,\dots,f_c)B)_{\perfd}=B/(f_1^{1/p^\infty},\dots, f_c^{1/p^\infty})^-$. In particular, since $B$ is big Cohen-Macaulay, by the way we choose $f_1,\dots,f_c, y_1,\dots, y_t$, we know that $y_1,\dots,y_t$ is a regular sequence on $(B/(f_1,\dots,f_c)B)_{\perfd}$.

Since $gP\in \sqrt{(f_1,\dots,f_c)}$, we know that $gP=0$ in $(B/(f_1,\dots,f_c)B)_{\perfd}$ and hence $g^{1/p^\infty}P=0$ in $(B/(f_1,\dots,f_c)B)_{\perfd}$ as the latter is reduced. Therefore $$g^{-1/p^\infty}(B/(f_1,\dots,f_c)B)_{\perfd}:=\Hom((g^{1/p^\infty}), (B/(f_1,\dots,f_c)B)_{\perfd})$$ is a $g^{1/p^\infty}$-almost big Cohen-Macaulay $S$-algebra with respect to $y_1,\dots,y_t$: to check the non-triviality condition, note that $B$ is a big Cohen-Macaulay $R$-algebra and $g\notin P$, applying \cite[Proposition 2.5.1]{AndreWeaklyFunctorialBigCM}) to $\pi=g$ we see that $B/(\m B+\sqrt{PB})$ is not $g^{1/p^\infty}$-almost zero, but since
$$\frac{(B/(f_1,\dots,f_c)B)_{\perfd}}{(y_1,\dots,y_t)(B/(f_1,\dots,f_c)B)_{\perfd}}\cong\frac{B}{(y_1,\dots,y_t)B+(f_1^{1/p^\infty},\dots, f_c^{1/p^\infty})^-B}\twoheadrightarrow B/(\m B+\sqrt{PB}),$$ the former is not $g^{1/p^\infty}$-almost zero. Now if $p\notin P$, then by \cite[2.3.1]{AndreWeaklyFunctorialBigCM}, we know that
$$\widetilde{B}:=(g^{-1/p^\infty}(B/(f_1,\dots,f_c)B)_{\perfd})^\natural\to g^{-1/p^\infty}(B/(f_1,\dots,f_c)B)_{\perfd}$$
is a $(pg)^{1/p^\infty}$-almost (and hence $g^{1/p^\infty}$-almost) isomorphism. Here, we implicitly used that $$g^{-1/p^\infty}(B/(f_1,\dots,f_c)B)_{\perfd}$$ is spectral: indeed $(B/(f_1,\dots,f_c)B)_{\perfd}$ is perfectoid and one can check that $\Hom((g^{1/p^\infty}),-)$ preserves spectrality (see \cite[(2.16)]{AndrePerfectoidAbhyankarLemma} or \cite[2.3.2]{AndreWeaklyFunctorialBigCM}).

If $p\in P$, then we set $\widetilde{B}=g^{-1/p^\infty}(B/(f_1,\dots,f_c)B)_{\perfd}$. Hence in both cases, $\widetilde{B}$ is a perfectoid $B/PB$-algebra and is $g^{1/p^\infty}$-almost big Cohen-Macaulay with respect to $y_1,\dots,y_t$. Thus, it is a perfectoid seed. By the universal property of the perfection functor, there exists a map $(B/PB)_{\perfd} \to \widetilde{B}$ which implies that the former ring is also a perfectoid seed.
\end{proof}


We can actually prove the following slight generalization of Andr\'{e}'s result. This will be used in the proof of \autoref{clm.ReesBCMRegularSoIsR}.

\begin{theorem}
\label{thm:GeneralweaklyfunctorialBCM}
Let $(R,\m)\to (S,\n)$ be a local map of  Noetherian complete local domains such that $R$ has mixed characteristic $(0,p)$ and $R$, $S$ have the same residue field. Then given any perfectoid big Cohen-Macaulay $R^+$-algebra $B$, there exists a commutative diagram:
\[\xymatrix{
R \ar[r] \ar[d] & S \ar[d] \\
R^+ \ar[r] \ar[d] & S^+ \ar[d]\\
B\ar[r] & C
}
\]
where $C$ is a perfectoid big Cohen-Macaulay $S^+$-algebra.
\end{theorem}
\begin{proof}
Since $R$, $S$ have the same residue field, the image of any coefficient ring of $R$ in $S$ is a coefficient ring of $S$. Hence by \cite[Proof of Theorem 1.1]{AvramovFoxbyHerzogStructureofLocalhomomorphisms}, the map $R\to S$ can be factored as $R\to T\to S$ where $T=R[[x_1,\dots,x_n]]$ and $T\to S$ is surjective. By \autoref{thm:weaklyfunctorialBCM}, it is enough to construct the diagram for $(R,\m)\to (T,\n)$ where $\n=\m+(x_1,\dots,x_n)$.

Let $p,y_2,\dots,y_d$ be a system of parameters of $R$ and let $\widetilde{B}$ be the $(p,y_2,\dots,y_d,x_1,\dots,x_n)$-adic completion of $B[x_1^{1/p^\infty},\dots,x_n^{1/p^\infty}]$. Since $p,y_2,\dots,y_d,x_1,\dots,x_n$ is a regular sequence on $B[x_1^{1/p^\infty},\dots,x_n^{1/p^\infty}]$ (as $B$ is big Cohen--Macaulay over $R$), applying \cite[Proposition 2.2.1]{AndreWeaklyFunctorialBigCM} we know that $\widetilde{B}$ is a perfectoid big Cohen--Macaulay $T$-algebra and the following diagram commutes:
\[\xymatrix{
R^+ \ar[r] \ar[d] & R^+\otimes_RT \ar[d] \\
B \ar[r] & \widetilde{B}
}
\]

For every module-finite domain extension $R'$ of $R$ inside $R^+$, we let $T':=R'\otimes_RT\cong R'[[x_1,\dots,x_n]]$ be compatibly chosen inside $T^+$.  For every module-finite domain extension $T''$ of $T'$ inside $T^+$, consider $(\widetilde{B}\otimes_{T'}T'')_{\perfd}$. By \cite[Theorem 10.9]{BhattScholzepPrismaticCohomology}, $(\widetilde{B}\otimes_{T'}T'')_{\perfd}/p^n$ is $(g)_{\perfd}$-almost finite projective over $\widetilde{B}/p^n$ where $g\in T'$ is such that $T'_g\to T''_g$ is finite \'{e}tale. In particular, by \autoref{lem.perfdptorsionfree}, $(\widetilde{B}\otimes_{T'}T'')_{\perfd}$ is a perfectoid seed and hence $B':=\widehat{\varinjlim}_{R',T''}(\widetilde{B}\otimes_{T'}T'')_{\perfd}$ is a perfectoid seed by \cite[Lemma 4.8]{MaSchwedeSingularitiesMixedCharBCM}. As a consequence, $\widetilde{B}\otimes_{(R^+\otimes_RT)}T^+=\varinjlim_{R',T''}\widetilde{B}\otimes_{T'}T''$ maps to a perfectoid seed $B'$, and hence to a perfectoid big Cohen--Macaulay $T$-algebra $C$, so we have the following commutative diagram:
\[\xymatrix{
R^+ \ar[r] \ar[d] & R^+\otimes_RT \ar[d] \ar[r] &  T^+ \ar[d] & & \\
B \ar[r] & \widetilde{B} \ar[r] & \widetilde{B}\otimes_{(R^+\otimes_RT)}T^+ \ar[r] & B' \ar[r] & C
}
\]
It is clear that $C$ is the desired perfectoid big Cohen--Macaulay $T^+$-algebra.
\end{proof}

The next result is a slightly different proof of \cite[Lemma 4.5]{MaSchwedeSingularitiesMixedCharBCM}.

\begin{theorem}
\label{thm:dominateBCM}
Let $(R,\m)$ be a  Noetherian complete local domain of mixed characteristic $(0,p)$ and let $B_1$, $B_2$ be two perfectoid big Cohen-Macaulay $R^+$-algebras. Then $B_1\widehat{\otimes}_{{R^+}}B_2$ maps $\widehat{R^+}$-linearly to another perfectoid big Cohen-Macaulay $R^+$-algebra $B$.
\end{theorem}
\begin{proof}
We fix a complete unramified regular local ring $A$ inside $R$ such that $A\to R$ is module-finite. Let $p, x_2,\dots, x_d$ be a regular system of parameters of $A$. Next, as in the first paragraph of the proof of \cite[Lemma 4.5]{MaSchwedeSingularitiesMixedCharBCM}, we can replace $B_1$, $B_2$ by their $\m$-adic completions to assume that $B_1, B_2$ are algebras over $A_0:=W(\overline{k})[[x_2,\dots,x_d]]$ where $k=R/\m$ (the reason we do this step is because $k$ is not necessarily perfect). We next {\it fix} $$A_{\infty,0}:=\text{$p$-adic completion of } {A_0[p^{1/p^\infty},x_2^{1/p^\infty},\dots,x_d^{1/p^\infty}]}\subseteq A_0\widehat{\otimes}_A{R^+},$$
which is a perfectoid algebra. Now for any module-finite domain extension $S$ of $R$ in $R^+$, consider $S^A_{\perfd}:=(A_{\infty,0}\otimes_A S)_{\perfd}$. Note that since $B_1$, $B_2$ are perfectoid, they are algebras over $S^A_{\perfd}$ by the universal property of the perfection functor.

Since $B_1$, $B_2$ are faithfully flat over $A_{\infty, 0}$ mod $p$, $B_1\widehat{\otimes}_{A_{\infty, 0}}B_2$ is also faithfully flat over $A_{\infty, 0}$ mod $p$. Therefore $B_1\widehat{\otimes}_{A_{\infty, 0}}B_2$ is perfectoid and is big Cohen-Macaulay with respect to $p, x_2,\dots, x_{d}$. Consider the following map:
\small
$$B_1\otimes_{A_{\infty, 0}}B_2\cong B_1\otimes_{S^A_{\perfd}} (S^A_{\perfd}\otimes_{A_{\infty, 0}}S^A_{\perfd})\otimes_{S^A_{\perfd}} B_2 \xrightarrow{\mu} B_1\otimes_{S^A_{\perfd}} {S^A_{\perfd}}\otimes_{S^A_{\perfd}}B_2\cong B_1\otimes_{S^A_{\perfd}}B_2.$$
\normalsize
By the almost purity theorem \cite[Theorem 10.9]{BhattScholzepPrismaticCohomology}, $A_{\infty, 0}/p^n \to S^A_{\perfd}/p^n$ is $(g)_{\perfd}$-almost unramified where $g$ is the discriminant of $A\to S$. In particular, the multiplication map $$(S^A_{\perfd}\otimes_{A_{\infty, 0}}S^A_{\perfd})/p\xrightarrow{\mu} S^A_{\perfd}/p$$ is $(g)_{\perfd}$-almost projective, i.e., $S^A_{\perfd}/p$ is a $(g)_{\perfd}$-almost direct summand of $(S^A_{\perfd}\otimes_{A_{\infty, 0}}S^A_{\perfd})/p$. Therefore $B_1\widehat{\otimes}_{S^A_{\perfd}}B_2$ is a $(g)_{\perfd}$-almost direct summand of $B_1\widehat{\otimes}_{A_{\infty, 0}}B_2$ mod $p$. 

We next claim that $B':=B_1\widehat{\otimes}_{S^A_{\perfd}}B_2$ is $(g)_{\perfd}$-almost big Cohen-Macaulay with respect to $x_1:=p, x_2,\dots,x_d$ and hence a perfectoid seed. To see this, first note that $x_{i+1}$ is a $(g)_{\perfd}$-almost nonzerodivisor on $B'/(x_1,\dots,x_i)B'$ for each $i$, because this module almost injects into $(B_1\widehat{\otimes}_{A_{\infty, 0}}B_2)/(x_1,\dots,x_i)(B_1\widehat{\otimes}_{A_{\infty, 0}}B_2)$ (by the $(g)_{\perfd}$-almost direct summand condition) and $x_{i+1}$ is a nonzerodivisor on the latter. Second, if we use $T$ to denote the perfectoid $(B_1\widehat{\otimes}_{A_{\infty, 0}}B_2)$-algebra obtained by adjoining compatible system of $p$-power roots of $g$ (see \cite[Theorem 7.12]{BhattScholzepPrismaticCohomology}), then it is shown in \cite[diagram on page 2836]{MaSchwedeSingularitiesMixedCharBCM} that the natural map
$$S\otimes_AS \to B_1 {\otimes}_{A_{\infty, 0}}B_2 \to (g^{-1/p^\infty}T)^\natural$$
factors through the multiplication map $S\otimes_AS \xrightarrow{\mu} S$. As a consequence we have an induced map $B_1 {\otimes}_{A_{\infty, 0}\otimes_A S}B_2\to (g^{-1/p^\infty}T)^\natural$ and hence by the universal property of the perfectoidizaton functor (see also \cite[Proposition 8.12]{BhattScholzepPrismaticCohomology}) we have an induced map $$B'=B_1\widehat{\otimes}_{S^A_{\perfd}}B_2 \to (g^{-1/p^\infty}T)^\natural \to g^{-1/p^\infty}T.$$
If $B'/(x_1,\dots,x_d)B'$ is $(g)_{\perfd}$-almost zero, then so is $(g^{-1/p^\infty}T)/(x_1,\dots,x_d)(g^{-1/p^\infty}T)$ which contradicts \cite[Claim 4.6]{MaSchwedeSingularitiesMixedCharBCM}. Therefore $B'/(x_1,\dots,x_d)B'$ is not $(g)_{\perfd}$-almost zero. It follows that $B'=B_1\widehat{\otimes}_{S^A_{\perfd}}B_2$ is $(g)_{\perfd}$-almost big Cohen-Macaulay with respect to $p, x_2,\dots,x_d$ as desired.


Finally, as each $S$ maps to $A_{\infty, 0}\otimes_AS$, we know that ${R^+}$ maps to ${\varinjlim}_SS^A_{\perfd}$. Since each $B_1\widehat{\otimes}_{S^A_{\perfd}}B_2$ is a perfectoid seed, by \cite[Lemma 4.8]{MaSchwedeSingularitiesMixedCharBCM} we know that $\widehat{\varinjlim}_S(B_1\widehat{\otimes}_{S^A_{\perfd}}B_2)$ is also a perfectoid seed. But then as $B_1\otimes_{R^+}B_2$ maps to $B_1\otimes_{{\varinjlim}_SS^A_{\perfd}}B_2\cong \varinjlim_SB_1\otimes_{S^A_{\perfd}}B_2$ and the latter maps to $\widehat{\varinjlim}_S(B_1\widehat{\otimes}_{S^A_{\perfd}}B_2)$, we know that $B_1\otimes_{R^+}B_2$ maps to a perfectoid seed and so $B_1{\widehat{\otimes}}_{R^+}B_2$ maps to another perfectoid big Cohen-Macaulay algebra $B$.
\end{proof}

We next prove the existence of weakly functorial perfectoid big Cohen-Macaulay algebras for certain diagrams (with a certain map factorizing through $R^+$). This will be a crucial ingredient in our comparison result with the adjoint ideal from birational geometry. Note that \cite[Theorem F on page 10]{GabberMSRINotes} claims a version of functorial big Cohen-Macaulay algebras, but to the best of our knowledge, that is not enough for our purpose.

\begin{theorem}
\label{thm:weaklyfunctorialdiagram}
Let $(R,\m)$ be a  Noetherian complete local domain of mixed characteristic $(0,p)$. Let $P_1$, $P_2$ be two height one primes of $R$ such that $Q=P_1+P_2$ is a height two prime with $R_Q$ regular. Then there exists a commutative diagram:
\[
  \xymatrix{%
    R \ar@{->}[rr]\ar@{->}[dd] & & R/P_1 \ar@{->}'[d][dd]
    \\
    & R/P_2 \ar@{<-}[ul]\ar@{->}[rr]\ar@{->}[dd] & &  R/Q
    \ar@{<-}[ul]\ar@{->}[dd]
    \\
    B \ar@{->}'[r][rr] & & B_1
    \\
    & B_2 \ar@{->}[rr]\ar@{<-}[ul] & &  C \ar@{<-}[ul]
  }
\]
where $B$, $B_1$, $B_2$ and $C$ are perfectoid big Cohen-Macaulay algebras over $R$, $R/P_1$, $R/P_2$ and $R/Q$ respectively. Moreover, we can take $B$, $B_1$ and $C$ to be algebras over $R^+$, $(R/P_1)^+$, and $(R/Q)^+$ respectively for a given compatibly chosen $R^+\to (R/P_1)^+\to (R/Q)^+$, and $B$ can be given in advance.
\end{theorem}
\begin{proof}
We first prove the following claim:
\begin{claim}
There exists $x\in P_2$ and $g\notin Q$ such that $gP_2\subseteq (x)$ 
and $x \notin P_1$.
\end{claim}
\begin{proof}[Proof of Claim]
Since $R_Q$ is regular and $P_2R_Q$ is a principal ideal, we can pick $x\in P_2$ and $h\notin Q$ such that $R_h$ is regular and such that $(x)R_h=P_2R_h$. Note that $x\notin P_1$ since otherwise $P_2R_h\subseteq P_1R_h$ and thus $P_2R_Q\subseteq P_1R_Q$ contradicting $P_1+P_2=Q$ has height two. Therefore there exists $g=h^N$ such that $gP_2\subseteq (x)$.
\end{proof}

Now we prove the theorem. We start with a perfectoid big Cohen-Macaulay $R^+$-algebra $B$. By \autoref{thm:weaklyfunctorialBCM}, we have a commutative diagram:
\[\xymatrix{
R \ar[r] \ar[d] & R/P_1\ar[d] \\
R^+\ar[r] \ar[d] & (R/P_1)^+  \ar[d]\\
B\ar[r] & B_1
}
\]
where $B_1$ is a perfectoid big Cohen-Macaulay $(R/P_1)^+$-algebra. Since {$gP_2\subseteq (x)$}, we know that the image of $P_2$ is $g$-torsion in $B/(x^{1/p^\infty})^-$ (respectively, the image of $Q$ is $g$-torsion in $B_1/(x^{1/p^\infty})^-$). Since $B/(x^{1/p^\infty})^-$ and $B_1/(x^{1/p^\infty})^-$ are perfectoid algebras, they are reduced so any $g$-torsion is $g^{1/p^\infty}$-torsion (note that $g^{1/p^\infty}$ exists in $B$ and $B_1$ since they are algebras over $R^+$ and $(R/P_1)^+$). Thus we have
\[\xymatrix{
R/P_2 \ar[r] \ar[d] & R/Q \ar[d] \\
g^{-1/p^\infty}(B/(x^{1/p^\infty})^-) \ar[r] & g^{-1/p^\infty}(B_1/(x^{1/p^\infty})^-).
}
\]
The bottom row in the above diagram are  $g^{1/p^\infty}$-almost perfectoid and $g^{1/p^\infty}$-almost big Cohen-Macaulay algebras over $R/P_2$ and $R/Q$ respectively\footnote{See \cite[proof of 4.3.3]{AndreWeaklyFunctorialBigCM}, which can be also adapted to characteristic $p>0$, or we can use the same argument as in the proof of \autoref{thm:weaklyfunctorialBCM} (the essential point is \cite[Proposition 2.5.1 (2)]{AndreWeaklyFunctorialBigCM}).}. At this point, we apply Gabber's method to construct perfectoid big Cohen-Macaulay algebras from the $g^{1/p^\infty}$-almost ones\footnote{Alternatively, we can also proceed carefully using the strategy as in \cite[proof of Theorem 3.1]{AndreWeaklyFunctorialBigCM} {\it simultaneously} to $g^{-1/p^\infty}(B/(x^{1/p^\infty})^-)$ and $g^{-1/p^\infty}(B_1/(x^{1/p^\infty})^-)$.}. More precisely, by \cite[second paragraph on page 3]{GabberMSRINotes} (see also \cite[17.5]{GabberRameroFoundationsAlmostRingTheory} for more details), we have $$B_2:=\text{$p$-adic completion of } {S_g^{-1}\left(g^{-1/p^\infty}(B/(x^{1/p^\infty})^-)\right)^\diamond}$$ is a perfectoid big Cohen-Macaulay $R/P_2$-algebra where $T^\diamond=(\prod^\mathbb{N}T)/(\bigoplus^\mathbb{N}T)$ for any commutative ring $T$ and $S_g$ denote the multiplicative system consisting of $(g^{\varepsilon_0}, g^{\varepsilon_1},\dots)$ such that $\varepsilon_i\in \mathbb{N}[1/p]$ and $\varepsilon_i\to 0$. Similarly, $C_0:=\text{$p$-adic completion of } {S_g^{-1}\left(g^{-1/p^\infty}(B_1/(x^{1/p^\infty})^-)\right)^\diamond}$ is a perfectoid big Cohen-Macaulay $R/Q$-algebra. Putting all these together, we have now constructed a commutative
\begin{equation}
\label{eqn:commdiag}
  \xymatrix{%
    R \ar@{->}[rr]\ar@{->}[dd] & & R/P_1 \ar@{->}'[d][dd]
    \\
    & R/P_2 \ar@{<-}[ul]\ar@{->}[rr]\ar@{->}[dd] & &  R/Q
    \ar@{<-}[ul]\ar@{->}[dd]
    \\
    B \ar@{->}'[r][rr] & & B_1
    \\
    & B_2 \ar@{->}[rr]\ar@{<-}[ul] & &  C_0 \ar@{<-}[ul]
  }
\end{equation}
satisfying all the conclusions {\it except} that $C_0$ is not necessarily an algebra over $(R/Q)^+$. The last step is to modify $C_0$ to an $(R/Q)^+$-algebra, the idea is very similar to the proof of \autoref{thm:weaklyfunctorialBCM}.

We set $\overline{R}=R/P_1$ and we abuse notations a bit and still use $Q$ to denote the image of $Q$ in $\overline{R}$. Note that $C_0$ is an algebra over $\overline{R}^+/Q\overline{R}^+$ by construction and the diagram:
\[\xymatrix{
R\ar[r] \ar[d] & \overline{R} \ar[r] \ar[d] & \overline{R}/Q \ar[d]\\
R^+\ar[r] \ar[d] & \overline{R}^+ \ar[r] \ar[d] & \overline{R}^+/Q\overline{R}^+\ar[d]\\
B\ar[r] & B_1 \ar[r] & C_0
}
\]
is commutative. Note that given the map $(R/P_1)^+=\overline{R}^+\to (\overline{R}/Q)^+$ is the same as given a prime $Q^+$ inside $\overline{R}^+$ such that $(\overline{R}/Q)^+=\overline{R}^+/Q^+$, and $Q^+$ contracts to a certain height one prime $Q'$ lying over $Q$ in each finite domain extensions $S$ of $\overline{R}$. Therefore we can write $(\overline{R}/Q)^+=\varinjlim_{Q'}\overline{R}^+/Q'\overline{R}^+$ and thus we have a commutative diagram: 
\[\xymatrix{
\overline{R}^+/Q\overline{R}^+ \ar[r]\ar[d] & (\overline{R}/Q)^+ \ar[d] \ar[rd] \\
C_0 \ar[r] & \varinjlim_{Q'}C_0/Q'C_0 \ar[r] & \widehat{\varinjlim}_{Q'}(C_0/Q'C_0)_{\perfd}.
}
\]
It is enough to show that $\widehat{\varinjlim}_{Q'}(C_0/Q'C_0)_{\perfd}$ is a perfectoid seed. By then by \cite[Lemma 4.8]{MaSchwedeSingularitiesMixedCharBCM}, it is enough to prove that each $(C_0/Q'C_0)_{\perfd}$ is a perfectoid seed. Since $\overline{R}/Q\to S/Q'$ is a finite domain extension, $Q'$ is a minimal prime of $\sqrt{QS}$. So there exists $g'\in S$, $g'\notin Q'$, such that $g'Q'\in \sqrt{QS}$. Since $C_0$ is reduced, $\sqrt{QS}$ maps to zero in $C_0$. It follows that the image of $Q'$ in $C_0$ is $g'$-torsion and thus $g'^{1/p^\infty}$-torsion in $C_0$ (again because $C_0$ is reduced). By the universal property of the perfection functor, we have
$$C_0\twoheadrightarrow (C_0/Q'C_0)_{\perfd} \to (g'^{-1/p^\infty}C_0)^\natural,$$
where $(g'^{-1/p^\infty}C_0)^\natural$ denotes $g'^{-1/p^\infty}C_0$ if it has characteristic $p>0$. The composition map is an almost isomorphism (where almost is measured with respect to $(pg')^{1/p^\infty}$ in mixed characteristic, and is measured with respect to $g'^{1/p^\infty}$ in characteristic $p>0$), thus as the first map is surjective, $(C_0/Q'C_0)_{\perfd} \to (g'^{-1/p^\infty}C_0)^\natural$ is also an almost isomorphism. Since $(g'^{-1/p^\infty}C_0)^\natural$ is almost big Cohen-Macaulay (for the non-triviality condition, again one makes use of \cite[Proposition 2.5.1]{AndreWeaklyFunctorialBigCM} as in the proof of \autoref{thm:weaklyfunctorialBCM}), so is $(C_0/Q'C_0)_{\perfd}$ and hence $(C_0/Q'C_0)_{\perfd}$ is a perfectoid seed as desired.

We have showed that $\widehat{\varinjlim}_{Q'}(C_0/Q'C_0)_{\perfd}$ is a perfectoid seed, therefore it maps to a perfectoid big Cohen-Macaulay algebra $C$. Moreover, by construction we have a commutative diagram:
\[\xymatrix{
R\ar[r] \ar[d] & \overline{R} \ar[r] \ar[d] & \overline{R}/Q \ar[d]\ar[rd]\\
R^+\ar[r] \ar[d] & \overline{R}^+ \ar[r] \ar[d] & \overline{R}^+/Q\overline{R}^+\ar[d]\ar[r] & (\overline{R}/Q)^+\ar[d]\\
B\ar[r] & B_1 \ar[r] & C_0 \ar[r] & C
}
\]
Putting this together with \autoref{eqn:commdiag} we get all the desired conclusions.
\end{proof}

To establish our last result, we need the following definition.

\begin{definition}
Let $(R,\m)\to (S,\n)$ be a local map of Noetherian complete local domains such that $R/\m$ has characteristic $p>0$. We say a map $B'\to C'$ is a {\it perfectoid seed morphism} if $B'$, $C'$ are perfectoid $R$- and $S$-algebras respectively and there exists a commutative diagram
\[\xymatrix{
R\ar[r]\ar[d] & S\ar[d] \\
B' \ar[r] \ar[d] & C'\ar[d]\\
B\ar[r] & C
}
\]
where $B$, $C$ are perfectoid big Cohen-Macaulay algebras over $R$ and $S$ respectively.
\end{definition}

\begin{lemma}
\label{lem:perfectoidseedmorphism}
Let $(R,\m)$ be a  Noetherian complete local domain such that $R/\m$ has characteristic $p>0$. Let $S=R/P$
{
for a prime ideal $P$} and suppose $\{B_\lambda\to C_\lambda\}_\lambda$ is a direct system of perfectoid seed morphisms for $R\to S$. Then $\widehat{\varinjlim}_\lambda B_\lambda \to \widehat{\varinjlim}_\lambda C_\lambda$ is also a perfectoid seed morphism.
\end{lemma}
\begin{proof}
The idea is basically to run the proofs of \cite[Lemma 4.8]{MaSchwedeSingularitiesMixedCharBCM} and \cite[Lemma 3.2]{DietzBCMSeeds} {\it for maps}. First we want to reduce to the case that $R$ has characteristic $p>0$. We fix a system of parameters $x_1,\dots,x_r, y_1,\dots,y_s$ of $R$ such that $y_1,\dots,y_s$ is a system of parameters for $S=R/P$, and such that if $S$ has mixed characteristic, then we set $y_1=p$ and if $S$ has characteristic $p>0$ and $R$ has mixed characteristic, then we set $x_1=p$. Note that by the refinement of Andr\'{e}'s flatness lemma \cite[Theorem 7.12]{BhattScholzepPrismaticCohomology}, we can enlarge each $B_\lambda$ to $B_\lambda'$ such that $x_1,\dots,x_r, y_1,\dots,y_s$ have compatible system of $p$-power roots in $B_\lambda'$ and that $B_\lambda'/p$ is faithfully flat over $B_\lambda/p$. Thus replacing $B_\lambda\to C_\lambda$ by $B_\lambda'\to C_\lambda':=C_\lambda\widehat{\otimes}_{B_\lambda}B_\lambda'$,\footnote{It is easy to check that $B_\lambda'\to C_\lambda'$ is still a perfectoid seed morphism: if $B_\lambda\to C_\lambda$ maps to $B\to C$, then $B_\lambda'\to C_\lambda'$ maps to $B':= B\widehat{\otimes}_{B_\lambda}B_\lambda'\to C':=C\widehat{\otimes}_{B_\lambda}B_\lambda'$ with $B'/p$ (resp. $C'/p$) faithfully flat over $B/p$ (resp. $C/p$), and so $\widehat{B'}^{\m}\to \widehat{C'}^{\m}$ is a map of balanced big Cohen-Macaulay algebras by \cite[Corollary 8.5.3]{BrunsHerzog}.} we may assume that each $B_\lambda$ admits a compatible system of $p$-power roots of $x_1,\dots,x_r, y_1,\dots,y_s$.

Now if If $B'\to C'$ is any perfectoid seed morphism such that $B'$ admits a compatible system of $p$-power roots of $x_1,\dots,x_r, y_1,\dots,y_s$, then we have
\[\xymatrix{
{B'}^\flat \ar[r] \ar[d] & {C'}^\flat\ar[d] \\
\widehat{B^\flat} \ar[r] & \widehat{C^\flat}
}
\]
where the completion is taken with respect to the ideal $(x_1^\flat,\dots, x_r^\flat, y_1^\flat,\dots,y_s^\flat)$, where $x_i^\flat:=(x_i, x_i^{1/p}, x_i^{1/p^2},\dots)\in B'^\flat$ exists by our assumption. By \cite[Corollary 8.5.3]{BrunsHerzog}, $\widehat{B^\flat} \to \widehat{C^\flat}$ is a map of (balanced) perfectoid big Cohen-Macaulay algebras for the map $$\mathbb{F}_p[[x_1^\flat,\dots, x_r^\flat, y_1^\flat,\dots,y_s^\flat]]\to \mathbb{F}_p[[y_1^\flat,\dots,y_s^\flat]].$$
Thus ${B'}^\flat \to {C'}^\flat$ is a perfectoid seed morphism for the above map.

By the above discussion, $\{B_\lambda^\flat\to C_\lambda^\flat\}_\lambda$ is a direct system of perfectoid seed morphisms in characteristic $p>0$. If we can show $\widehat{\varinjlim}_\lambda^{p^\flat} B_\lambda^\flat \to \widehat{\varinjlim}_\lambda^{p^\flat} C_\lambda^\flat$ is a perfectoid seed morphism in characteristic $p>0$, i.e., we have a commutative diagram
\[\xymatrix{
\widehat{\varinjlim}^{p^\flat}_\lambda B_\lambda^\flat \ar[r]\ar[d] & \widehat{\varinjlim}^{p^\flat}_\lambda C_\lambda^\flat \ar[d]\\
B \ar[r] & C,
}
\]
then we can untilt to get a commutative diagram
\[\xymatrix{
\widehat{\varinjlim}_\lambda B_\lambda \ar[r]\ar[d] & \widehat{\varinjlim}_\lambda C_\lambda \ar[d]\\
B^\sharp \ar[r] \ar[d] & C^\sharp \ar[d] \\
\widehat{B^\sharp}^{\m} \ar[r] & \widehat{C^\sharp}^{\n}
}
\]
where $\widehat{B^\sharp}^{\m}$ and $\widehat{C^\sharp}^{\n}$ are perfectoid big Cohen-Macaulay algebras over $R$ and $S$ respectively (see \cite[proof of Lemma 4.8]{MaSchwedeSingularitiesMixedCharBCM}). Therefore, without loss of generality we can assume that $R$ has characteristic $p>0$.

Now in characteristic $p>0$, it is clear that if we can map $B'\to C'$ to a map $B\to C$ of big Cohen-Macaulay algebras, then we can always replace $B$ and $C$ by their usual perfections (i.e., $\varinjlim_e F^e_*B$) to assume $B$ and $C$ are perfect (equivalently, perfectoid). But then by the proof of \cite[Theorem 4.2]{HochsterBigCohen-Macaulayalgebrasindimensionthree} (see also \cite[section 1.2]{HeitmannMaBigCohenMacaulayAlgebraVanishingofTor}), we know that $B'\to C'$ is a perfectoid seed morphism in characteristic $p>0$ if and only if there is no bad double sequence of partial algebra modifications of $B'$ over $B'\to C'$. Now if there exists a bad double sequence of partial algebra modification of $\varinjlim_\lambda B_\lambda$ over $\varinjlim_\lambda B_\lambda\to \varinjlim_\lambda C_\lambda$, then by the finiteness nature of (partial) algebra modifications, there exists $\lambda$ such that the bad double sequence is defined over $B_\lambda\to C_\lambda$, contradicting that $B_\lambda\to C_\lambda$ is a perfectoid seed. Therefore $\varinjlim_\lambda B_\lambda\to \varinjlim_\lambda C_\lambda$ is a perfectoid seed morphism and so passing to completions, we know that $\widehat{\varinjlim}^{p^\flat}_\lambda B_\lambda\to \widehat{\varinjlim}^{p^\flat}_\lambda C_\lambda$ is also a perfectoid seed morphism.
\end{proof}

We next prove a crucial result for our adjoint ideal, the idea is similar to the proof of \autoref{thm:dominateBCM} and \cite[Theorem 4.9]{MaSchwedeSingularitiesMixedCharBCM}.

\begin{theorem}
\label{thm:dominateMap}
With notation as in \autoref{def:adjoint}, let $\{B_\gamma\to C_\gamma\}_{\gamma\in\Gamma}$ be a set of compatible choices of perfectoid big Cohen-Macaulay $R^+$- and $(R/I_D)^+$-algebras. Then we can find another $\mathcal{B}\to \mathcal{C}$ such that $$\adj_{\mathcal{B}\shortrightarrow \mathcal{C}}(R, D+\Delta)\subseteq \adj_{B_\gamma\shortrightarrow C_\gamma}(R, D+\Delta)$$ for all $\gamma\in \Gamma$.
\end{theorem}
\begin{proof}
We assume $R$ has mixed characteristic $(0,p)$ in the proof.\footnote{We believe a similar and simpler argument also works when $R$ has characteristic $p>0$: one needs to replace the citation to \autoref{thm:dominateBCM} by \cite[Theorem 8.4]{DietzBCMSeeds}, and for choosing $S$ and $A$ we use \cite[Theorem 0.14]{HeitmannEtalelocus}. We omit the details since, at least when $R$ is $F$-finite of characteristic $p>0$, the conclusion of \autoref{thm:dominateMap} follows from our more general \autoref{thm.TauEqualityCharP}.} We first fix $S$ a finite normal domain extension of $R$ that contains $f^{1/n}$ and a height one prime ideal $P\subseteq S$ lying over $I_D$. By enlarging $S$ and $P$, we may assume that $S$ is module-finite over $A=V[[x_2,\dots,x_d]]$ where $(V,\pi)$ is a complete DVR such that the discriminant of the map $A\to S$ is not contained in $P$: if $S/P$ has characteristic $p>0$, this follows from \cite[Theorem 4.2.2]{OrgogozoGabber} and in this case $P$ contracts to $(\pi)$ in $A$; if $S/P$ has mixed characteristic, then this essentially follows from \cite[Theorem 0.14]{HeitmannEtalelocus}: more precisely, the proof of \cite[Theorem 0.14]{HeitmannEtalelocus} shows that we can find $x\in P$ such that $S/xS$ is reduced and $p, x$ is part of a system of parameters on $S$, thus we can use Cohen's structure theorem to find $A$ inside $S$ with $p,x_2:=x$ part of a regular system of parameters, and it follows that $P$ contracts to $(x_2)$ and the discriminant of $A\to S$ is not contained in $(x_2)$ because $A_{(x_2)}\to S_Q$ is \'{e}tale for each $Q$ lying over $(x_2)$ (since $S/x_2S$ is reduced and the residue fields of $A_{(x_2)}$ and $S_Q$ have characteristic $0$). 
It follows that in both cases, we can construct a diagram:
\[\xymatrix{
A\ar[r]\ar[d] & \overline{A}=A/zA \ar[d] \\
S\ar[r] & \overline{S}=S/P
}\]
where $z=x_2$ if $p\notin P$, and $z=\pi$ if $p\in P$. By the way we choose $S$ and $A$, we know that there exists $g\in A-zA$, such that $A\to S$ and $\overline{A}\to \overline{S}$ are both finite \'{e}tale.

We next consider an integral extension $V\to V'$ where $(V',\pi)$ has the same uniformizer as $V$ but the residue field $V'/\pi$ is perfect. Then as in the first paragraph of the proof of \cite[Lemma 4.5]{MaSchwedeSingularitiesMixedCharBCM}, we can replace each $B_\gamma$ by its $\m$-adic completion to assume that each $B_\gamma$ is an algebra over $A_0:=V'[[x_2,\dots,x_d]]$ (the reason we do this step is because the residue field of $S$ and $A$ may not be perfect). We next {\it fix} $$A_{\infty,0}:=\text{$p$-adic completion of } {V'[\pi^{1/p^\infty},x_2^{1/p^\infty},\dots,x_d^{1/p^\infty}]}\subseteq A_0\widehat{\otimes}_A{R^+},$$
which is a perfectoid algebra. Similarly we replace each $C_\gamma$ by its $\m$-adic completion and fix $\overline{A}_{\infty,0}$ (note that if $z=\pi$, then $\overline{A}_{\infty,0}$ is just $\overline{A}_{\text{perf}}$). Now we consider $S^A_{\perfd}:=(A_{\infty,0}\otimes_A S)_{\perfd}$ and $\overline{S}^{\overline{A}}_{\perfd}:=(\overline{A}_{\infty,0}\otimes_{\overline{A}} \overline{S})_{\perfd}=(A_{\infty,0}\otimes_{A} \overline{S})_{\perfd}$. Note that since the $B_\gamma$'s (respectively, $C_\gamma$'s) are perfectoid, they are algebras over $S^A_{\perfd}$ (respectively, $\overline{S}^{\overline{A}}_{\perfd}$) by the universal property of the perfection functor. Now, for any finite subset $\Lambda=\{\gamma_1,\dots,\gamma_n\}\subseteq \Gamma$, we consider the map
$$B_\Lambda:=B_{\gamma_1}\widehat{\otimes}_{{S}^{{A}}_{\perfd}}B_{\gamma_2}\widehat{\otimes}_{{S}^{{A}}_{\perfd}}\cdots\widehat{\otimes}_{{S}^{{A}}_{\perfd}}B_{\gamma_n} \to C_\Lambda:=C_{\gamma_1}\widehat{\otimes}_{\overline{S}^{\overline{A}}_{\perfd}}C_{\gamma_2}\widehat{\otimes}_{\overline{S}^{\overline{A}}_{\perfd}}\cdots \widehat{\otimes}_{\overline{S}^{\overline{A}}_{\perfd}}C_{\gamma_n}$$
of perfectoid algebras. By the same argument as in \autoref{thm:dominateBCM} (applied repeatedly), we know that $B_\Lambda$ is $(g)_{\perfd}$-almost big Cohen-Macaulay with respect to $\pi, x_2,\dots,x_d$ and $C_\Lambda$ is $(g)_{\perfd}$-almost big Cohen-Macaulay with respect to $x, x_3,\dots,x_d$ where $x=\pi$ if $p\notin P$ and $x=x_2$ if $p\in P$. Hence using the same strategy as in \cite[Theorem 3.1]{AndreWeaklyFunctorialBigCM} simultaneously to $B_\Lambda$ and $C_\Lambda$ (or use Gabber's method \cite[17.5]{GabberRameroFoundationsAlmostRingTheory} as in the proof of \autoref{thm:weaklyfunctorialdiagram}), we have a commutative diagram:
\[\xymatrix{
B_\Lambda \ar[r] \ar[d] & C_\Lambda \ar[d] \\
B \ar[r] & C
}
\]
where $B\to C$ is a map of perfectoid big Cohen-Macaulay algebras. In other words, $B_\Lambda\to C_\Lambda$ is a perfectoid seed morphism. Now clearly, $$\{B_\Lambda\to C_\Lambda\}_{\Lambda\subseteq \Gamma, |\Lambda|<\infty}$$ is a {\it direct system} of perfectoid seed morphisms (where the transition maps are the obvious ones). Therefore by \autoref{lem:perfectoidseedmorphism}, $\widehat{\varinjlim}_\Lambda B_\Lambda\to \widehat{\varinjlim}_\Lambda C_\Lambda$ is also a perfectoid seed morphism. So there exists a commutative diagram of perfectoid big Cohen-Macaulay algebras
\[\xymatrix{
B_\Lambda \ar[r] \ar[d] & C_\Lambda \ar[d] \\
\mathcal{B} \ar[r] & \mathcal{C}
}
\]
where $\mathcal{B} \to \mathcal{C}$ is a map of perfectoid big Cohen-Macaulay algebras. In particular, $\mathcal{B} \to \mathcal{C}$ dominates all $B_\gamma\to C_\gamma$ for $\gamma\in \Gamma$ via ${S}^{{A}}_{\perfd}$-linear maps. Finally, note that $\mathcal{B}$ and $\mathcal{C}$ are certainly algebras over $R^+$ and $(R/I_D)^+$ respectively, but there are potentially multiple $R^+$- and $(R/I_D)^+$-algebra structures on $\mathcal{B}$ and $\mathcal{C}$, coming from different $B_\gamma$ and $C_\gamma$. Nonetheless, the inclusion $$\adj_{\mathcal{B} \shortrightarrow \mathcal{C}}(R, D+\Delta)\subseteq \adj_{B_\gamma\shortrightarrow C_\gamma}(R, D+\Delta)$$ holds regardless of the different $R^+$- and $(R/I_D)^+$-algebra structures. This is because $f^{1/n}\in S\subseteq  {S}^{{A}}_{\perfd}$ and $f^{1/n}$ is all we need to define $\adj_{\mathcal{B} \shortrightarrow \mathcal{C}}(R, D+\Delta)$.
\end{proof}

\begin{remark}
In connection with \autoref{thm:dominateBCM}, it is natural to ask that, suppose $(R,\m)$ is a complete normal local domain of mixed characteristic $(0,p)$ and $P$ is a height one prime of $R$. Then whether given any two perfectoid seed morphisms of $R^+$- and $(R/P)^+$-algebras $B_1\to C_1$ and $B_2\to C_2$, the map $$B_1\widehat{\otimes}_{R^+}B_2\to C_1\widehat{\otimes}_{(R/P)^+}C_2$$ is also a perfectoid seed morphism? We expect this should be true but we do not have a precise argument at this moment. The subtlety is that, while we know that $B_1\widehat{\otimes}_{S^A_{\perfd}}B_2\to C_1\widehat{\otimes}_{\overline{S}^{\overline{A}}_{\perfd}}C_2$ is a perfectoid seed morphism for certain domain extension $S$ of $R$ inside $R^+$ as in the proof of \autoref{thm:uniformBC}, it is not clear that these $S^A_{\perfd}$ form a direct system (the choice of $A$ depends on $S$) and thus \autoref{lem:perfectoidseedmorphism} does not immediately apply.
\end{remark}

Finally, we prove the result on the uniform choice of $B\to C$ in our definition of adjoint ideal. This can be viewed as an analog of \cite[Proposition 6.10]{MaSchwedeSingularitiesMixedCharBCM}.

\begin{theorem}
\label{thm:uniformBC}
With notations as in \autoref{def:adjoint}, there exists a single compatible choice of perfectoid big Cohen-Macaulay $R^+$- and $(R/I_D)^+$-algebras $\mathcal{B}\to \mathcal{C}$, such that $$\adj_{\mathcal{B}\shortrightarrow \mathcal{C}}(R, D+\Delta)\subseteq \adj_{B\shortrightarrow C}(R, D+\Delta)$$
for all other compatible choices of perfectoid big Cohen-Macaulay $R^+$- and $(R/I_D)^+$-algebras $B\to C$.
\end{theorem}
\begin{proof}
Let $I$ be the largest ideal contained in $\adj_{B\shortrightarrow C}(R, D+\Delta)$ for every compatible choice of $B\to C$. By the axiom of 
choice and the collection principle,
 for every element $\lambda\in R-I$, there exists a choice of $B_\lambda \to C_\lambda$ such that $\lambda\notin \adj_{B_\lambda\shortrightarrow C_\lambda}(R, D+\Delta)$. Now apply \autoref{thm:dominateMap} to the set $\{B_\lambda\to C_\lambda\}_{\lambda\in R-I}$, we know that there exists $\mathcal{B}\to \mathcal{C}$, such that $$\adj_{\mathcal{B}\shortrightarrow \mathcal{C}}(R, D+\Delta)\subseteq \adj_{B_\lambda\shortrightarrow C_\lambda}(R, D+\Delta)$$ for all $\lambda$. In particular, $\adj_{\mathcal{B}\shortrightarrow \mathcal{C}}(R, D+\Delta)\subseteq I$ and hence $$\adj_{\mathcal{B}\shortrightarrow \mathcal{C}}(R, D+\Delta)\subseteq \adj_{B\shortrightarrow C}(R, D+\Delta)$$ for all $B\to C$.
\end{proof}

\section{Extended Rees algebras of Koll\'ar component extractions}
\label{sec.ExtendedRees}

Suppose $\pi : Y \to X = \Spec R$ is a projective birational map with $(R, \fram, k)$ normal and local of dimension $\geq 2$ and $Y$ normal.  Further suppose that $\pi$ is obtained by blowing up an $\fram$-primary ideal $I$ and so $\pi$ is an isomorphism away from $V(\fram)$.  Suppose further that $I\cdot \cO_Y = \cO_Y(-mE)$ where $E$ is a prime Weil divisor.  Note that this last assumption can be substantially weakened throughout, but it simplifies many arguments.

We can form the following rings:
\[
\begin{array}{rl}
S & = \bigoplus_{n \geq 0} H^0(Y, \cO_Y(-nE))t^n\\
T & = \bigoplus_{n \in \bZ} H^0(Y, \cO_Y(-nE))t^n
\end{array}
\]
Note that $S$ is a $\bN$-graded ring and $T$ is a $\bZ$-graded ring and we have a canonical inclusion $S \subseteq T$. Also note that the $m$th Veronese subring of $S$ is just the normalization of Rees algebra $R[Is]$ and the $m$th Veronese subring of $T$ is the normalization of the extended Rees algebra $R[Is, s^{-1}]$.  Hence, replacing $m$ by a multiple, and $I$ by $\overline{I^l}$, we may assume that $m$th Veronese of $S$ is exactly the Rees algebra, and the $m$th Veronese of $T$, $T'$
is exactly the extended Rees algebra $R[Is, s^{-1}]$.  It is also clear that $T_{\leq 0} = R[t^{-1}]$ and $\Proj S=Y$.

We let
\[
\begin{array}{rl}
\frn_S & = \fram S + S_{> 0}\\
\frn_T & = \fram T + T_{> 0} + T_{<0}
\end{array}
\]
denote the homogeneous maximal ideals of $S$ and $T$ respectively.  Note in the case when $m = 1$, $\frn_S$ is simply $\fram + It$ and $\frn_T$ is $\m+It+t^{-1}$.

\begin{lemma}
If $ft^j \in S_{>0}$ is homogeneous of degree $j > 0$, then $S[(ft^j)^{-1}] \to T[(ft^j)^{-1}]$ is an isomorphism.
\end{lemma}
\begin{proof}
Note that $S[(ft^j)^{-1}]$ contains $t^{-j} = f\cdot (ft^j)^{-1}$.  Since $S$ is normal, so is $S[(ft^j)^{-1}]$.  Therefore since $t^{-1}$ is in the fraction field, we have that
\[
t^{-1} \in S[(ft^j)^{-1}].
\]
But this proves the lemma.
\end{proof}

\begin{lemma}
Suppose that $0\neq f \in \fram \subseteq [S]_0 \subseteq S$.  Then $S[f^{-1}] = R[f^{-1}][t]$ and in particular, $S[f^{-1}] \to T[f^{-1}]$ is \'etale.
\end{lemma}
\begin{proof}
Since we are blowing up an $\fram$-primary ideal, we have that $ft^{j} \in [S]_j$ for some $j > 0$ and so $t^j \in S[f^{-1}]$.  Thus since $S[f^{-1}]$ is normal, it contains $t$ as well.  Hence $S[f^{-1}] = R[f^{-1}][t]$.
\end{proof}

Putting these two elements together we have that $S \to T$ is \'etale outside of $\frn_S$.  Now, notice that $T/\frn_S T = k[t^{-1}]$ hence $\frn_S T$ is of codimension $\geq 2$ in $T$.  Therefore $S \to T$ is \'etale outside a set of codimension $2$ on both $S$ and $T$.  In conclusion:

\begin{lemma}
\label{lem.MapBetweenReesAndExtendedRees}
If $\rho : \Spec T \to \Spec S$ is the canonical map, then $\rho^* K_S = K_T$.
\end{lemma}

We next prove a consequence of graded local duality.

\begin{proposition}
\label{prop.LocalDualMD}
With notation as above, for any divisor $D$ on $Y$ let $M_D$ denote the $T$-module $\bigoplus_{n \in \mathbb{Z}} H^0(Y, \cO_Y(-nE + D))$.  We have:
\[
\bigoplus_{n\in\mathbb{Z}}\Hom_R(H^0(Y, \cO_Y(K_Y + nE - D)), E_R) \cong H_{\n_T}^{d+1}(M_D).
\]
Here $E_R$ is the injective hull of the residue field of $R$.
\end{proposition}
\begin{proof}
By looking at the spectral sequence of low degree terms and then using Grothendieck duality, we have
\[
\begin{array}{rl}
& \Gamma(Y, \cO_Y(K_Y + nE - D)) \\
\cong & \myH^{-d} \myR\Gamma\big(Y, \myR \sHom(\cO_Y(-nE + D), \omega_Y^{\mydot}) \big) \\
\cong & \myH^{-d} \myR \Hom_R(\myR\Gamma(Y, \cO_Y(-nE + D)), \omega_R^{\mydot}).
\end{array}
\]
By Grothendieck local duality we know that
\[
\begin{array}{rl}
& \Hom_R(H^0(Y, \cO_Y(K_Y + nE - D)), E_R) \\
= & \Hom_R(\myH^{-d} \myR \Hom_R(\myR\Gamma(Y, \cO_Y(-nE + D)), \omega_R^{\mydot}), E_R) \\
\cong & \myH^d\myR\Gamma_{\fram}\big(\myR\Gamma(Y, \cO_Y(-nE + D))\big) \\
\cong & \myH^d \myR\Gamma_E(Y, \cO_Y(-nE+D)).
\end{array}
\]
Thus it is enough to show that
\[
\bigoplus_{n\in\mathbb{Z}}H_E^d(Y, \cO_Y(-nE+D))\cong H_{\n_T}^{d+1}(M_D).
\]
Fix a generating set $\{g_1,\dots, g_s\}$ of $S_{>0} = T_{>0}$, so each $g_i=f_it^j$ for some $j$. Since $T_{f_it^j}=S_{f_it^j}$, and by our choice $\{\Spec[S_{f_it^j}]_0\}$ forms an affine cover of $Y$, we know from the \v{C}ech complex of $T$ on $g_1,\dots,g_s$ that we have a triangle (see \cite[page 150]{LipmanCohenMacaulaynessInGradedAlgebras})
\[
\bigoplus_{n\in\mathbb{Z}} \mathbf{R}\Gamma(Y, \cO_Y(-nE+D))[-1]\to \mathbf{R}\Gamma_{T_{>0}}(M_D) \to M_D \xrightarrow{+1}.
\]
Applying $\mathbf{R}\Gamma_{\m T+ (t^{-1})}(-)$ to the triangle, we obtain
\[
\mathbf{R}\Gamma_{\m T+ (t^{-1})}\left(\bigoplus_{n\in\mathbb{Z}}\mathbf{R}\Gamma(Y, \cO_Y(-nE+D))\right)[-1]\to \mathbf{R}\Gamma_{\m T+ (t^{-1})}\mathbf{R}\Gamma_{T_{>0}}(M_D) \to \mathbf{R}\Gamma_{\m T+ (t^{-1})}(M_D) \xrightarrow{+1}.
\]
Next we note that in each $T_{f_it^j}$, $t^{-j}=f_i/(f_it^j)$ and hence $t^{-1}\in \sqrt{\m T_{f_it^j}}$. Therefore
\[
\begin{array}{rl}
 & \mathbf{R}\Gamma_{\m T+ (t^{-1})}\left(\bigoplus_{n\in\mathbb{Z}}\mathbf{R}\Gamma(Y, \cO_Y(-nE + D))\right)\\
=& \mathbf{R}\Gamma_{\m T}\left(\bigoplus_{n\in\mathbb{Z}}\mathbf{R}\Gamma(Y, \cO_Y(-nE + D))\right)\\
=& \bigoplus_{n\in\mathbb{Z}}\mathbf{R}\Gamma_E(Y, \cO_Y(-nE + D)).
\end{array}
\]
On the other hand, we have $\m T\subseteq \sqrt{(t^{-1})}$ and thus $\mathbf{R}\Gamma_{\m T+ (t^{-1})}(T)=\mathbf{R}\Gamma_{(t^{-1})}(T)$. So the above triangle simplifies to
\[
\bigoplus_{n\in\mathbb{Z}}\mathbf{R}\Gamma_E(Y, \cO_Y(-nE+D))[-1]\to \mathbf{R}\Gamma_{\n_T}(M_D) \to \mathbf{R}\Gamma_{(t^{-1})}(M_D)\xrightarrow{+1}.
\]
Taking cohomology and noting that since $d\geq 2$, $h^d(\mathbf{R}\Gamma_{(t^{-1})}(M_D))=h^{d+1}(\mathbf{R}\Gamma_{(t^{-1})}(M_D))=0$, we obtain
\[
h^{d+1}(\bigoplus_{n\in\mathbb{Z}}\mathbf{R}\Gamma_E(Y, \cO_Y(-nE+D))[-1])\cong h^{d+1}(\mathbf{R}\Gamma_{\n_T}(M_D)),
\]
which is precisely saying that
\[
\bigoplus_{n\in\mathbb{Z}}H_E^d(Y, \cO_Y(-nE+D))\cong H_{\n_T}^{d+1}(M_D)
\]
as desired.
\end{proof}

\begin{corollary}
\label{lem.GradedReflexiveDivisorSheaf}
The module $M_D = \bigoplus_{n \in \bZ} H^0(Y, \cO_Y(D-nE))$ defined in \autoref{prop.LocalDualMD} is S2, or equivalently reflexive.
\end{corollary}
\begin{proof}
First note that
$\bigoplus_{n\in\mathbb{Z}} \Hom_R(H^0(Y, \cO_Y(K_Y + nE - D)), E_R)$
is the graded Matlis dual of $\bigoplus_{n\in\mathbb{Z}} H^0(Y, \cO_Y(- nE + K_Y - D)) = M_{K_Y - D}$.  Since $D$ is arbitrary, it suffices to show that $M_{K_Y - D}$ is S2, or in other words, that the graded Matlis dual of $H_{\n_T}^{d+1}(M_D)$ is S2.  But the graded Matlis dual of top local cohomology is always S2.  Let $M_{K_Y - D} \to M'$ denote the S2-ification with respect to $T$.  The cokernel has codimension $\geq 2$ by definition, so their top local cohomologies are isomorphic.  This completes the proof.
\end{proof}


%

\begin{proposition}
\label{prop.CanonicalModuleOfExtendedRees}
With notation as above,
\[
\omega_{T} \cong \bigoplus_{n \in \bZ} H^0(Y, \omega_Y(-nE)).
\]
\end{proposition}
We provide multiple proofs.
\begin{proof}[Proof \#1 of \autoref{prop.CanonicalModuleOfExtendedRees}]
We know by definition that
\[
\omega_T=\bigoplus_{n\in \mathbb{Z}} \Hom_R(H_{\n_T}^{d+1}(T)_n, E_R).
\]
Thus by Matlis duality, it is enough to show that $\bigoplus_{n\in\mathbb{Z}}\Hom_R(H^0(Y, \omega_Y(nE)), E_R) \cong H_{\n_T}^{d+1}(T)$ where $E_R$ denotes the injective hull of the residue field of $R$.  But this is just \autoref{prop.LocalDualMD} taking $D = 0$.
\end{proof}
\begin{proof}[Proof \#2 of \autoref{prop.CanonicalModuleOfExtendedRees}]
We know that the reflexification of $\omega_S \cdot T$ is $\omega_T$
by \autoref{lem.MapBetweenReesAndExtendedRees}.  In positive degrees, $\omega_S \cdot T$ already agrees with our desired module.  We see that $\bigoplus_{n \in \bZ} H^0(Y, \omega_Y(-nE))$ is a S2 and hence reflexive $T$-module by \autoref{lem.GradedReflexiveDivisorSheaf}.

Thus we need to show that
\[
\omega_S \cdot T \to \bigoplus_{n \in \bZ} H^0(Y, \omega_Y(-nE))
\]
is an isomorphism in codimension 1.  It is already an isomorphism in degree $> 0$ since $\omega_S = \bigoplus_{n > 0} H^0(Y, \omega_Y(-nE))$ by \cite[2.6.2]{HyrySmithOnANonVanishingConjecture}.  To this end, observe that for all $n \ll 0$, we have that $H^0(\omega_Y(-nE)) = \omega_R$ since $\pi : Y \to \Spec R$ is an isomorphism outside of $E$.  On the other hand, $\omega_S \cdot T$ is simply $H^0(Y, \omega_Y(-E))$ in all negative degrees since we can multiply by $t^{-1}$.  Since $\omega_R / H^0(Y, \omega_Y(-E))$ is supported in dimension 0, it is easy to see that $\omega_S \cdot T \to \bigoplus_{n \in \bZ} H^0(Y, \omega_Y(-nE))$ is an isomorphism outside of codimension 1.  This completes the proof.
\end{proof}

\subsection{Discrepancy computations and extended Rees algebras}
\label{subsec.DiscrepComputationsForExtendedRees}

Suppose $\pi : Y \to X = \Spec R$ is a projective birational map between normal integral schemes with $(R, \fram, k)$ local.  For any (prime) Weil divisor $D_Y$ on $Y$, we can form the associated coherent $\cO_Y$-module $\cO_Y(D_Y)$ and then construct the $T$-module
\[
\Gamma_{*, T}(\cO_Y(D_Y)) = \bigoplus_{n \in \bZ} H^0(Y, \cO_Y(D_Y - nE)).
\]

This is an S2 $T$-module by \autoref{lem.GradedReflexiveDivisorSheaf}.  It corresponds to a (prime) divisor on $\Spec T$, $D_T$.  In other words $T(D_T) = \Gamma_{*, T}(\cO_Y(D))$.  Note our previous work guarantees that $\Gamma_{*, T}(\cO_Y(K_Y)) = T(K_T)$.
Let $\nu : \Spec T \to \Spec R$ be the canonical map.  Suppose we have that $\Delta$ is a $\bQ$-divisor on $X$ such that $K_X + \Delta$ is $\bQ$-Cartier, say that $\Div_X(f) = n(K_X+\Delta)$.  We can write the following two formulae:
\[
\begin{array}{rl}
\pi^*(K_X + \Delta) & = K_Y + \Delta_Y\\
\nu^*(K_X + \Delta) & = K_T + \Delta_T
\end{array}
\]
for some $\bQ$-divisors $\Delta_Y$ and $\Delta_T$.  Since $\cO_Y(n\pi^*(K_X+\Delta)) = {1 \over f} \cO_Y$, $T(n\nu^*(K_X + \Delta)) = {1 \over f} T$ and $\Gamma_*({1 \over f} \cO_Y) = {1 \over f} T$,  we see that $\Delta_T$ and $\Delta_Y$ also correspond to each other.  Note this \emph{requires} our choice of $K_T$ from \autoref{prop.CanonicalModuleOfExtendedRees}.
Explicitly, if
\[
\Delta_Y = \sum a_i D_Y,
\]
then
\[
\Delta_T = \sum a_i D_T
\]
where $D_T$ is the prime divisor corresponding to $D_Y$ as above.

\subsection{Veronese covers}

Consider $T$ as above and consider the $m$th Veronese subalgebra $T'$.  Note $T' = R[It^m, t^{-m}]$ is isomorphic to the standard extended Rees algebra.
\begin{lemma}
\label{lem.HomTT'}
The $T$-module $\Hom_{T'}(T, T')$ is generated by the map $\Psi : T \to T'$ which projects onto the $m$-divisible summands.   Furthermore, we have that $\Psi(\frn_{T}) = \frn_{T'}$.
\end{lemma}
\begin{proof}
The projection onto the $m$-divisible summands certainly sends $\frn_{T}$ to $\frn_{T'}$ so we only need to show that $\Psi$ generates the $\Hom$-set. Note that since $E_T$ and $E_{T'}$ corresponds to $E$ on $Y$ (more precisely, $E_T=\Div_T(t^{-1})$ and $mE_{T'}=\Div_{T'}(t^{-m})$), $T'(iE_{T'})=\oplus_nH^0(Y, \cO_Y(-nmE+iE))$ and thus $T\cong \oplus_{i=0}^{m-1}T'(iE_{T'})$ by construction. In other words, $T$ is precisely the $m$th cyclic cover of $T'$ associated to the divisor $E_{T'}$. Since the index of $E_{T'}$ is $m$ (because this is the index of $E$ on $Y$, alternatively, one can also check that $T'(iE_{T'})$ is not isomorphic to $T'$ for every $1<i\leq m-1$), by \cite[proof of Proposition 4.21]{CarvajalRojasFiniteTorsors}, we know that $\Psi$ generates the $\Hom_{T'}(T, T')$ as a $T$-module.
\end{proof}

\begin{lemma}
\label{lem.PullbackOfDOnExtendedRees}
Suppose that $T' \to T$ is the inclusion of the $m$th Veronese subalgebra as above with $\kappa : \Spec T \to \Spec T'$ the induced map.  Let $D$ be a Weil divisor on $Y$ and let $D_T$ and $D_{T'}$ denote associated divisors on $T$ and $T'$ respectively.  Then $\kappa^* D_{T'} = D_T$.
\end{lemma}
\begin{proof}
It suffices to consider the case when $D$ is a prime divisor.  In the case that $D = E$, this follows immediately from the observation that $T$ is an index-1 cover of the $\bQ$-Cartier Weil divisor $D_{T'}$ as in \autoref{lem.HomTT'}.  Otherwise we may assume that $D$ is the strict transform of a divisor on $\Spec R$.  In particular, $t^{-m}$ is not in $Q' = \Gamma_{*, T'}(\cO_Y(-D)) \subseteq T'$.  But now notice that
\[
R[t^{-m}, t^{m}] \cong T'[t^{m}] \subseteq T[t^{m}] \cong R[t, t^{-1}]
\]
In particular, we then see that the extension of $Q' R[t^{-m}, t^{m}]$ is still prime in $R[t, t^{-1}]$.  This completes the proof.
\end{proof}

\bibliographystyle{skalpha}
\bibliography{MainBib}
\end{document}